\documentclass[12pt]{article}
\usepackage{amsmath,amssymb,latexsym}
\usepackage{color}
\usepackage{tikz-cd}
\usepackage{hyperref}
\usetikzlibrary{shapes.misc}
\title{Rational invariant tori and band edge spectra for non-selfadjoint operators}
\author{Michael Hitrik\\Department of Mathematics \\University of California \\ Los Angeles
\\CA 90095-1555, USA\\\small hitrik@math.ucla.edu
\and Johannes Sj\"ostrand\\IMB, Universit\'e de Bourgogne\\
9 avenue Alain Savary, BP 47870\\21078 Dijon cedex, France\\and UMR 5584 CNRS \\\small johannes.sjostrand@u-bourgogne.fr}

\date{}

\def\wrtext#1{\relax\ifmmode{\leavevmode\hbox{#1}}\else{#1}\fi}
\def\abs#1{\left|#1\right|}
\def\begeq{\begin{equation}}
\def\endeq{\end{equation}}

\def\iint{\int\hskip -2mm\int}

\def\neigh{neighborhood}

\def\Re{{\rm Re\,}}
\def\Im{{\rm Im\,}}

\newcommand{\eps}{\varepsilon}
\def\part#1{\frac{\partial}{\partial #1}}

\def\norm#1{||\,#1\,||}

\newcommand{\real}{\mbox{\bf R}}
\newcommand{\comp}{\mbox{\bf C}}
\newcommand{\z}{\mbox{\bf Z}}
\newcommand{\nat}{\mbox{\bf N}}

\renewcommand{\Re}{\mbox{\rm Re\,}}
\renewcommand{\Im}{\mbox{\rm Im\,}}

\renewcommand{\exp}{\mbox{\rm exp\,}}

\newtheorem{dref}{Definition}[section]

\newtheorem{theo}[dref]{Theorem}
\newtheorem{prop}[dref]{Proposition}

\newenvironment{proof}{\vspace{.3cm}\noindent{{\em Proof:}}}{\hfill$\Box$}

\begin{document}
\maketitle

\begin{abstract}
We study semiclassical asymptotics for spectra of non-selfadjoint perturbations of selfadjoint analytic $h$-pseudodifferential operators in dimension 2, assuming
that the classical flow of the unperturbed part is completely integrable. Complete asymptotic expansions are established for all individual eigenvalues in suitable
regions of the complex spectral plane, near the edges of the spectral band, coming from rational flow-invariant Lagrangian tori.
\end{abstract}

\vskip 2.5mm
\noindent {\bf Keywords and Phrases:} Non-selfadjoint, eigenvalue, spectral asymptotics, resolvent,
semiclassical limit, completely integrable, Lagrangian torus, rational torus, secular perturbation theory, pseudospectrum, exponential weight, FBI transform.

\vskip 2mm
\noindent
{\bf Mathematics Subject Classification 2000}: 35P15, 35P20, 37J35,37J40, 53D22, 58J37, 58J40

\tableofcontents

\section{Introduction}
\setcounter{equation}{0}
Spectra for semiclassical non-selfadjoint operators often display fascinating features, from lattices of low-lying eigenvalues for operators of
Kramers-Fokker-Planck type~\cite{HeSjSt},~\cite{HiPrSt} to eigenvalues for operators with analytic coefficients in dimension one, concentrated to unions of curves,
\cite{LN}, \cite{R}, \cite{GaSh}, \cite{HiSj3a}. The work~\cite{MeSj2} has established that for wide and stable classes of non-selfadjoint analytic
pseudodifferential operators in dimension two, the individual eigenvalues can be determined accurately in the semiclassical limit, by means of a complex
Bohr-Sommerfeld quantization condition, and form a distorted two-dimensional lattice. Now in many natural situations~\cite{Le},~\cite{Sj00},
\cite{Sj96},~\cite{SjZw1}, one encounters non-selfadjoint operators
of the form
\begeq
\label{intro1}
P_{\eps} = p(x,hD_x) + i\eps q(x,hD_x),\quad 0 \leq \eps \ll 1,
\endeq
considered on $\real^n$ or a compact real analytic manifold, with $P_{\eps = 0}$ being selfadjoint. Here $0 < h \ll 1$ is the semiclassical parameter and the
second small parameter $\eps$ represents the strength of the non-selfadjoint perturbation. The principal symbol of $P_{\eps}$ in (\ref{intro1}) is of the form
$p_{\eps}(x,\xi) = p(x,\xi) + i\eps q(x,\xi)$, where $p$ is real, and let us also assume, to fix the ideas, that $q$ is real. Both $p$ and $q$ are assumed to be
analytic, with $p$ elliptic near infinity. The spectrum of $P_{\eps}$ near the origin is confined to a band of width ${\cal O}(\eps)$, and the general problem
is to understand the distribution of eigenvalues of $P_{\eps}$ near $0$, in the semiclassical limit $h\rightarrow 0^+$. To this end, let us assume that $0$ is a
regular value of $p$, so that the energy surface $p^{-1}(0)$ is a smooth compact submanifold of the phase space. We then know~\cite{Ma},~\cite{MaMa} that the real
parts of the eigenvalues of $P_{\eps}$ near $0$ are distributed according to the same Weyl law as that for the unperturbed operator $P_{\eps = 0}$. In order to
study the distribution of the imaginary parts of the eigenvalues, following the method of averaging~\cite{We},~\cite{CdV}, we let $H_p$ be the Hamilton vector
field of $p$ and introduce the time averages
\begeq
\label{intro2}
\langle{q\rangle}_T = \frac{1}{T} \int_{-T/2}^{T/2} q\circ \exp(tH_p)\,dt,\quad T>0,
\endeq
of $q$ along the $H_p$--trajectories. It follows from~\cite{Le},~\cite{Sj00},~\cite{HiSjVu07} that if $z\in {\rm Spec}(P_{\eps})$ is such that
$\abs{{\rm Re}\, z}\leq \delta$, then
\begeq
\label{intro3}
\lim_{T \rightarrow \infty} \inf_{p^{-1}(0)} \langle{q\rangle}_T - o(1) \leq \frac{{\rm Im}\,z}{\eps}  \leq \lim_{T \rightarrow \infty}
\sup_{p^{-1}(0)} \langle{q\rangle}_T +o(1),
\endeq
as $(\eps,\delta,h)\rightarrow 0^+$.

\medskip
\noindent
The spectral analysis for non-selfadjoint operators of the form (\ref{intro1}) has been pursued by the authors in the series of
papers~\cite{HiSj1}--~\cite{HiSjVu07}, the latter work jointly with S. V\~u Ng\d{o}c, when the dimension $n = 2$ and the $H_p$--flow is either periodic
or completely integrable. Let us focus, from now on, on the completely integrable case, which will be considered also in the present work. In this case, the
energy surface $p^{-1}(0)$ is foliated by invariant Lagrangian tori, along with possibly some other more complicated flow-invariant sets. When
$\Lambda \subset p^{-1}(0)$ is an invariant torus such that the rotation number of $H_p$ along $\Lambda$ is Diophantine, i.e. poorly approximated by rational
numbers, or more generally, irrational, we have that the time averages $\langle{q\rangle}_T$ along $\Lambda$ converge to the space average
$\langle{q\rangle}(\Lambda)$ of $q$ over $\Lambda$, as $T\rightarrow \infty$. When $\Lambda$ is a torus with a rational rotation number, or a singular set in
the foliation of $p^{-1}(0)$, then in analogy with (\ref{intro3}), we introduce the compact interval
\begeq
\label{intro4}
Q_{\infty}(\Lambda) = \left[\lim_{T \rightarrow \infty} \inf_{\Lambda} \langle{q\rangle}_T, \lim_{T \rightarrow \infty} \sup_{\Lambda} \langle{q\rangle}_T\right]
\endeq
of limits of the time averages above.

\medskip
\noindent
Let $F_0 \in \real$ be such that $F_0 = \langle{q\rangle}(\Lambda_d)$ for a single Diophantine Lagrangian torus $\Lambda_d \subset p^{-1}(0)$, and let us assume
that
\begeq
\label{intro5}
F_0 \notin Q_{\infty}(\Lambda),
\endeq
for any other invariant set $\Lambda\neq \Lambda_d$ in $p^{-1}(0)$. It was then shown in~\cite{HiSjVu07} that the spectrum of $P_{\eps}$ can be determined completely, modulo ${\cal O}(h^{\infty})$, in a rectangle of the form $[-h^\delta /C,h^\delta /C]+i\eps [F_0-h^\delta /C,F_0+h^\delta /C]$, where $\delta > 0$ and
$\eps$ satisfies $h^K<\eps \ll 1$, for $K \gg 1$. Similarly to~\cite{MeSj2}, the spectrum has a structure of a distorted two-dimensional lattice, with
the horizontal spacing $\sim h$ and the vertical one $\sim \eps h$. A closely related result was obtained in~\cite{HiSj12}, giving a Weyl type asymptotic formula
for the number of eigenvalues of $P_{\eps}$ in an intermediate spectral band, bounded from above and from below by Diophantine levels, such as $F_0$ above.
It turned out that the distribution of the imaginary parts of the eigenvalues of $P_{\eps}$ is governed by a Weyl law, expressed in terms of phase space
volumes associated to $p$ and the long time averages of $q$.

\medskip
\noindent
Having elucidated the role played by flow-invariant Diophantine tori in the spectral analysis of $P_{\eps}$, let us now turn the attention to spectral
contributions of tori that are rational, which constitutes the subject of the present work. Let $F_0 \in \real$ be such that
$F_0 = \langle{q\rangle}(\Lambda_d)$, for a Diophantine torus $\Lambda_d$ as above, and rather than demanding (\ref{intro5}), let us assume that there exists
a rational torus $\Lambda_r\subset p^{-1}(0)$ such that $F_0\in Q_{\infty}(\Lambda_r)$, $F_0 \neq \langle{q\rangle}(\Lambda_r)$, while
$F_0 \notin Q_{\infty}(\Lambda)$, for $\Lambda \neq \Lambda_d,\,\Lambda_r$. An attempt to determine the individual eigenvalues of $P_{\eps}$ near $i\eps F_0$
was made by the authors in the work~\cite{HiSj08b}, by means of the normal form techniques. As a result, the normal forms near $\Lambda_r$ that we obtained
were given by a family of one-dimensional "resonant" non-selfadjoint operators, and the possibility of quite serious
pseudospectral phenomena for this family~\cite{DeSjZw} prevented us from computing the eigenvalues individually. Correspondingly, the main result
of~\cite{HiSj08b} was weaker, establishing that the spectrum of $P_{\eps}$ near $i\eps F_0$ was of the form $E_d \cup E_r$, where the "Diophantine" contribution
$E_d$ is a distorted lattice that can be described explicitly, as in~\cite{HiSjVu07}, and the cardinality of the "rational" contribution $E_r$ is $\ll$ than
that of $E_d$.

\medskip
\noindent
Subsequently, in the course of some numerical experiments, the authors have encountered peculiar pictures of the spectra of $P_{\eps}$, where the eigenvalues
had the form of a "centipede", with the body agreeing with the range of torus averages of $q$ --- see Section \ref{num} for the illustrations and the details of the numerical
computations. The legs of the centipede were more mysterious at first, but things became clearer when we realized that they represented the influence of
suitable rational tori. It became then natural to hope that the eigenvalues near the extremities of the legs could be determined asymptotically in a rigorous way,
since the pseudospectral effects should become more moderate near the edges of the spectral band,~\cite{DeSjZw}. The main result of the present work, giving a
complete asymptotic description of the individual extremal eigenvalues of $P_{\eps}$, can be considered as a justification of this hope.


\medskip
\noindent
Let us conclude the introduction by formulating, in a rough way, the main result of the paper --- see Theorem 2.1 below for the precise statement.
Let $\Lambda_0 \subset p^{-1}(0)$ be a rational Lagrangian torus such that
\begeq
\label{intro6}
\inf Q_{\infty}(\Lambda_0) < \inf_{\Lambda \neq \Lambda_0} \left(\inf Q_{\infty}(\Lambda)\right).
\endeq
The restriction of the $H_p$--flow to $\Lambda_0$ is periodic with primitive period $T_0 > 0$, and the time average $\langle{q\rangle}_{T_0}$ in (\ref{intro2})
can naturally be viewed as a function on the space of closed orbits $\Lambda_0/{\rm exp}(\real H_p)$. Let us assume that $\langle{q\rangle}_{T_0}$, viewed as a
function on $\Lambda_0/{\rm exp}(\real H_p)$, has a unique minimum which is nondegenerate, and restrict $\eps$ to a suitable interval of the form
$h^{1 + \eta} \leq \eps \leq h^{1-\eta}$, $\eta > 0$. Then for any fixed $C_0 > 0$ the eigenvalues of $P_{\eps}$ in the region
$$
\left\{z\in {\bf C};\ |\Re z|<\frac{h}{C_0\sqrt{\eps}},\ \frac{\Im z}{\eps} \leq {\rm inf}\, Q_{\infty}(\Lambda_0) + C_0\frac{h}{\sqrt{\eps}} \right \}
$$
can be determined completely, modulo ${\cal O}(h^{\infty})$, and are given by
\begeq
\label{intro7}
\lambda_{j,k} = a(\xi_2) + i\eps b(\xi_2) + \eps^{1/2} h\, \Lambda_{j,k}, \quad \xi_2 = h(j - \frac{k_0}{4}) - \frac{S}{2\pi},\,\, \quad j\in {\bf Z},\,\,
k\in \nat.
\endeq
Here $a(0) = 0$, $a'(0)>0$, $b(0) = \inf Q_{\infty}(\Lambda_0)$, and $S$ and $k_0$ are the classical action and the Maslov index of a primitive closed
$H_p$--trajectory in $\Lambda_0$. We have a complete asymptotic expansion for $\Lambda_{j,k}$ in integer powers of $\widetilde{h} = h/\sqrt{\eps}$,
$$
\Lambda_{j,k} \sim \sum_{\nu=0}^{\infty} \widetilde{h}^{\nu} \lambda_k^{\nu}(\xi_2,\sqrt{\eps}),
$$
where
$$
\lambda_k^0(0,0) = d e^{i\pi/4} (2k +1),\quad d > 0.
$$

\medskip
\noindent
{\bf Acknowledgement}. We are very grateful to Michael Hall and San V\~u Ng\d{o}c for interesting discussions concerning numerical computations of eigenvalues
and for showing us promising numerical results for a damped spherical oscillator. The second author is supported by the project NOSEVOL ANR 2011 BS 01019 01.

\section{Statement of the main results}
\label{st}
\setcounter{equation}{0}

\subsection{General assumptions}
We shall start by describing the general assumptions on our operators, which will be the same as in \cite{HiSj08b}, \cite{HiSjVu07}, as well
as in the earlier papers mentioned above. Let $M$ denote either the space ${\bf R}^2$ or a real analytic compact manifold of dimension 2.
When $M={\bf R}^2$, let
\begin{equation}
\label{st.1}
P_{\eps}=P^w(x,hD_x,\eps;h),\quad 0<h \leq 1,
\end{equation}
be the $h$--Weyl quantization on ${\bf R}^2$ of a symbol $P(x,\xi,\eps;h)$ (i.e. the Weyl quantization of $P(x,h \xi,\eps;h)$), depending smoothly on
$\eps\in {\rm neigh}(0,{\bf R})$ and taking values in the space of holomorphic functions of $(x,\xi)$ in a tubular neighborhood of ${\bf R}^4$ in ${\bf C}^4$, with
\begin{equation}
\label{st.2}
\abs{P(x,\xi,\eps;h)}\leq {\cal O}(1) m(\Re (x,\xi)),
\end{equation}
there. Here $1\leq m\in C^{\infty}(\real^4)$ is an order function, in the sense that
\begin{equation}
\label{st.3}
m(X)\leq C_0 \langle{X-Y}\rangle^{N_0} m(Y),\quad X,\,Y\in {\bf R}^4,
\end{equation}
for some $C_0$, $N_0>0$. We shall assume, as we may, that $m$ belongs to its own symbol class, so that $\partial^{\alpha}m={\cal O}_{\alpha}(m)$ for
each $\alpha\in {\bf N}^4$. Then for $h>0$ small enough and when equipped with the domain $H(m):=\left(m^w(x,hD)\right)^{-1}\left(L^2({\bf R}^2)\right)$,
$P_{\eps}$ becomes a closed densely defined operator on $L^2({\bf R}^2)$.

\medskip
\noindent
{\it Remark}. The analyticity assumptions will allow us to treat the case when $\eps \asymp h^\delta $, for $0 < \delta < 1$. When
$\eps = {\cal O}(h)$, standard $C^{\infty}$--microlocal analysis would have been sufficient.

\medskip
\noindent
Assume furthermore that
\begin{equation}
\label{st.4}
P(x,\xi,\eps;h)\sim \sum_{j=0}^{\infty} h^j p_{j,\eps}(x,\xi)
\end{equation}
in the space of holomorphic functions depending smoothly on $\eps \in {\rm neigh}(0,\real)$ and satisfying (\ref{st.2}) in a
fixed tubular neighborhood of ${\bf R}^4$. We assume that $p_{0,\eps}$ is elliptic near infinity,
\begin{equation}
\label{st.5}
\abs{p_{0,\eps}(x,\xi)}\geq \frac{1}{C} m(\Re (x,\xi)),\quad \abs{(x,\xi)}\geq C,
\end{equation}
for some $C>0$.

\bigskip
\noindent
When $M$ is a compact manifold, for simplicity we shall take $P_{\eps}$ to be a differential operator on $M$, such that for every choice of local coordinates,
centered at some point of $M$, it takes the form
\begin{equation}
\label{st.6}
P_{\eps}=\sum_{\abs{\alpha}\leq m} a_{\alpha,\eps}(x;h)(hD_x)^{\alpha},
\end{equation}
where $a_{\alpha,\eps}(x;h)$ is a smooth function of $\eps\in {\rm neigh}(0,{\bf R})$ with values in the space of bounded holomorphic functions in a complex
neighborhood of $x=0$, independent of $h$ when $\abs{\alpha} = m$. We further assume that
\begin{equation}
\label{st.7}
a_{\alpha,\eps}(x;h)\sim \sum_{j=0}^{\infty} a_{\alpha,\eps,j}(x) h^j,\quad h\rightarrow 0,
\end{equation}
in the space of such functions. The semiclassical principal symbol $p_{0,\eps}$, defined on $T^*M$, takes the form
\begin{equation}
\label{st.8}
p_{0,\eps}(x,\xi)=\sum a_{\alpha,\eps,0}(x)\xi^{\alpha},
\end{equation}
if $(x,\xi)$ are the canonical coordinates on $T^*M$. We make the ellipticity assumption,
\begin{equation}
\label{st.9}
\abs{p_{0,\eps}(x,\xi)}\geq \frac{1}{C} \langle{\xi}\rangle^m,\quad (x,\xi)\in T^*M,\quad \abs{\xi}\geq C,
\end{equation}
for some large $C>0$. Here we assume that $M$ has been equipped with some real analytic Riemannian metric so that $\abs{\xi}$ and
$\langle{\xi}\rangle=(1+\abs{\xi}^2)^{1/2}$ are well-defined.

\medskip
\noindent
Sometimes, we write $p_{\eps}$ for $p_{0,\eps}$ and simply $p$ for $p_{0,0}$. We make the assumption that
$$
P_{\eps=0}\quad \hbox{is formally selfadjoint}.
$$
In the case when $M$ is compact, we let the underlying Hilbert space be $L^2(M, \mu(dx))$ where $\mu(dx)$ is the Riemannian volume element.

\medskip
\noindent
The assumptions above imply that the spectrum of $P_{\eps}$ in a fixed neighborhood of $0\in {\bf C}$ is discrete, when $0<h\leq h_0$, $0\leq \eps\leq \eps_0$,
with $h_0>0$, $\eps_0>0$ sufficiently small. Moreover, if $z\in {\rm neigh}(0,{\bf C})$ is an eigenvalue of $P_{\eps}$ then $\Im z ={\cal O}(\epsilon)$.

\medskip
\noindent
We furthermore assume that the real energy surface $p^{-1}(0)\cap T^*M$ is connected and that
$$
dp\neq 0\quad \hbox{along}\quad p^{-1}(0)\cap T^*M.
$$

\medskip
\noindent
In what follows we shall write
\begin{equation}
\label{st.10}
p_{\eps}=p+i\eps q+{\cal O}(\eps^2),
\end{equation}
in a neighborhood of $p^{-1}(0)\cap T^*M$, and for simplicity we shall assume throughout the paper that $q$ is real valued on the real domain.
(In the general case, we should simply replace $q$ below by $\Re q$.) We let $H_p=p'_{\xi}\cdot \partial_x-p'_x\cdot \partial_{\xi}$ be the Hamilton vector
field of $p$.

\subsection{Assumptions related to the complete integrability}
As in \cite{HiSjVu07}, \cite{HiSj08b}, let us assume that there exists an analytic real valued function $f$ near $p^{-1}(0) \cap T^*M$ such that $H_pf=0$, with the
differentials $df$ and $dp$ being linearly independent on an open and dense set $\subset {\rm neigh}(p^{-1}(0)\cap T^*M, T^*M)$. For each
$E\in {\rm neigh}(0,{\bf R})$, the level sets $\Lambda_{a,E}=f^{-1}(a)\cap p^{-1}(E)\cap T^*M$ are invariant under the $H_p$--flow and form a singular foliation
of the 3-dimensional hypersurface $p^{-1}(E)\cap T^*M$. At each regular point (i.e. non-critical point for the restriction of $f$ to $p^{-1}(E)$), the leaves of
this foliation are 2-dimensional analytic Lagrangian submanifolds, and each regular leaf is a finite union of tori. In what follows we shall use the word ``leaf'' and notation
$\Lambda$ for a connected component of some $\Lambda_{a,E}$. Let $J$ be the set of all leaves in $p^{-1}(0)\cap T^*M$. Then we have a disjoint union decomposition
\begin{equation}
\label{st.11}
p^{-1}(0)\cap T^*M=\bigsqcup_{\Lambda \in J} \Lambda,
\end{equation}
where $\Lambda$ are compact connected $H_p$--invariant sets. The set $J$ has a natural structure of a graph whose edges correspond to families of regular leaves
and the set $S$ of vertices is composed of singular leaves. The union of edges $J\backslash S$ possesses a natural real analytic structure and the corresponding
tori depend analytically on $\Lambda\in J\backslash S$ with respect to that structure. See section 7 in \cite{HiSj08b} for an explicit description of the
Lagrangian foliation in the case when $M$ is an analytic surface of revolution in ${\bf R}^3$.

\medskip
\noindent
In what follows, we shall assume that the graph $J$ is finite. We shall identify each edge of $J$ analytically with a real bounded interval and this
determines a distance on $J$ in the natural way. Assume that the following continuity property holds,
\begin{eqnarray}
\label{st.11.1}
& & \wrtext{For every}\,\,\,\Lambda_0\in J\,\,\wrtext{and
every}\,\,\eps>0,\, \wrtext{there exists}\, \delta>0,\,\,\wrtext{such that if} \\ \nonumber
& & \Lambda\in J,\,\,{\rm
dist}_J(\Lambda,\Lambda_0)<\delta,\,\,\wrtext{then}\,\, \Lambda\subset
\{\rho\in p^{-1}(0)\cap T^*M;\, {\rm dist}(\rho,\Lambda_0)<\eps\}.
\end{eqnarray}

\medskip
\noindent
{\it Remark}. Let us assume that $f$ is a Morse-Bott function when restricted to $p^{-1}(0)\cap T^*M$, in the sense that the set of critical points of the
restriction of $f$ to $p^{-1}(0)\cap T^*M$ is a disjoint union of connected submanifolds, with the transversal Hessian of $f$ being nondegenerate along each of
the submanifolds. In this case, the structure of the singular leaves is known~\cite{San}. The set $J$ is then a finite connected graph and the property
(\ref{st.11.1}) holds.

\medskip
\noindent
Each torus $\Lambda\in J\backslash S$ carries real analytic coordinates $x_1$, $x_2$, identifying $\Lambda$ with ${\bf T}^2 = {\bf R}^2/2\pi {\bf Z}^2$, so that
along $\Lambda$, we have
\begin{equation}
\label{st.12}
H_p=a_1\partial_{x_1}+a_2\partial_{x_2},
\end{equation}
where $a_1$, $a_2\in {\bf R}$. The rotation number is defined as the ratio
$$
\omega(\Lambda)=[a_1:a_2]\in {\bf R} {\bf P}^1,
$$
and it depends analytically on $\Lambda\in J\backslash S$. Recall also that the leading perturbation $q$ has been introduced in (\ref{st.10}). For each torus
$\Lambda\in J\backslash S$, we define the torus average $\langle{q}\rangle(\Lambda)$ obtained by integrating $q|_{\Lambda}$ with respect to the natural smooth
measure on $\Lambda$.

\medskip
\noindent
We introduce the time averages,
\begin{equation}
\label{st.13}
\langle{q}\rangle_T=\frac{1}{T} \int_{-T/2}^{T/2} q\circ
\exp(tH_p)\,dt,\quad T>0,
\end{equation}
and consider the compact intervals $Q_{\infty}(\Lambda)\subset {\bf R}$, $\Lambda\in J$, defined as in~\cite{HiSjVu07},
\begin{equation}
\label{st.14}
Q_{\infty}(\Lambda)=\left[\lim_{T \rightarrow \infty} \inf_{\Lambda}
\langle{q}\rangle_T, \lim_{T \rightarrow \infty}
\sup_{\Lambda}\langle{q}\rangle_T\right].
\end{equation}
Notice that when $\Lambda\in J\backslash S$ and $\omega(\Lambda)\notin {\bf Q}$ then
$Q_{\infty}(\Lambda)=\{\langle{q}\rangle(\Lambda)\}$. In the rational case,
we write $\omega(\Lambda)=\frac{m}{n}$, where $m\in {\bf Z}$ and $n\in {\bf N}$ are
relatively prime, and where we may assume that $m={\cal
O}(n)$. When $k(\omega(\Lambda)):=\abs{m}+\abs{n}$ is the height of $\omega(\Lambda)$,
we recall from Proposition 7.1 in~\cite{HiSjVu07} that
\begin{equation}
\label{st.15}
Q_{\infty}(\Lambda)\subset \langle{q}\rangle(\Lambda)+{\cal
O}\left(\frac{1}{k(\omega(\Lambda))^{\infty}}\right)[-1,1].
\end{equation}

\medskip
\noindent
{\it Remark}. As $J\backslash S\ni \Lambda \rightarrow \Lambda_0\in S$, the set of all accumulation points of $\langle{q}\rangle(\Lambda)$ is contained in the
interval $Q_{\infty}(\Lambda_0)$. See the related remark in \cite{HiSj08b}, Section 2.

\medskip
\noindent
From Theorem 7.6 in~\cite{HiSjVu07} we recall that
\begin{equation}
\label{st.16}
\frac{1}{\eps}\Im \left({\rm Spec}(P_{\eps})\cap \{z; \abs{\Re z}\leq \delta\}\right)
\subset \left[\inf \bigcup_{\Lambda\in J} Q_{\infty}(\Lambda)-o(1),
\sup \bigcup_{\Lambda\in J} Q_{\infty}(\Lambda)+o(1)\right],
\end{equation}
as $(\eps, h, \delta) \rightarrow 0$.

\subsection{The main result}

Let $\Lambda_0\in J\backslash S$ be a rational invariant Lagrangian torus, so that as above, $\omega_0:=\omega(\Lambda_0)=\frac{m}{n}\in
{\bf Q}$, $m={\cal O}(n)$. Assume that the isoenergetic condition holds,
\begin{equation}
\label{st.16.5}
(d_\Lambda \omega )(\Lambda _0)\ne 0.
\end{equation}
We recall from Section 2 of \cite{HiSj08b} the behavior of the interval $Q_{\infty}(\Lambda)$ when $\Lambda\neq
\Lambda_0$ is a rational torus in a neighborhood of $\Lambda_0$. Writing $\omega(\Lambda)=\frac{p}{q}$ where $p\in {\bf Z}$ and $q\in
{\bf N}$ are relatively prime, $p={\cal O}(q)$, we get, using that
$\omega(\Lambda)\neq \omega_0$,
\begin{equation}
\label{st.17}
\abs{\omega(\Lambda)-\omega_0}\geq \frac{1}{n q}\geq \frac{1}{n k(\omega(\Lambda))},\
\end{equation}
and therefore, in view of (\ref{st.15}),
\begin{equation}
\label{st.18}
Q_{\infty}(\Lambda)\subset \langle{q}\rangle(\Lambda)+{\cal
O}({\rm dist}(\omega(\Lambda),\omega_0)^{\infty})[-1,1].
\end{equation}
This estimate is uniform in $\omega_0$ provided that we have a uniform upper bound on the height of the rotation number $\omega_0\in {\bf Q}$.

\medskip
\noindent
Let us assume that the chosen rational torus $\Lambda _0$ is such that
\begin{equation}\label{st.19}
\inf Q_\infty (\Lambda _0)<\inf _{\Lambda \in J\setminus \{\Lambda _0
  \}}\inf Q_\infty (\Lambda ).
\end{equation}
The result below remains valid with the obvious modifications, if we replace (\ref{st.19}) by the assumption
\begin{equation}
\label{st.20}
\sup Q_\infty (\Lambda _0)>\sup _{\Lambda \in J\setminus \{\Lambda _0\}}\sup Q_\infty (\Lambda ).
\end{equation}

\bigskip
\noindent
Let us choose, as we may, action-angle coordinates $(x,\xi )$ near $\Lambda_0$, so that $\Lambda _0$ is given by $\{\xi =0 \}$ in
${\bf T}_x^2\times {\bf R}_{\xi}^2$, $p=p(\xi)$, and so that
\begin{equation}
\label{st.21}
p(0)=0,\, \partial_{\xi_1}p(0)=0,\, \partial _{\xi _2}p(0) > 0,\, \partial _{\xi _1}^2p(0)\ne 0.
\end{equation}
Here the last property follows from (\ref{st.16.5}), and in order to fix the ideas, we shall assume that $\partial _{\xi _1}^2p(0)>0$. By the implicit function
theorem we have
\begeq
\label{st.21.5}
\partial _{\xi _1}p(\xi )=0 \Leftrightarrow \xi _1=f (\xi_2),
\endeq
where $f$ is an analytic function with $f(0)=0$, and we obtain an analytic family of rational Lagrangian tori $\Lambda_E \subset p^{-1}(E)$,
$E\in {\rm neigh}(0,\real)$, given by
\begeq
\label{st.21.6}
\xi_2 = \xi_2(E),\quad \xi_1 = f(\xi_2(E)).
\endeq
Here $\xi_2 = \xi_2(E)$ is the unique smooth solution of the equation $p(f(\xi_2),\xi_2) = E$, close to $0$, such that $\xi_2(0) = 0$.

\medskip
\noindent
Writing $q = q(x,\xi)$ in terms of the action-angle coordinates $(x,\xi)$, let
\begin{equation}
\label{st.22}
\langle q\rangle_2(x_1,\xi )=\frac{1}{2\pi }\int_0^{2\pi }q(x,\xi )dx_2, \quad \xi \in {\rm neigh}(0,\real^2),
\end{equation}
be the average of $q$ with respect to $x_2$.
We assume that
\begin{equation}
\label{st.23}\begin{split}
{\bf T}\ni x_1\mapsto \langle q\rangle_2(x_1,0)\hbox{ has a unique
  minimum}\\ \hbox{ which is nondegenerate.}
\end{split}
\end{equation}
In order to give an invariant description of the assumption (\ref{st.23}), let us notice that when restricted to $\Lambda_0$, the Hamilton flow of $p$ is
periodic of primitive period $T_0 > 0$ and the average $\langle{q\rangle}_2(x_1,0)$ can naturally be viewed as the flow average $\langle{q\rangle}_{T_0}$, defined
as in (\ref{st.13}), considered as a function on the space of closed $H_p$--orbits in $\Lambda_0$,
$$
\Lambda_0 /\exp(\real H_p) \simeq {\bf T}.
$$
In its invariant formulation, the assumption (\ref{st.23}) therefore states that flow average $\langle{q\rangle}_{T_0}$, viewed as a function on $\Lambda_0 /\exp(\real H_p)$, has a unique
minimum which is nondegenerate.

\medskip
\noindent
It follows from (\ref{st.23}) that the function ${\bf T}\ni x_1 \mapsto \langle{q\rangle}_2(x_1,\xi)$ has a unique minimum $x_1 = x_1(\xi)$ which is nondegenerate,
for $\xi \in {\rm neigh}(0,\real^2)$. The range of $\langle q\rangle_2(\cdot ,0)$ is equal to $Q_\infty (\Lambda _0)$, so the minimal value,
$\langle q\rangle_2(x_1(0),0) = \inf Q_\infty (\Lambda _0)$ is situated strictly below $\inf _{\Lambda \in J\setminus \{\Lambda _0 \}}\inf Q_\infty (\Lambda )$.

\medskip
\noindent
In this paper, we shall work under the assumption that the subprincipal symbol of the unperturbed operator $P_{\eps = 0}$ vanishes,
\begeq
\label{st.24}
p_{1,0}(x,\xi) = 0.
\endeq

\bigskip
\noindent
The following is the main result of this work.
\begin{theo}
\label{st1}
We adopt the assumptions above, in particular, {\rm (\ref{st.16.5})}, {\rm (\ref{st.19})}, {\rm (\ref{st.23})}, and {\rm (\ref{st.24})}. Let us put
$x_1(\xi _2)=x_1(f(\xi _2),\xi _2)$. Let $\delta \in (1/18,1/9)$ be fixed and assume that
\begeq
\label{st.25}
h^{1/(1-\delta)} \ll \eps \ll h^{6/(5+12\delta)}.
\endeq
Set
$$
\widetilde{h} = \frac{h}{\sqrt{\eps}}.
$$
Then for every $C_0>0$, we have the following description of the eigenvalues of $P_\eps$ in the region
\begeq
\label{st.25.1}
\left\{z\in {\bf C};\ |\Re z|<\frac{h}{C_0\sqrt{\eps}},\ \frac{\Im z}{\eps} \leq {\rm inf}\, Q_{\infty}(\Lambda_0) + C_0\frac{h}{\sqrt{\eps}} \right \},:
\endeq
valid for all $h>0$ small enough: the eigenvalues are simple and given by
\begin{equation}\label{st.26}
\begin{split}
\lambda _{j,k}=&p(f(h(j-\theta_2)), h(j-\theta_2)) + i\eps \langle q\rangle_2(x_1(h(j-\theta_2)), f(h(j-\theta_2)), h(j-\theta_2))\\
&+\sqrt{\eps}h(\lambda _{j,k}^0+\lambda _{j,k}^1\widetilde{h} + \lambda _{j,k}^2\widetilde{h}^2+\ldots),
\end{split}
\end{equation}
with $j\in {\bf Z},\, h(j-\theta_2) = {\cal O}(h/\sqrt{\eps}),\,\, {\bf N}\ni k\le {\cal O}(1)$, where
$\lambda_{j,k}^{\nu} = \lambda_{k}^\nu (h(j-\theta_2),\sqrt{\eps})$ is a smooth function of $\xi_2= h(j-\theta_2)\in {\rm neigh}(0,\real)$ and
$\sqrt{\eps}\in {\rm neigh}(0,\overline{\real_+})$, and
\begin{equation}
\label{st.27}
\lambda _k^0(\xi_2,0)=e^{i\pi /4}(\partial _{\xi _1}^2p(f(\xi _2),\xi
_2))^{\frac{1}{2}}(\partial _{x_1}^2\langle q\rangle_2(x_1(\xi_2), f(\xi_2),\xi_2))^{\frac{1}{2}}\left(k+\frac{1}{2}\right).
\end{equation}
Here we have written $\theta_2 = k_0(\alpha_2)/4 + S_2/2\pi h$, where $k_0(\alpha_2)$ and $S_2$ are the Maslov index and the classical action, respectively,
of the fundamental cycle in $\Lambda_0$, given by a closed $H_p$--trajectory of minimal period.
\end{theo}

\medskip
\noindent
{\it Remark}. Choosing $\delta \in (1/18,1/9)$ in Theorem 2.1 to be close to $1/9$, we see from (\ref{st.25}) that the description of the eigenvalues in
Theorem \ref{st1} in the region (\ref{st.25.1}) is valid in the range
$$
h^{\frac{9}{8}-\eta} \ll \eps \ll h^{\frac{18}{19} + \eta},
$$
when $\eta > 0$ is small. In particular, we are able to reach some cases when $\eps \gg h$, and here
the analyticity assumptions seem essential.

\medskip
\noindent
{\it Remark}. The result of Theorem 2.1 admits a natural extension to the case when ${\rm Re}\, z \in {\rm neigh}(0,\real)$ varies in a sufficiently small but
fixed neighborhood of $0\in \real$. Indeed, let us recall the family of rational Lagrangian tori $\Lambda_E \subset p^{-1}(E)$, $E\in {\rm neigh}(0,\real)$,
introduced in (\ref{st.21.6}). A natural analog of the assumption (\ref{st.19}) is then valid for $\inf Q_{\infty}(\Lambda_E)$, relative to the Lagrangian foliation
in $p^{-1}(E)$, provided that $\abs{E}$ is small enough. It follows therefore from Theorem \ref{st1} that the description (\ref {st.26}) of the spectrum of $P_{\eps}$ remains valid
when
$$
\abs{{\rm Re}\, z - E}\leq \frac{h}{C_0 \sqrt{\eps}},\quad \frac{\Im z}{\eps} \leq {\rm inf}\, Q_{\infty}(\Lambda_E) + C_0\frac{h}{\sqrt{\eps}},
$$
uniformly in $E\in {\rm neigh}(0,\real)$. We conclude therefore that the result of Theorem \ref{st1} extends to the spectral region
$$
\left\{z\in {\bf C};\ |\Re z|<\frac{1}{C},\ \frac{\Im z}{\eps} \leq {\rm inf}\, Q_{\infty}(\Lambda_{{\rm Re}\, z}) +
{\cal O}\left(\frac{h}{\sqrt{\eps}}\right) \right\},
$$
for $C > 1$ large enough.

\medskip
\noindent
The plan of the paper is as follows. Section 3 is devoted to a general outline of the proof. In Section 4 we construct a global compactly supported
weight function $G$, such that the leading symbol of $P_{\eps}$, acting on the weighted space associated to $G$, becomes $\approx p + i\eps\left(q - H_p G\right)$,
with the imaginary part avoiding the value $\eps \inf Q_{\infty}(\Lambda_0)$ on $p^{-1}(0)$, away from the rational torus $\Lambda_0$. This effectively
microlocalizes the spectral problem for $P_{\eps}$ to a small neighborhood of $\Lambda_0$. The quantum normal form construction for $P_{\eps}$ in the rational
region is carried out in Section 5, using the techniques of secular perturbation theory, thereby reducing the analysis to the study of a one-parameter family of
non-selfadjoint operators in dimension one, having double characteristics with elliptic quadratic approximations. In Section 6 we recall the computation of
low-lying eigenvalues for such operators, following~\cite{HeSjSt} and~\cite{HiPrSt}, and extend the results there to the parameter-dependent case.
The final step in the proof of Theorem 2.1 is taken in Section 7, where we carry out a pseudospectral analysis for the family
of the one-dimensional operators in question, controlling the resolvent norms and obtaining the spectral localization. It then becomes possible to complete
the proof by solving a suitable globally well-posed Grushin problem for $P_{\eps}$ in a weighted space, using the ideas and techniques
of~\cite{HiSj1},~\cite{HiSjVu07}. In Section 8 we present the results of numerical computations illustrating Theorem 2.1. The Appendix establishes some
subelliptic resolvent bounds for non-selfadjoint operators of Schr\"odinger type, playing a principal role in the pseudospectral analysis of Section 7 in
the main text. These bounds seem to be of some independent interest, and their proofs are very much based on the techniques developed
in~\cite{HeSjSt},~\cite{HiPrSt}.

\section{Outline of the proof}
\setcounter{equation}{0}
In this section we shall give a general outline of the proof of Theorem 2.1. Some of the techniques come from the previous works~\cite{HiSjVu07},~\cite{HiSj08b},
and the presentation below will naturally focus on the new difficulties of pseudospectral nature, encountered in the analysis in the rational
region. We shall then also describe heuristically some of the essential ideas employed in overcoming those difficulties, referring to Section 7 and to the
Appendix for a detailed rigorous discussion.

\medskip
\noindent
The principal symbol of the operator $P_{\eps}$ in (\ref{st.1}), (\ref{st.6}) is of the form
\begeq
\label{outl1}
p_{\eps} = p + i\eps q + {\cal O}(\eps^2),
\endeq
in a neighborhood of $p^{-1}(0)\cap T^*M$, and thanks to the ellipticity assumptions (\ref{st.5}), (\ref{st.9}), we observe that it suffices to make a microlocal
study in the region where $p$ is small. Recalling the assumption (\ref{st.19}) and replacing $q$ by $q-\inf Q_{\infty}(\Lambda_0)$, in the following discussion 
we shall assume, for notational simplicity only, that $\inf Q_{\infty}(\Lambda_0) = 0$. The first step in the argument is a construction of a global weight 
function $G\in C^{\infty}_0(T^*M)$ such that away from a small neighborhood of the rational Lagrangian torus $\Lambda_0$ in $p^{-1}(0)\cap T^*M$, we have
\begeq
\label{outl2}
q - H_p G \geq \frac{1}{{\cal O}(1)}.
\endeq
Away from $\Lambda_0$, the weight $G$ satisfies
$$
H_p G = q -\langle{q\rangle}_T,
$$
where $\langle{q\rangle}_T$ has been introduced in (\ref{st.13}), and when constructing $G$ in a neighborhood of $\Lambda_0$, we introduce action-angle coordinates
$(x,\xi) \in T^*{\bf T}^2$, so that $\Lambda_0 = \{\xi = 0\} \subset T^*{\bf T}^2$, and
\begeq
\label{outl3}
p_{\eps}(x,\xi) = p(\xi) + i\eps q(x,\xi) + {\cal O}(\eps^2),
\endeq
where the frequencies $\partial_{\xi_1} p(0)$ and $\partial_{\xi_2} p(0)$ are commensurable. After a linear change of variables, we get
$\partial_{\xi_1} p(0) = 0$, and the isoenergetic condition (\ref{st.16.5}) shows that $\partial^2_{\xi_1}p(0) \neq 0$. In the following discussion, in order
to fix the ideas, we shall consider the model case $p(\xi) = \xi_2 + \xi_1^2$, which suffices to illustrate the difficulties.
The weight function $G$ near $\xi = 0$ satisfies the cohomological equation
\begeq
\label{outl4}
H_p G = q -\widetilde{q},
\endeq
modulo ${\cal O}(\xi^{\infty})$, where $\widetilde{q} = \widetilde{q}(x_1,\xi)$ is independent of $x_2$ and is such that
\begeq
\label{outl5}
\widetilde{q}(x_1,0) = \frac{1}{2\pi}\int_0^{2\pi} q(x,0)\,dx_2
\endeq
is the average of $q(x,0)$ in the $x_2$-direction. From (\ref{st.23}) we then know that $\widetilde{q}(x_1,0)\geq 0$ and that 
$x_1 \mapsto \widetilde{q}(x_1,0)$ has a unique minimum which is nondegenerate. The partial Birkhoff normal form construction, utilized in solving (\ref{outl4}) 
may be continued, first at the principal symbol level, and then
on the level of operators, leading to the conclusion that microlocally in the rational region, when acting on an exponentially weighted space, the operator
$P_{\eps}$ is unitarily equivalent to an operator of the form
\begeq
\label{outl6}
P(x_1,hD_x,\eps;h) + R(x,hD_x,\eps;h): L^2({\bf T}^2) \rightarrow L^2({\bf T}^2).
\endeq
We refer to Proposition 7.1 in Section 7 for the precise statement. Here the full symbol of $P(x_1,hD_{x_1},\eps;h)$ is independent of $x_2$ and is given by
\begeq
\label{outl7}
P(x_1,\xi,\eps;h) = p(\xi) + i\eps \widetilde{q}(x_1,\xi) + {\cal O}(\eps^2 + h^2).
\endeq
The contribution $R(x,\xi,\eps;h) = {\cal O}((\eps,\xi,h)^{\infty})$ in (\ref{outl6}) is a remainder, which becomes ${\cal O}(h^{\infty})$ when restricting
the attention to the region $\xi = {\cal O}(\eps^{\delta})$, for a suitable small fixed $\delta > 0$ --- as we shall see, understanding this region suffices
for the description of the eigenvalues in Theorem 2.1. In particular, since $\xi$ becomes small, in the following heuristic discussion, we shall make a 
simplification and assume that $\widetilde{q}$ in (\ref{outl7}) is independent of $\xi$ altogether, depending on $x_1$ only. Let us also suppress the 
error term ${\cal O}(\eps^2 + h^2)$ in (\ref{outl7}), for simplicity. When considering it in Section 7, it will be treated entirely as a perturbation. 

\medskip
\noindent
Taking a Fourier series decomposition in $x_2$, we may view the operator $P$ in (\ref{outl6}) as a one-parameter family of operators
$P(x_1,hD_{x_1},\xi_2,\eps;h) = P(\xi_2)$, acting on $L^2({\bf T})$, such that
\begeq
\label{outl8}
P(\xi_2) = \xi_2 + L_{\eps}, \quad \xi_2 = hj,\quad j\in \z,
\endeq
where 
\begeq
\label{outl9}
L_{\eps} = (hD_{x_1})^2 + i\eps \widetilde{q}(x_1),\quad \widetilde{q}\geq 0, 
\endeq
is a one-dimensional non-selfadjoint Schr\"odinger operator with $\eps \widetilde{q}$ as a potential. We are interested in the spectrum of the family (\ref{outl8}) 
in the region where ${\rm Re}\, z$ is small and $\abs{{\rm Im}\, z} \leq {\cal O}(h\sqrt{\eps})$, and the first observation 
is that the eigenvalues of the operator 
$$
\frac{1}{\eps} L_{\eps} = \left(\widetilde{h}D_{x_1}\right)^2 + i\widetilde{q}(x_1), \quad 
\widetilde{h} = \frac{h}{\sqrt{\eps}},
$$
can be determined asymptotically in any disc $\abs{w} < C\widetilde{h}$, by means of the harmonic approximation, 
provided that $\widetilde{h} \ll 1$. See~\cite{HeSjSt},~\cite{HiPrSt}, and the discussion in Section 6 below. The eigenvalues of $\eps^{-1}L_{\eps}$ in this region 
are of the form 
\begeq
\label{outl10}
\mu_k(\widetilde{h}) \sim \sum_{j=0}^{\infty} \mu_{k,j} \widetilde{h}^{j+1},\quad k\in \nat,
\endeq
where 
\begeq
\label{outl10.1}
\mu_{k,0} = \left(2\partial_{x_1}^2 \widetilde{q}(x_1^{\rm min})\right)^{1/2} e^{i\pi/4} \left(k + \frac{1}{2}\right),
\endeq
are the eigenvalues of the globally elliptic quadratic operator 
$$
D_y^2 + \frac{1}{2} \left(\partial_{x_1}^2 \widetilde{q}(x_1^{\rm min})\right)y^2 
$$
acting on $L^2(\real)$. Here $x_1^{\rm min}\in {\bf T}$ is the unique point such that $\widetilde{q}(x_1^{\rm min}) = 0$. The corresponding 
eigenvalues of $P(\xi_2)$ in (\ref{outl8}) are given by $\xi_2 + \eps \mu_k(\widetilde{h})$, and from~\cite{HeSjSt},~\cite{HiPrSt} we also know that 
\begeq
\label{outl11}
\norm{(P(\xi_2) - z)^{-1}}_{{\cal L}(L^2,L^2)} \leq {\cal O}\left(\frac{1}{\sqrt{\eps} h}\right),
\endeq
provided that $\abs{z-\xi_2}\leq Ch\sqrt{\eps}$ and that $(z-\xi_2)/h\sqrt{\eps}$ avoids the quadratic eigenvalues $\mu_{k,0}$ in (\ref{outl10.1}). 

\bigskip
\noindent
Now (\ref{outl8}) is only an approximate direct sum decomposition, and in order to be able to absorb the error terms there, when constructing the resolvent of $P_{\eps}$
globally, it is of crucial importance to control the resolvent norms of $L_{\eps}$ also near the interval $[Ch\sqrt{\eps}, 1/{\cal O}(1))$. To get such a control,
since the spectral parameter remains close to the boundary of the range of the symbol of $L_{\eps}$, we apply the method of "bounded exponential weights", which 
in effect consists of replacing $L_{\eps}$ by a new operator for which the infimum of the imaginary part is increased in the non-elliptic region for 
${\rm Re}\left(L_{\eps}-\omega\right)$. This method has been carried out in closely related situations in~\cite{DeSjZw},~\cite{HeSjSt},~\cite{HiPrSt}, and we 
apply some of those works in the actual proof in the Appendix. Here we shall merely recall the essential ideas. See also~\cite{LeLe},~\cite{PrSt1}.  

\medskip
\noindent
Let $G(x_1,\xi_1)\in C^{\infty}$ be real-valued and odd in $\xi_1$. Let us consider formally the conjugated operator 
$$
\widetilde{L}_{\eps} = e^{-\eps G(x_1,hD_{x_1})/h}\circ L_{\eps} \circ e^{\eps G(x_1,hD_{x_1})/h},
$$
acting on $L^2$, or equivalently, the operator $L_{\eps}$ acting on the weighted Hilbert space $e^{\eps G(x_1,hD_{x_1})/h}L^2$. We want this space to be equal to 
$L^2$, with its norm 
$$
\norm{e^{-\eps G(x_1,hD_{x_1})/h}u}_{L^2}
$$
uniformly equivalent to the standard $L^2$-norm. This is the case if the weight function $G$ satisfies suitable symbol estimates and has the fundamental property
\begeq
\label{outl12}
\frac{\eps G(x_1,\xi_1)}{h} = {\cal O}(1),
\endeq
uniformly with respect to the various parameters involved. 

\medskip
\noindent
We view $e^{\eps G({x_1},hD_{x_1})/h}$ as a Fourier integral operator with the associated canonical transformation $\exp (i\eps H_G)$, approximately
equal to $({x_1},{\xi_1} )\mapsto ({x_1},{\xi_1} )+i\eps H_G ({x_1},{\xi_1} ) $, since $\eps$ will be small. 
Here $H_G=G'_{\xi_1} \cdot \partial _{x_1}-G'_{x_1}\cdot \partial _{\xi_1} $ is the Hamilton vector field of $G$. By Egorov's theorem we expect
$\widetilde{L}_\eps$ to be an $h$-pseudodifferential operator with the symbol
\[
\begin{split}
\widetilde{L}_\eps ({x_1},{\xi_1} )\approx L_\eps \left(\exp
  \left(i\eps H_G(x_1,\xi _1)\right) \right)
\approx L_\eps\left(\left({x_1},{\xi_1}\right) +i\eps
  H_G(L_\eps)(x_1,\xi _1)\right)
\\
\approx L_\eps ({x_1},{\xi_1})-i\eps H_{{L_\eps}}(G).
\end{split}
\]
Here $L_\eps({x_1},{\xi_1} ) = \xi_1^2+i\eps\widetilde{q}({x_1})$ is the symbol of $L_\eps$ in (\ref{outl9}). With $\ell({x_1},{\xi_1} )={\xi_1}^2$, we get
$$
\widetilde{L}_\eps ({x_1},{\xi_1} )\approx {\xi_1} ^2+i\eps (\widetilde{q}-H_\ell (G))({x_1},{\xi_1} )=:{\xi_1} ^2+i\eps \widehat{q}({x_1},{\xi_1}).
$$
When considering $\widetilde{L}_\eps -\omega $ for ${\rm Re}\, \omega \geq h \sqrt{\eps}$, the most critical region is the one where
${\xi_1} ^2\approx \Re \omega $ and it is here that we want to increase $\inf_{x_1}\widetilde{q}$ as much as possible. Naturally, that will not be enough for 
the complete analysis, but in the following heuristic discussion, we shall restrict the attention to the region where $\xi_1^2=\Re \omega$. Here, we get
\[
\begin{split}
\widehat{q}({x_1},{\xi_1})=&\widetilde{q}({x_1})-2 \xi_1 \partial_{x_1} G({x_1},{\xi_1} )\\
=&\widetilde{q}({x_1})-2\sqrt{\Re\omega }\partial _{x_1} G({x_1},(\Re \omega)^{1/2} ),
\end{split}
\]
where we recall that $G$ is odd in ${\xi_1} $, so that $\widehat{q}$ is even in the same variable. Then 
$$
\partial_{x_1}G({x_1})=\frac{\widetilde{q}({x_1})-\widehat{q}({x_1})}{2\sqrt{\Re \omega}},
$$
omitting ${\xi_1} =\sqrt{\Re \omega }$ in the argument of $G$. We want
\begin{equation}
\label{outl13}
\inf_{x_1} \widehat{q}-\inf_{x_1}\widetilde{q}\asymp \gamma ^2,
\end{equation}
for a suitable small parameter $\gamma$, that we wish to have as large as possible, and to achieve this, we clearly have to modify 
$\widetilde{q}$ in a $\gamma $-neighborhood of $x_1^{\rm min}$. Since we also wish $\abs{G}$ to be as small as possible, we require
$$
\mathrm{supp\,}G\subset [x_1^{\rm min}-\gamma,x_1^{\rm min}+\gamma ],
$$
and it is not hard to see that we can find such a $G$ with
$$
\partial _{x_1}G={\cal O}\left(\frac{\gamma ^2 }{\sqrt{\Re \omega }} \right),\quad G={\cal O}\left(\frac{\gamma ^3 }{\sqrt{\Re \omega }}\right) 
$$
The condition (\ref{outl12}) is fulfilled, provided that
\begeq
\label{outl14}
\frac{\eps \gamma ^3}{h\sqrt{{\rm Re}\, \omega }}={\cal O}(1) \Longleftrightarrow 
\gamma ={\cal O}(1)\frac{h^{\frac{1}{3}}({\rm Re}\,\omega)^{\frac{1}{6}}}{\eps^{\frac{1}{3}}}.
\endeq
Let $C\gg 1$ and let us choose
\begin{equation}
\label{outl15}
\gamma =\frac{1}{C}\min \left( 1,\frac{h^{\frac{1}{3}}({\rm Re}\, \omega )^{\frac{1}{6}}}{\eps^{\frac{1}{3}}} \right). 
\end{equation}
It follows from the heuristic discussion above that in the region where $h\sqrt{\eps}\leq {\rm Re}\,\omega \leq 1/{\cal O}(1)$ we obtain the spectral gain 
\begin{equation}
\label{outl16}
\frac{\eps \gamma ^2}{{\cal O}(1)}\asymp \min \left(\eps,h^{\frac{2}{3}}(\Re \omega )^{\frac{1}{3}}\eps^{\frac{1}{3}} \right) \geq h\sqrt{\eps},
\end{equation}
in the sense that the resolvent $(L_{\eps} - \omega)^{-1}$ is well defined in the region 
$$
h\sqrt{\eps} \le \Re \omega \le \frac{1}{{\cal O}(1)},\ \ \Im \omega \le \frac{1}{C}
\min \left( h^{\frac{2}{3}}(\Re \omega )^{\frac{1}{3}}\eps^{\frac{1}{3}},\eps \right)
$$
and that in such a region, we have 
\begin{equation}
\label{outl18}
\norm{(L_\eps-\omega)^{-1}}_{{\cal L}(L^2,L^2)} \le \frac{{\cal
      O}(1)}{\min \left( h^{\frac{2}{3}}\eps^{\frac{1}{3}}(\Re \omega )^{\frac{1}{3}},\eps \right)}.
\end{equation}
The resolvent estimates such as (\ref{outl18}) are established in the Appendix, using the machinery of bounded exponential weights and relying on the 
techniques of~\cite{HeSjSt},~\cite{HiPrSt} --- see Propositions A.2 and A.4 there, in particular. With the bounds (\ref{outl18}) available, we get 
the corresponding pseudospectral control over the family $P(x_1,hD_{x_1},\xi_2,\eps;h)$ in (\ref{outl8}), in the region where
$\abs{{\rm Re}\,z - \xi_2}\geq C h\sqrt{\eps}$, ${\rm Im}\, z \leq {\cal O}(h\sqrt{\eps})$, and this allows us, eventually, to construct the resolvent of 
$P_{\eps}$ globally in this region. We therefore obtain some crucial spectral localization, making it possible to carry out the spectral analysis of 
$P_{\eps}$ working with one quantum number $\xi_2=hj$ at a time, roughly speaking. A globally well-posed Grushin problem for $P_{\eps}$ is finally built 
from the corresponding one-dimensional Grushin problems for the operator $L_{\eps}$ in (\ref{outl9}), and solving it along the same lines as 
in~\cite{HiSj1},~\cite{HiSjVu07},~\cite{HeSjSt}, we complete the proof of Theorem 2.1. 

\medskip
\noindent
{\it Remark}. Our heuristic arguments seem to indicate that the optimal range for the perturbation parameter $\eps$ could be 
\begeq
\label{outl19}
h^2 \ll \eps \ll h^{2/3},
\endeq
as we need $\widetilde{h} = h/\sqrt{\eps} \ll 1$ and $\eps \ll \widetilde{h}$. Due to many technicalities, we get a smaller range of values 
around $\eps \approx h$, and leave extension to the range (\ref{outl19}) as an open problem for future works.

\section{Secular reduction and the global weight}
\label{sec}
\setcounter{equation}{0}
The purpose of this section is to construct a globally defined compactly supported weight function, which will allow us to microlocalize the spectral problem
for $P_{\eps}$ to a small neighborhood of the rational torus $\Lambda_0$. In doing so, we shall proceed similarly to~\cite{HiSjVu07}, with the essential difference
that when working near the torus, the basic cohomological equation will have quite different properties, compared to the Diophantine analysis of~\cite{HiSjVu07},
and will be treated using the secular perturbation theory, see~\cite{LcLb},~\cite{HiSj08b}.

\medskip
\noindent
Let us keep all the assumptions of Section 2 and consider the operator $P_{\eps}$ with the leading symbol $p_{\eps}$ in (\ref{st.10}), in a neighborhood of
$p^{-1}(0)\cap T^*M$. Let
\begeq
\label{eq3.1}
\kappa_0: {\rm neigh}(\Lambda_0, T^*M) \rightarrow {\rm neigh}(\xi = 0, T^*{\bf T}^2),
\endeq
be a real analytic canonical transformation, given by the action-angle variables, such that the properties (\ref{st.21}) hold. By Taylor expansion and
(\ref{st.21}), we have
\begeq
\label{eq3.2}
p(\xi) = p(f(\xi_2),\xi_2) + g(\xi) \left(\xi_1 - f(\xi_2)\right)^2, \quad g(0) > 0,
\endeq
where $f$ is the analytic function introduced in (\ref{st.21.5}).

\medskip
\noindent
Implementing $\kappa_0$ in (\ref{eq3.1}) by means of a microlocally unitary multi-valued $h$--Fourier integral operator with a real phase, as explained in
Theorem 2.4 in~\cite{HiSj1} and conjugating $P_{\eps}$ by this operator, we obtain a new $h$-pseudodifferential operator, still denoted by $P_{\eps}$, defined
microlocally near $\xi=0$ in $T^*{\bf T}^2$. The full symbol of $P_{\eps}$ is holomorphic in a fixed complex \neigh{} of $\xi=0$, and the leading symbol is given by
\begeq
\label{eq3.3}
p_{\eps}(x,\xi)=p(\xi)+i\eps q(x,\xi)+{\cal O}(\eps^2),
\endeq
with $p(\xi)$ of the form (\ref{eq3.2}). The function $q$ in (\ref{eq3.3}) is real on the real domain. On the operator level, $P_{\eps}$ acts on the space of
microlocally defined Floquet periodic functions on ${\bf T}^2$, $L^2_{\theta}({\bf T}^2)\subset L^2_{{\rm loc}}(\real^2)$, elements $u$ of which satisfy
\begeq
\label{eq3.4}
u(x-\nu)=e^{i\theta\cdot \nu} u(x),\quad \theta=\frac{S}{2\pi h}+\frac{k_0}{4},\quad \nu\in 2\pi \z^2.
\endeq
Here $S=(S_1,S_2)$ is given by the classical actions,
$$
S_j=\int_{\alpha_j}\eta\,dy,\quad j=1,2,
$$
with $\alpha_j$ forming a system of fundamental cycles in
$\Lambda_{0}$, such that
$$
\kappa_0(\alpha_j)=\beta_j,\quad j=1,2, \quad \beta_j=\{x\in {\bf T}^2;\, x_{3-j}=0\}.
$$
The tuple $k_0=(k_0(\alpha_1),k_0(\alpha_2))\in \z^2$ stands for the Maslov indices of the cycles $\alpha_j$, $j=1,2$.

\medskip
\noindent
{\it Remark}. Using (\ref{eq3.2}), we see, using the implicit function theorem, that the energy surface $p(\xi) = E$, for $E\in {\rm neigh}(0,\real)$,
is given by
\begeq
\label{eq3.41}
\xi_2 + \ell(\xi_1,E)=0,
\endeq
where $\ell$ is analytic with $\ell(\xi_1,0) \sim \xi_1^2$, $\ell'_E(0,0)<0$.

\bigskip
\noindent
Working near the zero section $\xi=0$ in $T^*{\bf T}^2$ and following the method of normal forms~\cite{HiSjVu07},~\cite{HiSj08b}, we shall now discuss the
cohomological equation
\begin{equation}
\label{eq3.5}
H_p G=q-\widetilde{q},
\end{equation}
where we want the remainder $\widetilde{q}$ to be simpler than $q$. Here we have
$$
H_p=p'_\xi \cdot \partial_x,
$$
and thus, (\ref{eq3.5}) can be written more explicitly as follows,
$$
\partial _{\xi _2}p(\xi )\partial _{x_2}G+\partial _{\xi _1}p(\xi )\partial _{x_1}G =q-\widetilde{q}.
$$
To simplify, we divide this equation by $\partial _{\xi_2}p$. Writing $u=G$, $v=(\partial _{\xi _2}p)^{-1}q$,
$\widetilde{v}=(\partial _{\xi _2}p)^{-1}\widetilde{q}$, we get
\begin{equation}
\label{eq3.6}
(\partial_{x_2}+a(\xi )\partial _{x_1})u=v-\widetilde{v},
\end{equation}
where $a(\xi)=\partial_{\xi _1}p(\xi )/\partial _{\xi _2}p(\xi)$. To simplify further, we replace the variables $\xi$ by
\begin{equation}
\label{eq3.7}
\eta = (\eta_1,\eta_2)=(\xi_1-f(\xi_2),\xi _2),
\end{equation}
and write, abusing the notation slightly, $u=u(x,\eta)$, $v=v(x,\eta)$, $\widetilde{v}=\widetilde{v}(x,\eta )$. It follows from (\ref{eq3.2}) that the Taylor
expansion of $a$ has the form,
\begin{equation}
\label{eq3.8}
a(\eta )=a_1(\eta_2)\eta_1+a_2(\eta _2)\eta _1^2+\ldots,\ a_1(0)\ne 0,
\end{equation}
and let us Taylor expand $u$, $v$ and $\widetilde{v}$ similarly,
\begin{equation}
\label{eq3.9}
\begin{split}
u(x,\eta )&=\sum_{k=0}^\infty u_k(x,\eta_2)\eta_1^k,\\
v(x,\eta )&=\sum_{k=0}^\infty v_k(x,\eta_2)\eta_1^k,\\
\widetilde{v}(x,\eta)&=\sum_{k=0}^\infty \widetilde{v}_k(x,\eta_2)\eta_1^k.
\end{split}
\end{equation}
Inserting these equations into (\ref{eq3.6}) and identifying the powers of $\eta_1$, we get
\begin{equation}
\label{eq3.10}
\partial_{x_2}u_0=v_0-\widetilde{v}_0,
\end{equation}
\begin{equation}
\label{eq3.11}
\partial _{x_2}u_1+a_1\partial _{x_1}u_0=v_1-\widetilde{v}_1,
\end{equation}
\begin{equation}
\label{eq3.12}
\partial _{x_2}u_2+a_1\partial _{x_1}u_1+a_2\partial _{x_1}u_0=v_2-\widetilde{v}_2,
\end{equation}
and so on. The general equation is of the form
\begin{equation}
\label{eq3.13}
\partial _{x_2}u_k+a_1\partial _{x_1}u_{k-1}+a_2\partial _{x_1}u_{k-2}+ \ldots + a_k\partial _{x_1}u_0=v_k-\widetilde{v}_k.
\end{equation}
The parameter $\eta_2$ plays no essential role here and we sometimes suppress it from the notation. For a function $u$ on the torus ${\bf T}^2$, we
introduce its averages in $x_k$, $k=1,2$, and its total average by
\[\begin{split}
&\langle u\rangle_k(x_{3-k})=\frac{1}{2\pi }\int_0^{2\pi }u(x_1,x_2)dx_k,\\
&\langle \langle u\rangle\rangle =\langle \langle
u\rangle_1\rangle_2=\frac{1}{(2\pi )^2}\iint _{{\bf T}^2}u(x_1,x_2)dx_1dx_2.
\end{split}
\]

\begin{prop}
\label{sec1}
Let $v_0,\, v_1,\ldots \in C^\infty ({\bf T}^2)$ be smooth functions on ${\bf T}^2$. A necessary and sufficient condition on the smooth functions
$\widetilde{v}_0,\, ,\widetilde{v}_1,\ldots \in C^\infty ({\bf T}^2)$ for the existence of $u_0,\, u_1,\ldots \in C^\infty ({\bf T}^2)$ solving
{\rm (\ref{eq3.10})} and {\rm (\ref{eq3.13})} for $k\ge 1$, is that
\begin{equation}
\label{eq3.14}
\begin{split}&\langle \widetilde{v}_0\rangle_2=\langle v_0\rangle_2,\\
&\langle \langle \widetilde{v}_k\rangle\rangle = \langle \langle
v_k\rangle\rangle,\ k\ge 1.
\end{split}
\end{equation}
\end{prop}
\begin{proof}
The necessity of (\ref{eq3.14}) follows from taking the $x_2$-mean of (\ref{eq3.10}) and the total mean of (\ref{eq3.13}).

\medskip
\noindent
Assume that the first equation in (\ref{eq3.14}) holds so that $\langle v_0-\widetilde{v}_0\rangle_2=0$. Then (\ref{eq3.10}) has a solution
$u_0=u_0^0\in C^\infty ({\bf T}^2)$, given by
\begin{equation}
\label{eq3.15}
u_0^0(x)=\int_0^{x_2} (v_0-\widetilde{v}_0)(x_1,t)dt.
\end{equation}
The general solution of (\ref{eq3.10}) is of the form $u_0^0(x)+f_0(x_1)$, where $f_0(x_1)$ is any smooth periodic function.

\medskip
\noindent
We next consider (\ref{eq3.11}) (i.e. (\ref{eq3.13}) with $k=1$), which we write as
\begin{equation}
\label{eq3.16}
\partial _{x_2}u_1=v_1-\widetilde{v}_1-a_1\partial_{x_1}u_0^0-a_1\partial _{x_1}f_0(x_1).
\end{equation}
Here the total average of $v_1-\widetilde{v}_1-a_1\partial_{x_1}u_0^0$ vanishes,
$$
\langle \langle v_1-\widetilde{v}_1-a_1\partial_{x_1}u_0^0 \rangle_2\rangle_1=0,
$$
and hence we can find a periodic smooth function $f_0(x_1)$, unique up to a constant, such that
$$
\langle v_1-\widetilde{v}_1-a_1\partial_{x_1}u_0^0 \rangle_2 -a_1\partial _{x_1}f_0(x_1)=0.
$$
Equivalently,
$$
\langle v_1-\widetilde{v}_1-a_1\partial_{x_1}u_0^0 -a_1\partial _{x_1}f_0(x_1)\rangle_2 =0,
$$
and we can therefore find a solution $u_1^0\in C^\infty ({\bf T}^2)$ to (\ref{eq3.16}), and hence to (\ref{eq3.11}).

\medskip
\noindent
Assume by induction that we have found $u_0,\, u_1,\ldots\, u_{k-1}$, solving (\ref{eq3.10}) and $(\ref{eq3.13})$, with $k$ there replaced by
$j=1,2,\ldots,k-1$. We notice that the general solution of (\ref{eq3.13}) with $k$ replaced by $k-1$, is of the form $u_{k-1}=u_{k-1}^0+f_{k-1}(x_1)$ for any
smooth periodic function $f_{k-1}$. We rewrite (\ref{eq3.13}) as
\begin{equation}
\label{eq3.17}
\partial _{x_2}u_k=w_k-a_1\partial _{x_1}f_{k-1}(x_1),
\end{equation}
where,
$$
w_k=v_k-\widetilde{v}_k-a_1\partial _{x_1}u_{k-1}^0-a_2\partial_{x_1}u_{k-2}-...-a_k\partial _{x_1}u_0,
$$
and we notice that $\langle \langle w_k\rangle_2\rangle_1=\langle\langle w_k\rangle\rangle=0$. Choose $f_{k-1}$ such that $\langle
w_k\rangle_2=a_1\partial _{x_1}f_{k-1}(x_1)$, or equivalently, so that $\langle w_k-a_1\partial _{x_1}f_{k-1}\rangle_2=0$. Then there is a smooth
periodic solution $u_k=u_k^0$ to (\ref{eq3.17}) and hence to (\ref{eq3.13}).
\end{proof}

\bigskip
\noindent
An application of Proposition \ref{sec1} allows us to conclude that for any fixed $N\in \nat$, there exists an analytic function $G_0$, defined in a fixed
neighborhood of $\xi = 0$, such that
\begeq
\label{eq3.17.1}
H_p G_0 = q - \widetilde{q} + {\cal O}((\xi_1-f(\xi _2))^N),
\endeq
for any analytic periodic function $\widetilde{q}$ which satisfies,
\begin{equation}
\label{eq3.18}
\langle \widetilde{q}(\cdot,\xi )\rangle_2=\langle q(\cdot ,\xi)\rangle_2,\hbox{ when }\xi _1=f(\xi _2)
\end{equation}
and
\begin{equation}
\label{eq3.19}
\langle \langle \widetilde{q}(\cdot ,\xi )\rangle\rangle= \langle
\langle q(\cdot ,\xi )\rangle\rangle,\ \xi \in \mathrm{neigh\,}(0,\real^2).
\end{equation}
The following choice is convenient and will be made in what follows: let $\chi: \real \rightarrow \real$ be real analytic such that $\chi(0)=0$.
We can then take
\begeq
\label{eq3.20}
\widetilde{q}(x_1,\xi)=(1-\chi (\xi _1-f(\xi _2)))\langle q\rangle_2(x_1,\xi)+\chi (\xi _1-f(\xi _2))\langle \langle q(\cdot,\xi)\rangle\rangle,
\endeq
which is independent of $x_2$.

\medskip
\noindent
Let us now restrict the attention to the energy surface $p^{-1}(0)$. According to (\ref{eq3.41}), we have $p(\xi) = 0 \Leftrightarrow \xi_2 + \ell(\xi_1,0) = 0$,
$\ell(\xi_1,0) \sim \xi_1^2$. It follows from (\ref{eq3.20}) that when $p(\xi) = 0$, we may write
$$
\inf_{x_1}\widetilde{q}(x_1,\xi) = \left(1-\psi(\xi_1)\right) k(\xi_1) + \psi(\xi_1) g(\xi_1),
$$
where
$$
k(\xi_1) = \inf_{x_1} \langle{q\rangle}_2(x_1,\xi_1,-\ell(\xi_1,0))  = \langle{q\rangle}_2(x_1(\xi_1,-\ell(\xi_1,0)),\xi_1,-\ell(\xi_1,0)),
$$
$g(\xi_1) = \langle \langle q\rangle\rangle(\xi_1,-\ell(\xi_1,0))$, and $\psi$ is an analytic function such that $\psi(0) = 0$, $\psi'(0) = \chi'(0)$. We
compute next
$$
\partial_{\xi_1} \inf_{x_1}\widetilde{q}(x_1,\xi) = k'(\xi_1) + \psi'(\xi_1)\left(g(\xi_1)-k(\xi_1)\right) + \psi(\xi_1)\left(g'(\xi_1)-k'(\xi_1)\right),
$$
\begin{multline*}
\partial_{\xi_1}^2 \inf_{x_1}\widetilde{q}(x_1,\xi) = k''(\xi_1) + \psi''(\xi_1)\left(g(\xi_1)-k(\xi_1)\right) + 2\psi'(\xi_1)\left(g'(\xi_1)-k'(\xi_1)\right)\\
+ \psi(\xi_1)\left(g''(\xi_1)-k''(\xi_1)\right).
\end{multline*}
Using that
$$
k(0) = \inf_{x_1 \in {\bf T}} \langle{q\rangle}_2(x_1,0) = \inf Q_{\infty}(\Lambda_0) < \langle{q\rangle}(\Lambda_0) =
\langle \langle q\rangle\rangle(0) = g(0),
$$
we see that the derivatives $\chi'(0)$ and $\chi''(0)$ of the analytic function $\chi$ in (\ref{eq3.20}) can be chosen so that when $p(\xi) = 0$, we have
\begeq
\label{eq3.21}
\inf_{x_1\in {\bf T}} \widetilde{q}(x_1,\xi) \geq \inf_{x_1 \in {\bf T}} \langle{q\rangle}_2(x_1,0) + C \xi^2,
\endeq
where the constant $C>0$ is large. In other words, for $\Lambda \in {\rm neigh}(\Lambda_0,J)$, we get
\begeq
\label{eq3.21.1}
\inf_{\Lambda} \widetilde{q} \geq \inf Q_{\infty}(\Lambda_0) + \frac{1}{{\cal O}(1)} {\rm dist}(\Lambda,\Lambda_0)^2.
\endeq

\medskip
\noindent
{\it Remark}. In the preceding discussion, we do not have to restrict ourselves to the energy surface $p^{-1}(0)$. Indeed, introducing the variables
$$
\eta = (\eta_1,\eta_2) = (\xi_1 - f(\xi_2), \xi_2),
$$
as in (\ref{eq3.7}), and repeating the computations above we get
$$
\inf_{x_1\in {\bf T}}\widetilde{q}(x_1,\xi) \geq \inf_{x_1 \in {\bf T}} \langle{q\rangle}_2(x_1,f(\xi_2),\xi_2) + C \left(\xi_1- f(\xi_2)\right)^2,
$$
when $p(\xi) = E$, for $E\in {\rm neigh}(0,\real)$. Introducing the rational Lagrangian tori $\Lambda_E \subset p^{-1}(E)$, defined in (\ref{st.21.6}),
we therefore obtain on $p^{-1}(E)$,
\begeq
\label{eq3.21.01}
\inf_{x_1\in {\bf T}} \widetilde{q}(x_1,\xi) \geq \inf Q_{\infty}(\Lambda_E) + C_1 {\rm dist}(\Lambda,\Lambda_E)^2.
\endeq

\bigskip
\noindent
We shall now construct a suitable global weight function. In doing so, let $G_T$ be an analytic function defined in a neighborhood of $p^{-1}(0)\cap T^*M$, such that
\begeq
\label{eq3.22}
H_p G_T = q - \langle{q\rangle}_T,
\endeq
where
$$
\langle{q\rangle}_T = \frac{1}{T}\int_{-T/2}^{T/2} q\circ \exp(tH_p)\,dt,\quad T>0
$$
has been introduced in (\ref{st.13}). An application of Lemma 2.4 of~\cite{HiSjVu07} together with the assumption (\ref{st.19}) allows us to conclude that outside
an arbitrarily small neighborhood of $\Lambda_0$ in $p^{-1}(0)\cap T^*M$, we have
\begeq
\label{eq3.23}
\inf\left(q - H_p G_T\right) \geq \inf Q_{\infty}(\Lambda_0) + \frac{1}{C_0},
\endeq
provided that $T$ is taken large enough. Here $C_0>0$ is independent of the neighborhood taken. In these considerations, we are allowed to vary the real energy a
little, and we conclude that for any fixed neighborhood $W$ of
\begeq
\label{eq3.23.1}
\bigcup_{\abs{E} \leq E_0} \Lambda_E,\quad 0 < E_0 \ll 1,
\endeq
in $p^{-1}([-E_0,E_0])$ there exists $T$ large enough such that
\begeq
\label{eq3.24}
\inf_{p^{-1}([-E_0,E_0])\backslash W} \left(q  - H_p G_T\right) \geq \inf_{\abs{E}\leq E_0}\, \inf Q_{\infty}(\Lambda_E) + \frac{1}{C_0}.
\endeq
Here $C_0 > 0$ is independent of the neighborhood chosen.

\medskip
\noindent
The global weight function will be obtained by gluing together the functions $G_T:=G_T\circ \kappa_0^{-1}$ and $G_0$, both viewed as analytic functions
defined in a neighborhood of the zero section $\xi = 0$ in $T^*{\bf T}^2$. Let $\psi = \psi(\xi) \in C^{\infty}(\real^2; [0,1])$ be constant on each
invariant torus $\xi = {\rm Const}$, and assume that $\psi = 1$ near the rational region (\ref{eq3.23.1}), and with support in a small neighborhood of that set.
Let us set,
\begeq
\label{eq3.25}
G = \left(1 - \psi\right) G_T + \psi G_0.
\endeq
It follows that
\begeq
\label{eq3.26}
q - H_p G = \psi \left(q-H_p G_0\right) + (1-\psi) \langle{q\rangle}_T.
\endeq
In a neighborhood of the rational region (\ref{eq3.23.1}), we have
$$
q - H_p G = \widetilde{q} + {\cal O}\left((\xi_1 - f(\xi_2))^N\right),
$$
with $\widetilde{q}$ given in (\ref{eq3.20}), while further away from this set, we have
$q - H_p G = \langle{q\rangle}_T$. In order to understand the behavior of $\langle{q\rangle}_T$ near $\xi = 0$ for $T$ large, we write
$$
\langle{q\rangle}_T(x,\xi) = \frac{1}{T}\int_{-T/2}^{T/2} q(x+tp'(\xi),\xi)\,dt,
$$
and expanding $q(\cdot,\xi)$ in a Fourier series, we obtain that
\begeq
\label{eq3.27}
\langle{q\rangle}_T(x,\xi) = \sum_{k = (k_1,k_2) \in {\bf Z}^2} e^{ik\cdot x} \widehat{q}(k,\xi) \widehat{K}(T k\cdot p'(\xi)).
\endeq
Here $\widehat{K}$ is the Fourier transform of the characteristic function $K$ of the interval $[-1/2,1/2]$. Let us decompose,
\begeq
\label{eq3.28}
\langle{q\rangle}_T (x,\xi) = \sum_{k_2\neq 0} \widehat{K}(T p'(\xi)\cdot k)\widehat{q}(k,\xi) e^{ix\cdot k} +
\sum_{k_2=0} \widehat{K}(T p'(\xi)\cdot k)\widehat{q}(k,\xi) e^{ix\cdot k}
= {\rm I}+{\rm II},
\endeq
with the natural definitions of ${\rm I}$ and ${\rm II}$. When estimating ${\rm I}$, we use (\ref{eq3.2}) and notice that when $k_2\neq 0$, we have
$$
\abs{p'(\xi)\cdot k}\geq \abs{p'_{\xi_2} k_2}-{\cal O}(1)\abs{\xi_1-f(\xi_2)}\abs{k_1} \geq 1 - C\abs{\xi_1 - f(\xi_2)}\abs{k},\quad C>0.
$$
Here for notational simplicity we assume that the derivative of the function $\xi_2 \mapsto p(f(\xi_2),\xi_2)$ is $\geq 1$ near $0$. It follows that
$$
\abs{p'(\xi)\cdot k} \geq \frac{1}{2},
$$
provided that $2C\abs{\xi_1-f(\xi_2)}\abs{k}\leq 1$. Let now $0\leq \chi\in C^{\infty}_0((-1,1))$ be such that $\chi=1$ on $[-1/2,1/2]$ and write,
\begin{eqnarray}
\label{eq3.29}
{\rm I}& = & \sum_{k_2\neq 0} \chi(2C\abs{\xi_1-f(\xi_2)}\abs{k})
\widehat{K}(T p'(\xi)\cdot k) \widehat{q}(k,\xi)e^{ix\cdot k} \\ \nonumber
& & + \sum_{k_2\neq 0} (1-\chi(2C\abs{\xi_1-f(\xi_2)}\abs{k}))
\widehat{K}(T p'(\xi)\cdot k) \widehat{q}(k,\xi)e^{ix\cdot k} \\ \nonumber
& = & \sum_{k_2\neq 0} \chi(2C\abs{\xi_1-f(\xi_2)}\abs{k})
{\cal O}\left(\frac{1}{T \abs{p'(\xi)\cdot k}}\right) \widehat{q}(k,\xi) e^{ix\cdot k} \\ \nonumber
& & +\sum_{k_2\neq 0} (1-\chi(2C\abs{\xi_1-f(\xi_2)}\abs{k}))
\widehat{K}(T p'(\xi)\cdot k) \widehat{q}(k,\xi)e^{ix\cdot k}.
\end{eqnarray}
It is now easy to see, using the smoothness of $q$, that
\begeq
\label{eq3.30}
{\rm I}={\cal O}\left(\frac{1}{T} + \abs{\xi_1 - f(\xi_2)}^{\infty}\right),\quad T\geq 1.
\endeq
When considering the contribution coming from ${\rm II}$,  we notice that
\begeq
\label{eq3.31}
{\rm II} = \langle{q\rangle}_2(x_1,\xi) + \sum_{k_2=0, k_1\neq 0} \left(\widehat{K}(T p'_{\xi_1}k_1) -1\right) e^{ix_1 k_1}
\widehat{q}(k,\xi).
\endeq
Here $\abs{p'_{\xi_1}} \sim \abs{\xi_1-f(\xi_2)}$, in view of (\ref{eq3.2}), and we conclude that in the rational region where $\xi_1 = f(\xi_2)$,
$\xi_2 \in {\rm neigh}(0,\real)$, we get
$$
\langle{q\rangle}_T(x,\xi) = \langle{q\rangle}_2(x_1,\xi) + {\cal O}\left(\frac{1}{T}\right).
$$
Away from the rational region $\xi_1 = f(\xi_2)$, we see directly from (\ref{eq3.28}) that ${\rm II}$ converges to the torus average
$\langle\langle{q\rangle\rangle}(\xi)$, as $T\rightarrow \infty$.

\bigskip
\noindent
Combining the equations and estimates (\ref{eq3.17.1}), (\ref{eq3.21.1}), (\ref{eq3.22}), (\ref{eq3.24}), and (\ref{eq3.25}), we may summarize the discussion
above in the following proposition.
\begin{prop}
Let us make the assumption {\rm (\ref{st.19})}. Let $G_0$ be an analytic solution near $\xi=0$ of the equation {\rm (\ref{eq3.5})}, with $\widetilde{q}$ being of
the form {\rm (\ref{eq3.20})}, modulo ${\cal O}((\xi_1 - f(\xi_2))^N)$, for some $N$ fixed large enough. There exists a real-valued function
$G\in C^{\infty}_0(T^*M)$ such that $G = G_0\circ \kappa_0$ in a neighborhood of $\Lambda_0$, and such that away from a small neighborhood of $\Lambda_0$ in
the region $p^{-1}([-E_0,E_0])$, $0 < E_0 \ll 1$, we have
\begeq
\label{eq3.32}
q - H_p G \geq {\inf}\,Q_{\infty}(\Lambda_0) + \frac{1}{C_0},\quad C_0 > 0.
\endeq
When $\Lambda \subset p^{-1}(0)$, $\Lambda \in {\rm neigh}(\Lambda_0,J)$, we have furthermore
$$
\inf_{\Lambda} \left(q - H_p G\right) \geq \inf Q_{\infty}(\Lambda_0) + \frac{1}{C}{\rm dist}(\Lambda,\Lambda_0)^2.
$$
\end{prop}

\bigskip
\noindent
Associated to the weight function $G$ defined in Proposition 4.2, we shall now introduce a suitable small but globally defined deformation of the real phase
space $T^*M$ into the complex domain. When doing so, let $\widetilde{G}\in C^{\infty}_0(T^*\widetilde{M})$ be an almost holomorphic extension of $G$, and let us set
\begeq
\label{eq3.32.1}
\Lambda_{\eps G} = \exp(\eps H_{{\rm Re}\, \widetilde{G}}^{{\rm Im}\, \sigma})(T^*M) \subset T^*\widetilde{M}.
\endeq
Here $\sigma$ is the complex symplectic (2,0)--form on $T^*\widetilde{M}$ and $H_{{\rm Re}\, \widetilde{G}}^{{\rm Im}\, \sigma}$ is the Hamilton vector field
of ${\rm Re}\, \widetilde{G}$ computed with respect to the real symplectic form ${\rm Im}\, \sigma$ on $T^*\widetilde{M}$. It follows that the manifold
$\Lambda_{\eps G}$ is I-Lagrangian and being a small deformation of $T^*M$, it is also R-symplectic, i.e. an IR-manifold. From~\cite{MeSj1} and~\cite{Sj82},
we recall the general relation
$$
\widehat{i\eps H_{\widetilde{G}}} = \eps H_{{\rm Re}\, \widetilde{G}}^{{\rm Im}\, \sigma},
$$
valid to infinite order along the real domain $T^*M$. Here $\widehat{i\eps H_{\widetilde{G}}}$ stands for the real vector field in $T^*\widetilde{M}$,
naturally associated to the complex (1,0) vector field,
$$
i\eps H_{\widetilde{G}} = i\eps \sum_{j=1}^2
\left(\frac{\partial \widetilde{G}}{\partial \xi_j}\frac{\partial}{\partial x_j} - \frac{\partial \widetilde{G}}{\partial x_j}\frac{\partial}{\partial \xi_j}\right).
$$
It follows that in the region where $G$ is analytic, including a sufficiently small but fixed neighborhood
of $\Lambda_0$, we have
$$
\Lambda_{\eps G} = \exp(i\eps H_G)(T^*M),
$$
where we write $G$ also for the holomorphic extension and recall that $\exp(i\eps H_G)$ is a holomorphic canonical transformation.

\medskip
\noindent
Associated to the IR-manifold $\Lambda_{\eps G}$ is the micro\-locally exponentially weigh\-ted Hil\-bert spa\-ce $H(\Lambda_{\eps G})$, defined using the
FBI--Bargmann approach, by modifying the exponential weight on the FBI transform side. We refer to~\cite{MeSj2},~\cite{HiSj12} for the detailed defi\-nition of
the space $H(\Lambda_{\eps G})$ in the case when $M = \real^2$, and to~\cite{Sj96} and the Appendix of~\cite{HiSj1} for the case when $M$ is compact.
Following~\cite{MeSj1},~\cite{MeSj2},~\cite{Sj96}, let us introduce a microlocally unitary $h$--Fourier integral operator
\begeq
\label{eq3.32.2}
{U}_G: L^2(M) \rightarrow H(\Lambda_{\eps G}),
\endeq
defined microlocally near $p^{-1}(0)\cap T^*M$ and associated to a suitable canonical transformation
$$
\kappa_G: {\rm neigh}(p^{-1}(0), T^*M) \rightarrow {\rm neigh}(p^{-1}(0), \Lambda_{\eps G}),
$$
such that $\kappa_G = \exp(i\eps H_G)$ near $\Lambda_0$. It follows that the operator
\begeq
\label{eq3.32.3}
P_{\eps}: H(\Lambda_{\eps G}) \rightarrow H(\Lambda_{\eps G})
\endeq
is microlocally near $p^{-1}(0)$ unitarily equivalent to the conjugated operator
$$
{U}_G^{-1} P_{\eps} {U}_G: L^2 \rightarrow L^2,
$$
with the leading symbol
$$
p_{\eps}|_{\Lambda_{\eps G}} \simeq p + i\eps \left(q-H_p G\right) + {\cal O}(\eps^2).
$$
Letting
$$
U_0: L^2(M) \rightarrow L^2_{\theta}({\bf T}^2)
$$
be the semiclassical microlocally unitary Fourier integral operator with a real phase associated to the canonical transformation $\kappa_0$ in (\ref{eq3.1})
and using the operator $U_0 U_G^{-1}$ associated to the canonical transformation
$$
\kappa_0 \circ \kappa_G^{-1}: {\rm neigh}(\exp(i\eps H_G)(\Lambda_0), \Lambda_{\eps G}) \rightarrow {\rm neigh}(\xi = 0, T^*{\bf T}^2),
$$
we get that microlocally near the Lagrangian torus $\exp(i\eps H_G)(\Lambda_0) \subset \Lambda_{\eps G}$, the operator
in (\ref{eq3.32.3}) is unitarily equivalent to an operator $\widetilde{P}_{\eps}$, acting on $L^2_{\theta}({\bf T}^2)$, defined microlocally near
$\xi = 0$ in $T^*{\bf T}^2$, given by
\begeq
\label{eq3.33}
\widetilde{P}_{\eps} \sim \sum_{\nu = 0}^{\infty} h^{\nu} \widetilde{p}_{\nu}(x,\xi,\eps).
\endeq
Here $\widetilde{p}_{\nu}$ are holomorphic functions in a fixed complex neighborhood of $\xi = 0$, smooth in $\eps \in {\rm neigh}(0,\real)$, and
\begeq
\label{eq3.34}
\widetilde{p}_0 = p(\xi) + i\eps \widetilde{q}(x_1,\xi) + {\cal O}(\eps^2) + \eps {\cal O}((\xi_1-f(\xi_2))^N),
\endeq
with $\widetilde{q}(x_1,\xi)$ independent of $x_2$ and of the form (\ref{eq3.20}). Furthermore, the assumption (\ref{st.24}) implies that
$$
\widetilde{p}_1 (x,\xi,\eps) = {\cal O}(\eps).
$$
We may illustrate the microlocal unitary equivalence above by the following commutative diagram,
\begeq
\label{eq3.35}
\begin{tikzcd}
P_{\eps}: H(\Lambda_{\eps G}) \arrow{r} \arrow[swap]{d}{U_0 U_G^{-1}} & H(\Lambda_{\eps G}) \arrow{d}{U_0 U_G^{-1}} \\
\widetilde{P}_{\eps}: L^2_{\theta}({\bf T}^2) \arrow{r}  & L^2_{\theta}({\bf T}^2)
\end{tikzcd}
\endeq

\medskip
\noindent
In what follows, we shall drop the tildes from the notation in (\ref{eq3.33}) and write simply $P_{\eps}$ and $p_{\nu}$, $\nu \geq 0$.

\section{Quantum normal forms near rational tori}
\label{qnf}
\setcounter{equation}{0}
In this section, we shall be concerned with a classical $h$--pseudodifferential operator $P_{\eps}(x,hD_x;h)$, defined microlocally near $\xi = 0$ in $T^*{\bf T}^2$,
given by the expansion (\ref{eq3.33}), with the leading symbol of the form (\ref{eq3.34}). Our purpose here is to obtain a normal secular reduction of $P_{\eps}$,
also on the level of lower order symbols, and this will be accomplished in a way very similar to~\cite{HiSjVu07},~\cite{HiSj08b}.

\medskip
\noindent
Let us first discuss the normal form construction at the level of principal symbols. In doing so, we let $\widetilde{q}_0 := \widetilde{q}$ in (\ref{eq3.34}), and
write
\begeq
\label{eq4.1}
p_0(x,\xi,\eps) = p(\xi) + i\eps \widetilde{q}_0(x_1,\xi) + i\eps^2 q_1(x,\xi) + {\cal O}\left(\eps^3 + \eps (\xi_1-f(\xi_2))^N\right).
\endeq
Arguing as in Section 3, we can construct an analytic function $G_1$, defined near $\xi = 0$, such that modulo ${\cal O}((\xi_1 - f(\xi_2))^N)$, we have
$$
H_p G_1 = q_1 - \widetilde{q}_1,
$$
where $\widetilde{q}_1$ is any analytic function satisfying (\ref{eq3.18}), (\ref{eq3.19}), with $q$ replaced there by $q_1$. It follows that
$$
p_0(\exp(i\eps^2 H_{G_1})(x,\xi)) = p(\xi) + i\eps \widetilde{q}_0(x_1,\xi) + i\eps^2 \widetilde{q}_1(x_1,\xi) + {\cal O}(\eps^3 + \eps (\xi_1-f(\xi_2))^N).
$$
Continuing this procedure, we get the following result.

\begin{prop}
\label{prop4.1}
Let $p_0(x,\xi,\eps) = p(\xi) + i\eps \widetilde{q}_0(x_1,\xi) + {\cal O}(\eps^2) + \eps {\cal O}((\xi_1-f(\xi_2))^N)$ be analytic defined near $\xi = 0$, depending
smoothly on $\eps \in {\rm neigh}(0,\real)$. Here $N\geq 2$ is arbitrarily large but fixed. Assume that
$$
p(\xi) = p(f(\xi_2),\xi_2) + g(\xi) (\xi_1 - f(\xi_2))^2,\quad g(0)>0,\quad f(0) = 0,
$$
where $p(f(\xi_2),\xi_2) = \alpha \xi_2 + {\cal O}(\xi_2^2)$, $\alpha>0$. There exists a holomorphic canonical transformation $\kappa_{\eps}^{(N)}$ of the form
\begeq
\label{eq4.2}
\kappa_{\eps}^{(N)} = \exp(i\eps^2 H_{G_1})\circ\cdots \circ \exp(i\eps^{N}H_{G_{N-1}}),
\endeq
with $G_j$ analytic near $\xi = 0$, $1\leq j \leq N-1$, such that modulo an error term of the form ${\cal O}(\eps^{N+1} + \eps (\xi_1-f(\xi_2))^N)$, we have
$$
p_0(\kappa_{\eps}^{(N)}(x,\xi)) \equiv p(\xi) + i\eps \widetilde{q}_0(x_1,\xi) + i\eps^2 (\widetilde{q}_1(x_1,\xi) + \ldots\, + \eps^{N-2}\widetilde{q}_{N-1})
$$
is independent of $x_2$. Here, as discussed before, $\widetilde{q}_0$ is any analytic function satisfying {\rm (\ref{eq3.18})}, {\rm (\ref{eq3.19})}, and
inductively $\widetilde{q}_k$ is any analytic function satisfying {\rm (\ref{eq3.18})}, {\rm (\ref{eq3.19})}, with $q$ there replaced by a certain function $q_k$ that depends on the
previously chosen $\widetilde{q}_0,\ldots \widetilde{q}_{k-1}$.
\end{prop}

\bigskip
\noindent
As will be discussed in Section 7, the complex canonical transformation $\kappa_{\eps}^{(N)}$ in (\ref{eq4.2}) can be quantized by means of an elliptic
classical $h$-Fourier integral operator $U_{\eps}$ in the complex domain, depending smoothly on $\eps \in {\rm neigh}(0,\real)$, introduced rigorously on the
FBI transform side. In this section, we shall proceed formally, and an application of Egorov's theorem allows us to conclude that the operator
$$
\widetilde{P}_{\eps}(x,hD_x;h) = U_{\eps}^{-1} P_{\eps}(x,hD_x;h) U_{\eps}
$$
is an $h$--pseudodifferential operator, defined microlocally near $\xi = 0$, whose symbol has a complete asymptotic expansion
\begeq
\label{eq4.21}
\widetilde{P}_{\eps}(x,\xi;h) \sim \widetilde{p}_0 + h\widetilde{p}_1 + \ldots ,
\endeq
with all $\widetilde{p}_j = \widetilde{p}_j(x,\xi,\eps)$ being smooth functions of $\eps \in {\rm neigh}(0,\real)$ with values in the space of holomorphic
functions in a fixed complex neighborhood of $\xi = 0$, such that
\begeq
\label{eq4.22}
\widetilde{p}_0(x,\xi,\eps) = p(\xi) + i\eps \widetilde{q}_0(x_1,\xi) + {\cal O}(\eps^2) + {\cal O}(\eps^{N+1} + \eps (\xi_1 - f(\xi_2))^{N}).
\endeq
Here the ${\cal O}(\eps^2)$--term is independent of $x_2$ and has the properties described in Proposition \ref{prop4.1}. Furthermore, we still have
$\widetilde{p}_1(x,\xi,\eps) = {\cal O}(\eps)$.

\bigskip
\noindent
We shall now simplify the lower order terms $\widetilde{p}_j$, $j\geq 1$, in (\ref{eq4.21}). To that end, let $A_\eps (x,hD;h)$ be a classical analytic elliptic
$h$--pseudodifferential operator of order $0$, with symbol
\begin{equation}
\label{eq4.3}
A_\eps (x,\xi ;h)\sim a_0 (x,\xi,\eps)+h a_1(x,\xi,\eps)+\ldots ,
\end{equation}
depending smoothly on $\eps \in \mathrm{neigh\,}(0,{\bf R})$. Then
$$
A_\eps ^{-1}U_\eps ^{-1}P_\eps U_\eps A_\eps = A_{\eps}^{-1} \widetilde{P}_{\eps} A_{\eps} =:\widehat{P}_{\eps}(x,hD_x;h),
$$
where
\begin{equation}
\label{eq4.4}
\widehat{P}_\eps (x,\xi ;h)\sim \widetilde{p}_0(x,\xi,\eps) + h\widehat{p}_{1}(x,\xi,\eps)+h^2 \widehat{p}_{2}(x,\xi,\eps)+h^3 \widehat{p}_{3}(x,\xi,\eps)+\ldots ,
\end{equation}
with
\begin{equation}
\label{eq4.5}
\widehat{p}_{1}=\widetilde{p}_{1}+\frac{1}{i}a_0^{-1} H_{\widetilde{p}_0} a_0 = \widetilde{p}_{1}+\frac{1}{i}H_{\widetilde{p}_0}b_0,
\end{equation}
if $b_0 =\ln a_0 $, well-defined up to a constant. Thus, looking for $b_0$ in terms of a formal power series in $\eps$ and choosing the terms there suitably, we
can arrange so that
$$
\widehat{p}_{1}(x,\xi,\eps) = \widehat{p}_{1,0}(x_1,\xi) + \eps \widehat{p}_{1,1}(x_1,\xi) + \ldots\, +{\cal O}(\eps^{N+1} + (\xi _1-f(\xi _2))^{N}),
$$
where $\widehat{p}_{1,0}$ is any analytic function satisfying (\ref{eq3.18}), (\ref{eq3.19}) with $q$ replaced by $\widetilde{p}_{1,\eps = 0}$,
and inductively, $\widehat{p}_{1,k}$ is any analytic function satisfying (\ref{eq3.18}), (\ref{eq3.19}), with $q$ replaced by a function depending on the
previously chosen $\widehat{p}_{1,0},...,\widehat{p}_{1,k-1}$.

\bigskip
\noindent
Iterating this procedure, by choosing also the lower order terms in the expansion of $A_\eps$, we get the following result, giving a quantum secular normal form
construction.
\begin{prop}
\label{prop4.2}
Let
$$
P_{\eps} \sim p_0 + hp_1 + \ldots\, \quad (x,\xi)\in {\rm neigh}(\xi = 0, T^*{\bf T}^2),
$$
be such that $p_0(x,\xi,\eps)$ has the properties stated in Proposition {\rm \ref{prop4.1}}. Let $U_\eps$ be an elliptic classical analytic $h$--Fourier
integral operator of order 0, associated to the canonical transformation $\kappa_\eps^{(N)}$ in Proposition {\rm \ref{prop4.1}}. Then we can construct an
elliptic classical analytic $h$--pseudodifferential operator of order 0 with symbol as in {\rm (\ref{eq4.3})}, such that
\begin{equation}
\label{sec.31}
A_\eps^{-1}U_\eps^{-1}P_\eps(x,hD;h)U_\eps A_\eps =\widehat{P}_\eps(x,hD_x;h),
\end{equation}
where $\widehat{P}_{\eps}(x,hD_x;h)$ is of the form
\begin{equation}
\label{sec.32}
\widehat{P}_\eps(x,\xi;h)\sim \widetilde{p}_{0}(x,\xi,\eps)+h \widehat{p}_{1}(x,\xi,\eps)+h^2 \widehat{p}_{2}(x,\xi,\eps)+\ldots .
\end{equation}
Here the leading term $\widetilde{p}_0$ is as in {\rm (\ref{eq4.22})},
$$
\widehat{p}_{0}(x,\xi,\eps) = p(\xi )+i\eps \widetilde{q}_0(x_1,\xi) + {\cal O}(\eps^2) + {\cal O}(\eps^{N+1} + \eps(\xi_1 - f(\xi_2))^N),
$$
with the ${\cal O}(\eps^2)$--term being independent of $x_2$. For $1 \leq k\leq N$, we have
\begin{equation}
\label{sec.33}
\widehat{p}_{k}(x,\xi,\eps) = \widehat{p}_{k,0}(x_1,\xi) + \eps \widehat{p}_{k,1}(x_1,\xi) + \ldots\, +{\cal O}((\eps^{N+1} + (\xi _1-f(\xi _2))^{N}),
\endeq
where $\widehat{p}_{k,\ell}$ is any function satisfying {\rm (\ref{eq3.18})}, {\rm (\ref{eq3.19})}, with $q$ replaced by
$\widetilde{p}_{k,\ell}$; a function that depends on $\widehat{p}_{\widetilde{k},\widetilde{\ell}}$, for all $(\widetilde{k},\widetilde{\ell})$ such that either
$\widetilde{k}<k$, or $\widetilde{k}=k$ and $\widetilde{\ell}<\ell$. In particular, we can choose $\widehat{p}_{k}$ independent of $x_2$ modulo
${\cal O}(\eps^{N+1} + (\xi _1-f(\xi _2))^N)$. We also have
$$
\widehat{p}_1(x,\xi,\eps) = {\cal O}(\eps).
$$
\end{prop}

\section{Harmonic appro\-xima\-tion for non-self\-ad\-jo\-int o\-pe\-ra\-tors}
\setcounter{equation}{0}
In the previous section, we have seen how to eliminate the $x_2$-dependence in the complete symbol of our operator, by means of successive averaging procedures,
when working in a small neighborhood of the rational torus $\Lambda_0 = \{\xi = 0\} \subset T^*{\bf T}^2$. Following Proposition \ref{prop4.2} and neglecting
the remainder terms there, we shall now consider an operator of the form
\begeq
\label{eq5.1}
P_{\eps} = P_{\eps}(x_1,hD_x;h),
\endeq
defined microlocally near $\xi = 0$ in $T^*{\bf T}^2$ and acting on $L^2_{\theta}({\bf T}^2)$, with a complete symbol independent of $x_2$. We assume that the
leading symbol of $P_{\eps}$ is of the form
\begeq
\label{eq5.2}
p_0(x_1,\xi,\eps) = p(\xi) + i\eps \widetilde{q}(x_1,\xi) + {\cal O}(\eps^2),\quad p(\xi) = p(f(\xi_2),\xi_2) + g(\xi)\left(\xi_1 - f(\xi_2)\right)^2,
\endeq
with $\widetilde{q}$ given in (\ref{eq3.20}), and let us recall the assumption (\ref{st.23}) implying that the function
${\bf T}\ni x_1 \mapsto \widetilde{q}(x_1,f(\xi_2),\xi_2)$ has a unique minimum which is non-degenerate, for $\xi_2 \in {\rm neigh}(0,\real)$.
When discussing the spectral analysis of $P_{\eps}$, it is natural, in view of its independence of $x_2$, to take a Fourier series expansion in $x_2$, thereby
reducing the problem, at least formally, to a direct sum of one-dimensional operators
$$
P_{\eps}(x_1,hD_{x_1}, h(j-\theta_2);h),\quad j\in \z, \quad \theta_2 = \frac{k_0(\alpha_2)}{4} + \frac{S_2}{2\pi h},
$$
considered for those $j$ for which $h(j-\theta_2)\in {\rm neigh}(0,\real)$.

\bigskip
\noindent
In what follows, we shall write $\xi_2 = h(j-\theta_2) \in {\rm neigh}(0,\real)$, and concentrate our attention on the one-dimensional operator
\begeq
\label{eq5.3}
P_{\eps}(x_1,hD_{x_1},\xi_2;h),
\endeq
acting on $L^2_{\theta_1}({\bf T})$. Modifying the Floquet conditions on ${\bf T}$, we may replace (\ref{eq5.3}) by the conjugated operator
$$
e^{-if(\xi_2)x_1/h}P_{\eps}(x_1,hD_{x_1},\xi_2;h)e^{if(\xi_2)x_1/h} = P_{\eps}(x_1,f(\xi_2) + hD_{x_1},\xi_2;h).
$$
The full symbol of $P_{\eps}(x_1,hD_{x_1}+f(\xi_2),\xi_2;h)$ is of the form
\begeq
\label{eq5.4}
P_{\eps}(x_1,\xi_1+f(\xi_2),\xi_2;h) = \sum_{j=0}^{\infty} h^j p_{j,\eps}(x_1,\xi).
\endeq
Here
\begeq
\label{eq5.5}
p_{0,\eps}(x_1,\xi)= p(\xi_1+f(\xi_2),\xi_2) + i\eps \widetilde{q}(x_1,\xi_1+f(\xi_2),\xi_2) + {\cal O}(\eps^2),
\endeq
and from Proposition \ref{prop4.2} we recall that
\begeq
\label{eq5.51}
p_{1,\eps}(x_1,\xi) = {\cal O}(\eps).
\endeq
We can then write $p_{1,\eps} = \eps q_{1,\eps}$, $q_{1,\eps} = {\cal O}(1)$.

\bigskip
\noindent
Let us set
\begeq
\label{eq5.6}
\widetilde{h} = \frac{h}{\sqrt{\eps}},
\endeq
and assume that
\begeq
\label{eq5.61}
\widetilde{h} \ll 1.
\endeq
We have
\begin{multline*}
{P}_{\eps}(x_1,\xi_1+f(\xi_2),\xi_2;h) \\
= p(\xi_1+f(\xi_2),\xi_2) + i\eps \widetilde{q}(x_1,\xi_1+f(\xi_2),\xi_2) + {\cal O}(\eps^2)+ h\eps q_{1,\eps}(x_1,\xi) +
\sum_{j=2}^{\infty} h^j p_{j,\eps} \\
= p(\xi_1+f(\xi_2),\xi_2)
+ \eps\left(i\widetilde{q}(x_1,\xi_1+f(\xi_2),\xi_2) + {\cal O}(\eps) + \widetilde{h}\eps^{1/2} q_{1,\eps} + \sum_{j=2}^{\infty} \widetilde{h}^j
\eps^{\frac{j}{2}-1} p_{j,\eps}\right).
\end{multline*}
Here, according to (\ref{eq5.2}),
$$
p(\xi_1+f(\xi_2),\xi_2) = p(f(\xi_2),\xi_2) + g(\xi_1+f(\xi_2),\xi_2) \xi_1^2, \quad g(0)>0.
$$
It follows that on the operator level, we have
\begin{multline}
\label{eq5.7}
{P}_{\eps}(x_1,f(\xi_2) + hD_{x_1},\xi_2;h) = p(f(\xi_2),\xi_2) + g(f(\xi_2)+hD_{x_1},\xi_2) (hD_{x_1})^2 \\
+ i\eps \left(\widetilde{q}(x_1,f(\xi_2) + hD_{x_1},\xi_2) + {\cal O}(\eps + \widetilde{h}\eps^{1/2} + \widetilde{h}^2)\right).
\end{multline}
We shall be interested in computing eigenvalues of the one-dimensional operator $P_{\eps}(x_1,f(\xi_2) + hD_{x_1},\xi_2;h)$ in the region
$$
\abs{{\rm Re}\,z - p(f(\xi_2),\xi_2)} \leq {\cal O}(\eps \widetilde{h}),
$$
and directly from (\ref{eq5.7}), using cut-offs of the form $\chi(\xi_1/\sqrt{\eps})$, we see that the corresponding eigenfunctions are microlocally
concentrated to the region where $\xi_1 = {\cal O}(\sqrt{\eps})$, provided that the smallness condition (\ref{eq5.61}) is strengthened to the following one,
\begeq
\label{eq5.701}
\frac{h}{\sqrt{\eps}} \leq h^{\eta}, \quad \eta > 0.
\endeq
It will then be convenient to perform a rescaling of the cotangent variable, corresponding to a suitable change of the semiclassical parameter. Let us write
$$
hD_{x_1} = \sqrt{\eps}\widetilde{h}D_{x_1},
$$
and if $\xi_1$, $\widetilde{\xi}_1$ denote the cotangent variables corresponding to $hD_{x_1}$ and $\widetilde{h}D_{x_1}$, respectively, we have
$$
\xi_1 = \sqrt{\eps}\widetilde{\xi}_1.
$$
It follows that
\begeq
\label{eq5.71}
\frac{1}{\eps} {P}_{\eps}(x_1,f(\xi_2) + hD_{x_1},\xi_2;h)
\endeq
can be viewed as an $\widetilde{h}$--pseudodifferential operator of the form
\begin{multline}
\label{eq5.8}
\frac{1}{\eps} P_{\eps}(x_1,f(\xi_2) + hD_{x_1},\xi_2;h) =
\frac{p(f(\xi_2),\xi_2)}{\eps} \\ + g(f(\xi_2) + \sqrt{\eps}\widetilde{h}D_{x_1},\xi_2) \left(\widetilde{h}D_{x_1}\right)^2
+ i\widetilde{q}(x_1,f(\xi_2) + \sqrt{\eps}\widetilde{h}D_{x_1},\xi_2) + {\cal O}(\eps) \\
+ \widetilde{h} {\cal O}(\sqrt{\eps}) + {\cal O}(\widetilde{h}^2).
\end{multline}
Ignoring the constant term $p(f(\xi_2),\xi_2)/\eps$ in the right hand side, we recognize here essentially a one-dimensional Schr\"odinger operator with a
purely imaginary potential, and to be precise, we can write
$$
\frac{1}{\eps} P_{\eps}(x_1,f(\xi_2) + hD_{x_1},\xi_2;h) = \frac{p(f(\xi_2),\xi_2)}{\eps} + A(x_1,\widetilde{h}D_{x_1},\xi_2, \sqrt{\eps};\widetilde{h}),
$$
where $A(x_1,\widetilde{h}D_{x_1},\xi_2, \sqrt{\eps};\widetilde{h})$ is a well-behaved $\widetilde{h}$--pseudodifferential operator, depending smoothly on
$\xi_2\in {\rm neigh}(0,\real)$ and $\sqrt{\eps}\geq 0$, with the leading symbol
\begeq
\label{eq5.81}
g(f(\xi_2) + \sqrt{\eps}\xi_1,\xi_2)\xi_1^2 + i\widetilde{q}(x_1,f(\xi_2) + \sqrt{\eps}\xi_1,\xi_2) + {\cal O}(\eps),
\endeq
and with a subprincipal symbol which is ${\cal O}(\sqrt{\eps})$. Here we have dropped the tilde from the notation for the cotangent variable corresponding to
$\widetilde{h}D_{x_1}$, and let us also recall that the operator (\ref{eq5.8}) is to be considered microlocally in the region where $\xi_1 = {\cal O}(1)$.
The function $g$ in (\ref{eq5.81}) satisfies $g>0$.

\medskip
\noindent
{\it Remark}. If (\ref{eq5.51}) is no longer assumed, we can write, assuming that $h/\eps \ll 1$,
\begin{multline*}
{P}_{\eps}(x_1,\xi_1+f(\xi_2),\xi_2;h) = p(\xi_1+f(\xi_2),\xi_2) + i\eps \widetilde{q}(x_1,\xi_1+f(\xi_2),\xi_2) \\
+ {\cal O}(\eps^2)+ h p_{1,\eps}(x_1,\xi) +
\sum_{j=2}^{\infty} h^j p_{j,\eps}(x_1,\xi) \\
= p(\xi_1+f(\xi_2),\xi_2)
+ i\eps\left(\widetilde{q}(x_1,\xi_1+f(\xi_2),\xi_2) + {\cal O}\left(\eps + \frac{h}{\eps}\right) + \sum_{j=2}^{\infty} \widetilde{h}^j \eps^{j/2-1}
p_{j,\eps}\right),
\end{multline*}
and we can then view $h/\eps$ as an additional small parameter. As will be seen in Section 6, some pseudospectral considerations will force us to assume that
$\eps/h$ should not be too large, and for that reason, in this work we make the assumption (\ref{st.24}), leading to (\ref{eq5.51}).

\bigskip
\noindent
The discussion pursued in this section so far indicates that the spectral analysis of the original operator $P_{\eps}$ should reduce to
that for a family of $\widetilde{h}$--pseudodifferential operators on ${\bf T}$, with leading symbols of the form (\ref{eq5.81}). Letting $\eps = 0$
in (\ref{eq5.81}) for a while and suppressing the parameter $\xi_2$ altogether, we shall now pause to make a digression, in order to recall
semiclassical asymptotics for the low lying eigenvalues of non-selfadjoint $h$--pseudodifferential operators with double characteristics. In doing so, we shall
follow the analysis of~\cite{HiPrSt}, which in turn follows~\cite{HeSjSt} closely. Let us also remark that in the present one-dimensional case, the quadratic
approximations along the double characteristics are elliptic and consequently, our discussion is considerably simplified, when compared
with~\cite{HiPrSt},~\cite{HeSjSt}.

\bigskip
\noindent
Let $P_0(x,hD_x;h): C^{\infty}({\bf T}) \rightarrow C^{\infty}({\bf T})$ be such that $P_0(x,hD_x;h) \in {\rm Op}_h^w(S(\langle{\xi\rangle}^2))$, and assume that
the semiclassical leading symbol of $P_0$ is of the form
\begeq
\label{eq5.91}
p_0(x,\xi) = \xi^2 + i V(x),
\endeq
where $V\in C^{\infty}({\bf T};\real)$. Assume also, for simplicity, that the subprincipal symbol of $P_0$ vanishes.
Let us assume that if $a = {\rm min}\, V$ then $V^{-1}(a) = \{x_0\}$ with $V''(x_0)>0$. We are interested in the eigenvalues of $P_0$ in an open disc
$\{z\in \comp; \abs{z - ia} < Ch\}$, for some $C>0$ fixed and all $h>0$ small enough, and to that end we consider the operator
\begeq
\label{eq5.10}
P(x,hD_x;h) = (1-i)(P_0(x,hD_x;h) - ia),
\endeq
whose leading symbol $p(x,\xi) = (1-i)(p_0(x,\xi) -ia)$ is such that
$$
{\rm Re}\, p(x,\xi) = \xi^2 + V(x) -a \geq 0
$$
is elliptic for large $\xi$ and vanishes precisely at the point $(x_0,0)\in T^*{\bf T}$. In a neighborhood of $(x_0,0)$ we have
\begeq
\label{eq5.11}
p(x+x_0,\xi) = q(x,\xi) + {\cal O}((x,\xi)^3), \quad (x,\xi)\rightarrow (0,0),
\endeq
where $q$ is a quadratic form, such that $\Re q > 0$. When determining the eigenvalues of $P(x,hD_x;h)$ in an ${\cal O}(h)$--neighborhood of $0$,
naturally only the behavior of the operator in a small neighborhood of $(x_0,0)$ matters, and by composing $p$ with an inverse of the translation
\begeq
\label{eq5.12}
\kappa: {\rm neigh}((x_0,0),T^*{\bf T}) \rightarrow {\rm neigh}((0,0),T^*\real), \quad \kappa((x_0,0)) = (0,0),
\endeq
we obtain an $h$--pseudodifferential operator
\begeq
\label{eq5.15}
P(x,\xi; h) \sim \sum_{j=0}^{\infty} h^j p_j(x,\xi),\quad p_1 = 0,
\endeq
defined microlocally near $(0,0)\in T^*\real$, such that the leading symbol $p_0 = p$ satisfies
\begeq
\label{eq5.16}
p(x,\xi) = q(x,\xi) + {\cal O}((x,\xi)^3),
\endeq
where $q$ is quadratic with
\begeq
\label{eq5.16.1}
{\rm Re}\, q>0.
\endeq
We extend $P(x,\xi;h)$ to be globally defined on $\real^2$ as an element of the symbol class $S(1)$, such that
\begeq
\label{eq5.17}
{\rm Re}\, p(x,\xi) \geq 0, \quad \left({\rm Re}\, p\right)^{-1}(0) = \{(0,0)\},
\endeq
and such that
\begeq
\label{eq5.18}
{\rm Re}\, p(x,\xi) \geq \frac{1}{C},\quad \abs{(x,\xi)} \geq C > 0.
\endeq
An application of Theorem 1.1 of~\cite{HiPrSt} allows us to conclude that the following result holds, which we state directly for the operator $P_0(x,hD_x;h)$.
See also~\cite{Hi04} for related results in the analytic case.

\begin{theo}
\label{theo5.1}
Let the operator $P_0(x,hD_x;h): C^{\infty}({\bf T})\rightarrow C^{\infty}({\bf T})$ have the principal symbol of the form {\rm (\ref{eq5.91})}, a vanishing
subprincipal symbol, and let us assume that if $a = {\rm min}\,V$ then $V^{-1}(a) = \{x_0\}$ with $b:= V''(x_0)>0$. Let $C>0$. Then there exists $h_0>0$ such that
for all $0< h \leq h_0$, the spectrum of the operator $P_0(x,hD_x;h)$ in the open disc in the complex plane $D(ia,Ch)$ is given by the simple eigenvalues of
the form,
\begeq
\label{eq5.19}
z_{k} \sim ia + h \left(\lambda_{k,0} + h \lambda_{k,1} + h^{2} \lambda_{k,2} +\ldots\right).
\endeq
Here $\lambda_{k,0}$ are the eigenvalues in $D(0,C)$ of the elliptic quadratic operator
$$
q^w(x,D_x) = D_x^2 + i\frac{b}{2} x^2,
$$
acting on $L^2(\real)$, which are given by
$$
\lambda_{k,0} = \left(\frac{b}{2}\right)^{1/2}e^{i\pi/4}(2k+1),\quad k\in \nat,\quad k = {\cal O}(1).
$$
\end{theo}

\medskip
\noindent
{\it Remark}. Theorem \ref{theo5.1} continues to be valid when the operator $P_0(x,hD_x;h)$ acts on an $L^2$--space of Floquet periodic functions on ${\bf T}$
and indeed, the eigenvalues described in this result do not depend on the Floquet conditions, modulo ${\cal O}(h^{\infty})$.

\bigskip
\noindent
Coming back to the operator in (\ref{eq5.8}), with the leading symbol (\ref{eq5.81}), we shall next have to extend the result of Theorem \ref{theo5.1} to the parameter dependent case, and to this end it will be convenient to
recall briefly the main steps in the proof of Theorem \ref{theo5.1}. Let $P = P(x,hD_x;h)$ be an $h$--pseudo\-dif\-ferential ope\-rator on $\real$ satis\-fying
(\ref{eq5.15}) -- (\ref{eq5.18}). Following~\cite{HiPrSt}, let us recall that the proof of Theorem \ref{theo5.1} proceeds by
constructing a well-posed Grushin problem for the operator $P$, of the form
\begeq
\label{eq5.20}
(P-hz)u + R_- u_- = v,\quad R_+ u = v, \quad z\in {\rm neigh}(\lambda_0,\comp),
\endeq
in the space $L^2(\real)\times \comp$. Here $\lambda_0$ is an eigenvalue of $q^w(x,D_x)$ such that $\abs{\lambda_0} < C$. The operators
$R_-: \comp \rightarrow L^2$ and $R_+: L^2 \rightarrow \comp$ are defined as follows,
\begeq
\label{eq5.20.01}
R_- u_- = u_- e,\quad R_+u = (u,f)_{L^2},
\endeq
where $e$ is an eigenfunction of $q^w(x,hD_x)$ corresponding to the eigenvalue $h\lambda_0$, and $f$ is an eigenfunction of the adjoint operator
$\overline{q}(x,hD_x)$, corresponding to the eigenvalue $h\overline{\lambda}_0$.

\medskip
\noindent
The verification of the well-posedness of (\ref{eq5.20}) consists of two steps, both carried out after a metaplectic FBI transform,
\begeq
\label{eq5.20.1}
T: L^2(\real) \rightarrow H_{\Phi_0}(\comp).
\endeq
Here
$$
H_{\Phi_0}(\comp) = {\rm Hol}(\comp) \cap L^2(\comp; e^{-2\Phi_0/h}L(dx)),
$$
and $\Phi_0$ is a suitable strictly subharmonic quadratic form. In the first step, we concentrate on the region $\abs{x}\leq h^{\rho}$, $x\in \comp$,
for some $1/3 < \rho < 1/2$. Arguing as in~\cite{HiPrSt}, we obtain the following a priori estimate for the problem (\ref{eq5.20}), based on the quadratic
approximation of $P$ near the origin --- see formula (3.25) in~\cite{HiPrSt},
\begin{multline}
\label{eq5.21}
\norm{(h+\abs{x}^2)^{1/2}\chi_0\left(\frac{x}{h^{\rho}}\right) u} + h^{-1/2}\abs{u_-} \\
\leq C\norm{(h+\abs{x}^2)^{-1/2}\chi_0\left(\frac{x}{h^{\rho}}\right) v} +
C\norm{(h+\abs{x}^2)^{-1/2}\chi_0\left(\frac{x}{h^{\rho}}\right)(P-Q)u} \\
+ {\cal O}(h^{1/2})\abs{v_+} + C\sqrt{\frac{h}{h^{2\rho}}} \norm{(h+\abs{x}^2)^{1/2}1_K\left(\frac{x}{h^{\rho}}\right)u}.
\end{multline}
Here $u$, $v\in H_{\Phi_0}(\comp)$, the norms are taken in the space $L^2(\comp; e^{-2\Phi_0/h}L(dx))$, and we have also written $P$ for the conjugated operator
$T P T^{-1}$. The function $\chi_0\in C^{\infty}_0(\comp)$ is fixed, with $\chi_0=1$ near $0$, and $K$ is a fixed compact neighborhood of
${\rm supp}(\nabla \chi_0)$, $0\notin K$. Furthermore, $Q = T q^w(x,hD_x) T^{-1}$, and therefore, as explained in~\cite{HiPrSt}, we have
\begeq
\label{eq5.21.1}
\norm{(h+\abs{x}^2)^{-1/2}\chi_0\left(\frac{x}{h^{\rho}}\right) ({P}-{Q})u} = {\cal O}\left(\frac{h^{3\rho}}{h^{1/2}}\right)\norm{u}.
\endeq
Using (\ref{eq5.21}) and (\ref{eq5.21.1}), we obtain
\begin{multline}
\label{eq5.22}
h \norm{\chi_0\left(\frac{x}{h^{\rho}}\right)u}^2 + h^{-1}\abs{u_-}^2 \\
\leq \frac{{\cal O}(1)}{h}\norm{v}^2 + {\cal O}(h^{6\rho-1})\norm{u}^2 + {\cal O}(h)\abs{v_+}^2 +{\cal O}(h)\norm{1_K\left(\frac{x}{h^{\rho}}\right)u}^2.
\end{multline}
Notice that the lower bound $\rho > 1/3$ implies that here $h^{6\rho-1} \ll h$.

\bigskip
\noindent
In the second step of the proof, we consider the exterior region, $\abs{x}\geq h^{\rho}$, and here we use that
$$
{\rm Re}\, p\left(x,\frac{2}{i}\frac{\partial \Phi_0}{\partial x}(x)\right) \geq \frac{h^{2\rho}}{C}, \quad C>0.
$$
Exploiting the sharp G\aa{}rding inequality in the form of a quantization-multiplication formula, as explained in~\cite{HiPrSt}, see also~\cite{Viola12},
we obtain the following exterior a priori estimate for the problem (\ref{eq5.20}),
\begin{multline}
\label{eq5.23}
h^{2\rho} \int \chi\left(\frac{x}{h^{\rho}}\right) \abs{u(x)}^2
e^{-2\Phi_0(x)/h}\, L(d{x}) \leq {\cal O}(1) \norm{v}\, \norm{u}\\
+ {\cal O}(h^{\infty}) \abs{u_-}\,\norm{u}
+ {\cal O}(h)\norm{u}^2.
\end{multline}
Here $\chi\in C^{\infty}_b(\comp; [0,1])$ vanishes near $x=0$ and $\chi = 1 $ for large $x$. Assuming that $1/3 < \rho < 1/2$, the bounds (\ref{eq5.22}) and
(\ref{eq5.23}) can be glued together, and we get the a priori estimate
\begeq
\label{eq5.24}
h\norm{u} + \abs{u_-} \leq {\cal O}(1)\left(\norm{v} + h\abs{v_+}\right),
\endeq
and the consequent well-posedness of the Grushin problem (\ref{eq5.20}). Asymptotic expansions for the eigenvalues of $P$ follow exactly as explained
in~\cite{HeSjSt},~\cite{HiPrSt}. In the present one-dimensional situation, the eigenvalues are simple and only integer powers of $h$ occur in the expansions
(\ref{eq5.19}).

\bigskip
\noindent
Turning the attention to the parameter-dependent case, let $P_{\eps} = P_{\eps}(x,hD_x;h)\in {\rm Op}_h^w(S(1))$, $\eps \geq 0$, be an $h$--pseudodifferential
operator depending smoothly on $\sqrt{\eps}$, such that $P_{\eps = 0} = P$ satisfies (\ref{eq5.15}), (\ref{eq5.16}), (\ref{eq5.16.1}), (\ref{eq5.17}),
(\ref{eq5.18}). In particular, the leading symbol $p_{\eps}$ of $P_{\eps}$ satisfies
\begeq
\label{eq5.25}
p_{\eps}(x,\xi) = p(x,\xi) + {\cal O}(\sqrt{\eps})
\endeq
in the sense of symbols in $S(1)$, and the subprincipal symbol of $P_{\eps}$ is ${\cal O}(\sqrt{\eps})$. Assume also that near $(0,0)$, (\ref{eq5.25}) improves to
\begeq
\label{eq5.26}
p_{\eps}(x,\xi) = p(x,\xi) + {\cal O}\left(\sqrt{\eps}\abs{\xi} + \eps\right),
\endeq
see also (\ref{eq5.81}). We would like to conclude that the Grushin problem (\ref{eq5.20}) with $P$ replaced by $P_{\eps}$ remains well-posed, provided that
$\eps>0$ is not too large, and to that end, we shall simply inspect the two steps above.

\medskip
\noindent
In the region $\abs{x}\leq h^{\rho}$, we argue as above, with $P$ replaced by $P_{\eps}$, and using (\ref{eq5.26}), together with the fact that subprincipal
symbol of $P_{\eps}$ is ${\cal O}(\sqrt{\eps})$, we see that we get an additional term in the right hand side of (\ref{eq5.21}) of the form
\begeq
\label{eq5.27}
\norm{(h+\abs{x}^2)^{-1/2}\chi_0\left(\frac{x}{h^{\rho}}\right) ({P}_{\eps}-{P})u} = {\cal O}\left(\frac{\sqrt{\eps} h^{\rho} + \eps}{h^{1/2}}\right)\norm{u}.
\endeq
Here we also assume that we have chosen the FBI transform in (\ref{eq5.20.1}) so that (\ref{eq5.26}) holds on the transform side.
As for the exterior region $\abs{x}\geq h^{\rho}$, replacing $P$ by $P_{\eps}$, we get an additional term in the right hand side of (\ref{eq5.23}), given by
\begeq
\label{eq5.27.1}
{\cal O}(1)\sqrt{\eps}\norm{u}^2.
\endeq
It follows that to absorb the two extra terms (\ref{eq5.27}), (\ref{eq5.27.1}), we need to meet the following conditions,
$$
\frac{\eps^{1/2} h^{\rho} + \eps}{h^{1/2}} \ll h^{1/2},
$$
and
$$
\sqrt{\eps} \ll h^{2\rho}.
$$
The first condition is satisfied provided that
$$
\eps \ll h^{2-2\rho},
$$
since $\rho < 1/2$, and the second one holds when
$$
\eps \ll h^{4\rho}.
$$
We conclude that the Grushin problem (\ref{eq5.20}) remains well-posed when $P$ is replaced by $P_{\eps}$, provided that
\begeq
\label{eq5.28}
\eps \ll h^{4\rho},
\endeq
since $1/3 < \rho < 1/2$. Combining this observation with the standard perturbation theory for eigenvalues of multiplicity one~\cite{Kato}, we obtain
the following result.
\begin{prop}
\label{prop5.2}
Let $P_{\eps}(x,hD_x;h)$, $\eps \geq 0$, be a smooth function of $\sqrt{\eps}$ with values in ${\rm Op}_h^w(S(1))$, such that when $\eps=0$, we have the properties
{\rm (\ref{eq5.15})}--{\rm (\ref{eq5.18})}. Assume that {\rm (\ref{eq5.26})} holds. Then for $\eps \leq h^{\frac{4}{3} + \eta}$, $\eta >0$, the eigenvalues of
$P_{\eps}(x,hD_x;h)$ in the region $\{z\in \comp; \abs{z} < Ch\}$ are given by the simple eigenvalues of the form
$$
z_k \sim h \left(\lambda_{k,0}(\sqrt{\eps}) + h\lambda_{k,1}(\sqrt{\eps}) + \ldots\,\right), \quad k\in \nat, \quad k = {\cal O}(1),
$$
where $\lambda_{k,j}(\sqrt{\eps})$ are smooth functions of $\sqrt{\eps}\geq 0$, $j\geq 0$, with $\lambda_{k,0}(0)$ being the eigenvalues of the quadratic
operator $q^w(x,D_x)$, described explicitly in Theorem {\rm \ref{theo5.1}}. When $z\in \comp$ is such that $\abs{z} < Ch$ and
${\rm dist}(z,{\rm Spec}\,(P_{\eps})) \geq h/{\cal O}(1)$, we have
\begeq
\label{eq5.28.1}
\left(z-P_{\eps}\right)^{-1} = \frac{{\cal O}(1)}{h}: L^2 \rightarrow L^2.
\endeq
\end{prop}

\bigskip
\noindent
In our considerations, see (\ref{eq5.81}), when applying Proposition \ref{prop5.2}, we should replace the semiclassical parameter $h$ by
$\widetilde{h} = h/\sqrt{\eps}$, which in view of (\ref{eq5.28}) leads to the condition
\begeq
\label{eq5.29}
\eps \ll \widetilde{h}^{4\rho}, \quad 1/3 < \rho < 1/2,
\endeq
so that
$$
\eps \ll h^{\frac{4\rho}{2\rho +1}}.
$$
When $\rho = 1/3$, the power in the right hand side $=4/5$, and it follows that we have the well-posedness of the Grushin problem provided that
\begeq
\label{eq5.30}
\eps \leq {\cal O}(h^{\frac{4}{5} + \eta}), \quad \eta >0.
\endeq

\medskip
\noindent
{\it Remark.} In the proof of Proposition 6.2 above, the presence of the parameter $\sqrt{\eps}$ was treated by a direct perturbation argument, leading to the upper
bound (\ref{eq5.28}). The purpose of this remark is to outline an alternative approach to the parameter-dependent case, leading to sharper bounds on $\eps$.
While sharpening the result of Proposition 6.2 below would not lead to an improvement in Theorem 2.1, which is the main result of this work, we believe that the
alternative approach sketched below may be of some independent interest. Since its precise realization is likely to demand a greater technical investment, the
argument developed in this remark will be quite brief and we hope to be able to develop it further in a future work.

\medskip
\noindent
Let $P_{\eps}(x,\xi;h)$ be a real analytic function of $\eps \in {\rm neigh}(0,\real)$ with values in the space of bounded holomorphic functions in a tubular
neighborhood of $\real^2$, such that as $h\rightarrow 0^+$,
$$
P_{\eps}(\rho;h) \sim p_{\eps}(\rho) + hp_{1,\eps}(\rho) + \ldots, \quad \rho = (x,\xi).
$$
For $\eps = 0$, let us assume that the leading symbol $p := p_0$ is such that ${\rm Re}\, p\geq 0$ is elliptic at infinity, vanishing precisely at $\rho = 0$.
Assume furthermore that we have,
$$
p(\rho) = q(\rho) + {\cal O}(\rho^3), \quad \rho \rightarrow 0,
$$
where $q$ is quadratic with ${\rm Re}\, q$ positive definite. In particular, $\rho = 0$ is a non-degenerate critical point for $p$ and an application of the
implicit function theorem shows that for $\eps$ small, $p_{\eps}$ has a non-degenerate critical point $\rho(\eps)$ in the complex domain, depending analytically
on $\eps$, with $\rho(\eps) = {\cal O}(\eps)$. Passing to the FBI transform side by means of a metaplectic FBI transform $T$, as in (\ref{eq5.20.1}), let us
continue to write $\rho(\eps) = (x(\eps),\xi(\eps)) = {\cal O}(\eps)$ for the image of the critical point $\rho(\eps)$ under the complex linear canonical
transformation $\kappa_T$ associated to $T$.

\medskip
\noindent
We know from~\cite{Sj82} that $\kappa_T(\real^2) = \Lambda_{\Phi_0} = \{(x,(2/i)\partial_x \Phi_0(x));\, x\in \comp\}$, where $\Phi_0$ is the strictly subharmonic
quadratic form introduced in (\ref{eq5.20.1}). We shall now discuss the problem of constructing a weight function $\Phi_{\eps}\in C^{\infty}(\comp)$ such that
\begeq
\label{eq5.30.1}
\Phi_{\eps} = \Phi_0 + {\cal O}(h), \quad \abs{\nabla^2\left(\Phi_{\eps} - \Phi_0\right)}\ll 1,
\endeq
and with $\rho(\eps) \in \Lambda_{\Phi_{\eps}} = \{(x,(2/i)\partial_x \Phi_{\eps}(x));\, x\in \comp\}$. The function $\Phi_{\eps}$ is then strictly subharmonic and
if we set $H_{\Phi_{\eps}}(\comp)  = {\rm Hol}(\comp) \cap L^2(\comp; e^{-2\Phi_{\eps}/h}L(dx))$, then we have $H_{\Phi_{\eps}} = H_{\Phi_0}$, with
uniformly equivalent norms. To get the complete asymptotic expansions of the eigenvalues of $P_{\eps}$ in $D(p_{\eps}(\rho_{\eps}), Ch)$, as in Proposition 6.2,
one should then work with the operator $P_{\eps}$ acting on the space $H_{\Phi_{\eps}}$. We need
$$
\xi(\eps) = \frac{2}{i} \frac{\partial \Phi_{\eps}}{\partial x}(x(\eps)),
$$
and let us notice that
$$
\xi(\eps) - \frac{2}{i} \frac{\partial \Phi_{0}}{\partial x}(x(\eps)) = {\cal O}(\eps).
$$
With $\partial_x \Phi_{\eps}(x(\eps))$ already determined, we try
\begin{multline*}
\Phi_{\eps}(x) = \Phi_0(x) + 2 {\rm Re}\,\biggl((\partial_x \Phi_{\eps}(x(\eps)) - \partial_x \Phi_0(x(\eps)))\cdot (x - x(\eps))\biggr)
\chi \left(\frac{x - x(\eps)}{h^{\alpha}}\right) \\
= \Phi_0(x) + h^{\alpha} \left(\ell_{\eps} \chi\right)\left(\frac{x - x(\eps)}{h^{\alpha}}\right).
\end{multline*}
Here $\chi \in C^{\infty}_0(\comp)$ is a standard cut-off near $0$ and
$$
\ell_{\eps}(y) = 2{\rm Re}\, \biggl((\partial_x \Phi_{\eps}(x(\eps)) - \partial_x \Phi_0(x(\eps)))\cdot y\biggr)
$$
is linear, such that $\ell_{\eps}={\cal O}(\eps)$ as a linear form. Then
$$
\nabla^k \left(\Phi_{\eps} - \Phi_0\right) = {\cal O}(\eps) h^{\alpha - k\alpha},\quad k\geq 0,
$$
and in view of (\ref{eq5.30.1}), we need $\eps h^{\alpha} \leq {\cal O}(h)$, $\eps/h^{\alpha} \ll 1$. We get the conditions $\eps \leq {\cal O}(h^{1-\alpha})$,
$\eps \ll {\cal O}(h^{\alpha})$, and it follows that the optimal choice of $\alpha$ is given by $\alpha = 1/2$. This leads to the condition
$\eps \ll {\cal O}(\sqrt{h})$. In our applications, we should replace $\eps$ by $\sqrt{\eps}$ and $h$ by $\widetilde{h} = h/\sqrt{\eps}$, leading to the condition
$$
\eps \ll \widetilde{h} = \frac{h}{\sqrt{\eps}},
$$
so that we get
\begeq
\label{eq5.30.2}
0\leq \eps \ll {\cal O}(h^{2/3}),
\endeq
which is sharper than (\ref{eq5.30}). One conjectures therefore that the result of Proposition 6.2 extends to this range of $\eps$ and we hope to
return to this observation in a future paper.

\bigskip
\noindent
We shall finish this section by a formal application of Proposition \ref{prop5.2} to the microlocally defined operator $P_{\eps}(x_1,hD_{x_1},\xi_2;h)$ in
(\ref{eq5.3}), acting on $L^2_{\theta_1}({\bf T})$: assume that $\eps > 0$ is such that
$$
\widetilde{h} = \frac{h}{\sqrt{\eps}} \leq h^{\eta},\quad \eta > 0,
$$
and
that (\ref{eq5.30}) holds. It follows that the eigenvalues of $P_{\eps}(x_1,hD_{x_1},\xi_2;h)$ in the region
$$
\abs{z - p(f( \xi_2),\xi_2)-i\eps \langle q\rangle_2(x_1(\xi_2), f(\xi_2),\xi_2)} \leq  {\cal O}(\sqrt{\eps} h)
$$
are given by
\begin{equation}
\label{eq5.31}
\begin{split}
z_{k}=&p(f( \xi_2),\xi_2)+i\eps \langle q\rangle_2(x_1(\xi_2), f(\xi_2),\xi_2)\\
&+\sqrt{\eps}h(\lambda_{k,0} +\lambda _{k,1}\widetilde{h}+\lambda_{k,2} \widetilde{h}^2+\ldots),\quad {\bf N}\ni k\le {\cal O}(1),
\end{split}
\end{equation}
where $\lambda_{k,j} =\lambda_{k,j}(\xi_2,\sqrt{\eps})$, $j\geq 0$, is a smooth function of $\xi_2\in {\rm neigh}(0,\real)$, $\sqrt{\eps}\geq 0$, with
\begin{equation}
\label{eq5.32}
\lambda_{k,0}(\xi_2,0)=e^{i\pi /4}(\partial _{\xi _1}^2p(f(\xi _2),\xi
_2))^{\frac{1}{2}}(\partial _{x_1}^2\langle q\rangle_2(x_1(\xi_2), f(\xi_2),\xi_2))^{\frac{1}{2}}\left(k+\frac{1}{2}\right).
\end{equation}
Here we recall from (\ref{st.23}) that $x_1(\xi_2)\in {\bf T} $ is the unique point of minimum of the function $x_1 \mapsto \langle{q\rangle}_2(x_1, f(\xi_2),\xi_2)$.

\section{Pseudospectral bounds and the global Grushin problem}
\setcounter{equation}{0}
The discussion pursued in the previous section shows that we are able to determine the low-lying eigenvalues of suitable localized one-dimensional operators
$$
P_{\eps}(x_1,hD_{x_1},\xi_2;h)
$$
in (\ref{eq5.3}), occurring in the normal form reduction, provided that the perturbative parameter $\eps$ satisfies
\begeq
\label{eq6.0}
h^{2 -\eta} \leq \eps \leq h^{4/5 + \eta},\quad \eta > 0.
\endeq
The purpose of this section is to complete the proof of Theorem \ref{st1} by constructing a global well-posed Grushin problem for
$P_{\eps}-z$, leading to the description of the eigenvalues in the region described in Theorem \ref{st1}. In doing so, we shall have to strengthen the bounds
in (\ref{eq6.0}), as a consequence of some precise pseudospectral analysis for the family of the one-dimensional non-selfadjoint operators
$P_{\eps}(x_1,hD_{x_1},\xi_2;h)$, with $\xi_2$ playing the role of parameters.

\medskip
\noindent
Our first task is to give a global definition of the $h$--dependent weighted Hilbert space, where the Grushin problem will be studied.
Similarly to~\cite{HiSjVu07}, the weighted space in question will be associated to a globally defined IR-manifold $\Lambda \subset T^*\widetilde{M}$, which is
${\cal O}(\eps)$--close to $T^*M$ and agrees with it outside a compact set. Specifically, the manifold $\Lambda$ will be obtained as an
${\cal O}(\eps^2)$--perturbation of the IR-manifold $\Lambda_{\eps G}$, introduced in (\ref{eq3.32.1}), where the perturbative modification will only take
place in a sufficiently small but fixed neighborhood of the rational torus $\Lambda_0$.

\medskip
\noindent
Let us recall therefore that in Section 4, we have shown that microlocally near the Lagrangian torus $\exp(i\eps H_G)(\Lambda_0)\subset \Lambda_{\eps G}$, the
operator in (\ref{eq3.32.3}) is unitarily equivalent to an analytic $h$--pseudodifferential operator $P_{\varepsilon}$,
defined microlocally near $\xi = 0$ in $T^*{\bf T}^2$ and acting on $L^2_{\theta}({\bf T}^2)$, such that the leading symbol of $P_{\varepsilon}$ is of the form
\begeq
\label{eq6.1}
p_{0}(x,\xi,\eps) = p(\xi) + i\eps \widetilde{q}(x_1,\xi) + {\cal O}(\eps^2) + \eps {\cal O}((\xi_1-f(\xi_2))^N),
\endeq
where $p(\xi)$ is given in (\ref{eq3.2}) and $N\geq 2$ is arbitrarily large but fixed. See also (\ref{eq3.35}) for an illustration of the unitary equivalence
by means of a commutative diagram.

\medskip
\noindent
Let
\begeq
\label{eq6.2}
\kappa_{\eps}^{(N)}: {\rm neigh}(\xi = 0, T^*\widetilde{{\bf T}}^2) \rightarrow {\rm neigh}(\xi = 0, T^*\widetilde{{\bf T}}^2), \quad
\widetilde{{\bf T}}^2 = \comp^2/2\pi \z^2
\endeq
be the holomorphic canonical transformation, introduced in Proposition \ref{prop4.1}. Considering the IR-manifold
$\kappa_{\eps}^{(N)}(T^*{\bf T}^2)\subset T^*\widetilde{{\bf T}}^2$, defined in a
complex neighborhood of $\xi = 0$, we conclude, arguing as in Section 5 in~\cite{HiSjVu07}, that there exists a $C^{\infty}$ strictly plurisubharmonic
function $\Phi_{\eps}(x)$, defined for $x\in \comp^2/2\pi \z^2$, $\abs{{\rm Im}\, x} \leq 1/{{\cal O}(1)}$, such that in the $C^{\infty}$--sense,
$$
\Phi_{\eps}(x) = \Phi_0(x) + {\cal O}(\eps^2),\quad \Phi_0(x) = \frac{1}{2} ({\rm Im}\, x)^2,
$$
and such that the operator
\begeq
\label{eq6.4}
P_{\eps} = {\cal O}(1): T^{-1} H_{\Phi_{\eps}}(\abs{{\rm Im}\,x} < 1/C) \rightarrow T^{-1} H_{\Phi_{\eps}}(\abs{{\rm Im}\,x} < 1/C)
\endeq
is, microlocally near $\kappa_{\eps}^{(N)}({\bf T}^2 \times \{\xi=0\})$, unitarily equivalent to an operator $\widetilde{P}_{\eps}$, given in (\ref{eq4.21}),
(\ref{eq4.22}), acting on $L^2_{\theta}({\bf T}^2)$. Here
$$
T: L^2({\bf T}^2) \rightarrow H_{\Phi_0}(\comp^2 /2\pi \z^2)
$$
is the standard unitary FBI--Bargmann transform on the 2-torus, associated to the quadratic phase function $i(x-y)^2/2$, as discussed in~\cite{HiSj08b}, and we
have written
$$
H_{\Phi_{\eps}}(\Omega) = {\rm Hol}(\Omega) \cap L^2(\Omega, e^{-2\Phi_{\eps}/h} L(dx)),
$$
for $\Omega \subset \comp^2 /2\pi \z^2$ open, including the Floquet periodic versions of the spaces. Let us also point out that the unitary equivalence between the
operators $P_{\eps}$ in (\ref{eq6.4}) and $\widetilde{P}_{\eps}$ is realized by means of a microlocally unitary $h$-Fourier integral operator $U_{\eps}$ in
the complex domain, quantizing the canonical transformation in (\ref{eq6.2}). Similarly to (\ref{eq3.35}), we may illustrate it in a commutative
diagram,
\begeq
\label{eq6.401}
\begin{tikzcd}
P_{\eps}: T^{-1} H_{\Phi_{\eps}}(\abs{{\rm Im}\,x} < 1/C) \arrow{r}  & T^{-1} H_{\Phi_{\eps}}(\abs{{\rm Im}\,x} < 1/C)   \\
\widetilde{P}_{\eps}: L^2_{\theta}({\bf T}^2) \arrow{r}  \arrow[swap]{u}{U_{\eps}}& L^2_{\theta}({\bf T}^2) \arrow{u}{U_{\eps}}
\end{tikzcd}
\endeq

\medskip
\noindent
In particular, according to Proposition \ref{prop4.1}, the leading symbol of $\widetilde{P}_{\eps}$ is independent of $x_2$, modulo
${\cal O}(\eps^{N+1} + \eps (\xi_1 - f(\xi_2))^N)$. The subprincipal symbol of $\widetilde{P}_{\eps}$ is ${\cal O}(\eps)$.

\medskip
\noindent
{\it Remark}. From~\cite{HiSjVu07}, we may recall that writing
$$
\Lambda_{\Phi_{\eps}}: \quad \xi = \frac{2}{i} \frac{\partial \Phi_{\eps}}{\partial x}(x),\quad \abs{{\rm Im}\, x}\leq \frac{1}{{\cal O}(1)},
$$
we have $\Lambda_{\Phi_{\eps}} = \kappa_T \circ \kappa_{\eps}^{(N)} \left(T^*{\bf T}^2\right)$, where the canonical transformation $\kappa_T$ associated to $T$
is given by
$$
T^*\widetilde{{\bf T}}^2 \ni (y,\eta) \mapsto (y-i\eta,\eta) = (x,\xi) \in T^*\widetilde{{\bf T}}^2.
$$
Let us also remark that the writing (\ref{eq6.4}) is somewhat informal, and a precise statement is obtained by considering the action of the conjugated operator
$T P_{\eps} T^{-1}$ on the space $H_{\Phi_{\eps}}(\abs{{\rm Im}\,x} < 1/C)$, see also~\cite{HiSjVu07}.

\bigskip
\noindent
We obtain a globally defined IR-manifold $\Lambda \subset T^*\widetilde{M}$, which is $\eps$-close to $T^*M$ everywhere in the $C^{\infty}$--sense, agrees with
that set away from $p^{-1}(0)$, and in a complex neighborhood of $\Lambda_0$, it is obtained by replacing
$$
\exp(i\eps H_G) \circ \kappa_0^{-1}(T^*{\bf T}^2)
$$
by
\begeq
\label{eq6.5}
\exp(i\eps H_G) \circ \kappa_0^{-1} \circ \kappa_{\eps}^{(N)}(T^*{\bf T}^2),
\endeq
which amounts to an ${\cal O}(\eps^2)$--deformation $\Lambda_{\eps G}$ in a neighborhood of $\Lambda_0$. Here we recall the holomorphic canonical transformation
$\exp(i\eps H_G)$, identifying $\Lambda_{\eps G}$ and $T^*M$ in a neighborhood of $\Lambda_0$, and the real analytic canonical transformation $\kappa_0$ in
(\ref{eq3.1}), given by the action-angle coordinates near $\Lambda_0$. The spectral analysis required in order to compute the extremal
eigenvalues of $P_{\eps}$ in Theorem \ref{st1} will be carried out in the globally defined $h$--dependent Hilbert space $H(\Lambda)$, associated to the IR-manifold
$\Lambda$ by the FBI--Bargmann approach.

\medskip
\noindent
Recalling Proposition 4.2 and taking into account also Proposition \ref{prop4.2}, eliminating the $x_2$--dependence in the normal form by means of a
pseudodifferential conjugation, we may summarize the discussion so far in the following result.

\begin{prop}
\label{prop6.1}
There exists a globally defined smooth IR-manifold $\Lambda \subset T^*\widetilde{M}$ and a $C^{\infty}$-Lagrangian torus $\widehat{\Lambda}_0 \subset \Lambda$,
which is an ${\cal O}(\eps)$--perturbation of the rational torus $\Lambda_0$ in the $C^{\infty}$--sense, such that when
$\rho\in \Lambda$ is away from an $\eps^{\delta}$-\neigh{} of $\widehat{\Lambda}_0$ in $\Lambda$ and
\begeq
\label{eq6.6}
\abs{{\rm Re}\, P_{\eps}(\rho;h)}\leq \frac{\eps^{2\delta}}{C},
\endeq
for $C>0$ large enough, then we have
\begeq
\label{eq6.7}
{\rm Im}\, P_{\eps}(\rho;h) \geq \eps \inf Q_{\infty}(\Lambda_0) + \frac{\eps^{2\delta +1}}{{\cal O}(1)}.
\endeq
Here $0<\delta< 1/2$ is so small that $\eps^{\delta} \gg {\rm max}(h^{1/2},\eps^{1/2})$. The manifold $\Lambda$ is ${\cal O}(\eps)$-close to $T^*M$ and
agrees with it away from a \neigh{} of $p^{-1}(0)\cap T^*M$. We have
$$
P_{\eps}  = {\cal O}(1): H(\Lambda,m)\rightarrow H(\Lambda).
$$
Furthermore, there exists an elliptic $h$--Fourier integral operator with a complex phase
$$
U = {\cal O}(1): H(\Lambda)\rightarrow L^2_{\theta}({\bf T}^2),
$$
such that microlocally near $\widehat{\Lambda}_0$, we have
$$
U P_{\eps}  = \left(P(x_1,hD_x,\eps;h)+R(x,hD_x,\eps;h)\right)U.
$$
Here $P(x_1,hD_x,\eps;h)+R(x,hD_x,\eps;h)$ is defined microlocally near $\xi=0$ in $T^*{\bf T}^2$, the full symbol of
$P(x_1,hD_x,\eps;h)$ is independent of $x_2$, and
\begeq
\label{eq6.701}
R(x,\xi,\eps;h) = {\cal O}(\eps^{N+1} + (\xi_1-f(\xi_2))^{N} + h^{N+1}),\quad f(0) = 0.
\endeq
Here $N$ is arbitrarily large but fixed. The leading symbol of $P(x_1,hD_x,\eps;h)$ is of the form
$$
p(\xi) + i\eps \widetilde{q}(x_1,\xi)+ {\cal O}(\eps^2),
$$
where
\begeq
\label{eq6.71}
p(\xi) = p(f(\xi_2),\xi_2) + g(\xi) (\xi_1 - f(\xi_2))^2,\quad g(0)>0,\quad f(0) = 0,
\endeq
and ${\bf T}\ni x_1 \mapsto \widetilde{q}(x_1,f(\xi_2),\xi_2)$ has a unique minimum, when $\xi_2 \in {\rm neigh}(0,\real)$, which is also non-degenerate.
The subprincipal symbol of $P(x_1,hD_x,\eps;h)$ is ${\cal O}(\eps)$.
\end{prop}

\bigskip
\noindent
Using Proposition \ref{prop6.1}, we shall now discuss a priori estimates for the equation
\begeq
\label{eq6.8}
(P_{\eps} -z)u = v,
\endeq
when $u\in H(\Lambda, m)$, $v\in H(\Lambda)$, and the spectral parameter $z\in \comp$ is confined to the region
\begeq
\label{eq6.9}
\abs{{\rm Re}\, z} \leq \frac{\eps^{2\delta}}{{\cal O}(1)},\quad {\rm Im}\, z \leq \eps \inf Q_{\infty}(\Lambda_0) + {\cal O}(\sqrt{\eps} h).
\endeq
When doing so, following~\cite{HiSjVu07},~\cite{HiSj12}, we shall make use of a suitable partition of unity on the manifold $\Lambda$, defined using
Proposition \ref{prop6.1} and consisting of smooth functions satisfying slightly degenerate symbolic estimates. Indeed, the presence of such slightly exotic
symbols is natural here, as we are dealing with methods based on the techniques of normal forms, introducing error terms vanishing to a high order
along the invariant tori. See also~\cite{Sj92}. When quantizing the corresponding symbols defined on $\Lambda$, in the case when $M = \real^2$,
we follow~\cite{HiSj12} and reduce the quantization procedure to that of Weyl on the standard phase space $T^*\real^2$, by means of a
$C^{\infty}$--canonical transformation
$$
\kappa: {\rm neigh}(p^{-1}(0),T^*\real^2) \rightarrow {\rm neigh}(p^{-1}(0), \Lambda),
$$
such that
$$
\kappa(X) = X + i\eps H_G(x) + {\cal O}(\eps^2),
$$
and the corresponding unitary Fourier integral operator with a complex phase mapping $L^2(\real^2)$ to $H(\Lambda)$. In the case when $M$ is compact,
we use the Toeplitz quantization, following~\cite{Sj96}.

\bigskip
\noindent
Let us consider a smooth partition of unity on the manifold $\Lambda$,
\begeq
\label{eq6.10}
1 = \chi + \psi_1 + \psi_2.
\endeq
Here $\chi \in C^{\infty}_0(\Lambda)$, $\nabla^m \chi = {\cal O}(\eps^{-2\delta m})$, $m\geq 0$, is a cutoff function supported in an $\eps^{\delta}$--neighborhood
of $\widehat{\Lambda}_0$ intersected with the region where $\abs{{\rm Re}\, P_{\eps}} \leq \eps^{2\delta}/C$. Specifically, we shall obtain $\chi$ by choosing a
suitable function $\chi_0 \in C^{\infty}_0(T^*{\bf T}^2)$, $\partial^{\alpha} \chi_0 = {\cal O}(\eps^{-2\delta \abs{\alpha}})$, $\abs{\alpha}\geq 0$, depending on
$\xi$ only, $\chi_0 = \chi_0(\xi)$, and conjugating the operator $\chi_0(hD_x)$ by the microlocal inverse of the operator $U$ in Proposition 7.1. In particular,
we get, using that the subprincipal symbol of $P_{\eps = 0}$ vanishes,
\begeq
\label{eq6.10.1}
[P_{\eps},\chi] = {\cal O}\left(\frac{h^3}{\eps^{6\delta}}\right) + {\cal O}\left(\frac{\eps h}{\eps^{2\delta}}\right): H(\Lambda) \rightarrow H(\Lambda).
\endeq

\medskip
\noindent
The function $0\leq \psi_1 \in C^{\infty}(\Lambda)$ in (\ref{eq6.10}) satisfies
$$
\nabla^m \psi_1 = {\cal O}_m(\eps^{-2\delta m}),\quad m\geq 0,
$$
and is such that
\begeq
\label{eq6.11}
\abs{{\rm Re}\, P_{\eps}(\rho;h)} \geq \frac{\eps^{2\delta}}{C}
\endeq
near the support of $\psi_1$. Finally, $0 \leq \psi_2 \in C^{\infty}_0(\Lambda)$ in (\ref{eq6.10}) is such that
\begeq
\label{eq6.11.1}
\nabla^m \psi_2 = {\cal O}_m(\eps^{-2\delta m}),\quad m\geq 0,
\endeq
and furthermore, $\psi_2$ is supported in a region invariant under the $H_p$--flow, where
\begeq
\label{eq6.11.2}
{\rm Im}\, P_{\eps}(\rho;h) \geq \eps \inf Q_{\infty}(\Lambda_0) + \frac{\eps^{1+2\delta}}{{\cal O}(1)}.
\endeq
We also arrange, as we may, so that $\psi_2$ Poisson commutes with $p$, the leading symbol of $P_{\eps=0}$ acting on $H(\Lambda)$.

\bigskip
\noindent
Let us now return to the equation (\ref{eq6.8}). Assume that $\delta \in (0,1/2)$ is so small that
\begeq
\label{eq6.12}
\eps^{2\delta} \geq h^{\frac{1}{2}-\eta},
\endeq
for some fixed $\eta > 0$. We can then follow the slightly degenerate parametrix construction for $P_{\eps} - z$, near the support of $\psi_1$,
described in detail in Section 4 of~\cite{HiSj12} and obtain that
\begeq
\label{eq6.13}
\norm{\psi_1 u} \leq \frac{{\cal O}(1)}{\eps^{2\delta}}\norm{v} + {\cal O}(h^{\infty})\norm{u}.
\endeq
Here and in what follows the norms are taken in the space $H(\Lambda)$.

\medskip
\noindent
When discussing estimates for $\psi_2 u$, let us notice that ${\rm Im}\, P_{\eps} (\rho;h) = {\cal O}(\eps)$ on $\Lambda$, and near ${\rm supp}\, \psi_2$,
we have, in view of (\ref{eq6.11.2}) and (\ref{eq6.9}),
\begeq
\label{eq6.16}
{\rm Im}\, \left(P_{\eps}(\rho;h) - z\right) \geq \frac{\eps^{1+2\delta}}{{\cal O}(1)} - {\cal O}(\sqrt{\eps}h).
\endeq
Therefore, with a new implicit constant, we get near the support of $\psi_2$,
\begeq
\label{eq6.17}
\frac{1}{\eps}{\rm Im}\, \left(P_{\eps}(\rho;h) - z\right) \geq \frac{\eps^{2\delta}}{{\cal O}(1)},
\endeq
provided that the following lower bound on $\eps$ holds,
\begeq
\label{eq6.18}
h^{2/(1+4\delta)} \ll \eps.
\endeq
The lower bound (\ref{eq6.18}) is of the same form as (\ref{eq6.0}). Using $h/\eps^{4\delta}$ as the natural semiclassical parameter and applying the sharp G\aa{}rding inequality,
we get in view of (\ref{eq6.17}),
\begin{multline}
\label{eq6.18.1}
\frac{1}{\eps}{\rm Im}\, ((P_{\eps}-z)\psi_2 u,\psi_2 u) \geq \left(\frac{\eps^{2\delta}}{{\cal O}(1)} - {\cal O}(1) \frac{h}{\eps^{4\delta}}\right)
\norm{\psi_2 u}^2 - {\cal O}(h^{\infty})\norm{u}^2 \\
\geq \frac{\eps^{2\delta}}{{\cal O}(1)}\norm{\psi_2 u}^2 - {\cal O}(h^{\infty})\norm{u}^2,
\end{multline}
provided that we strengthen (\ref{eq6.12}) by assuming that
\begeq
\label{eq6.18.11}
\frac{h}{\eps^{6\delta}} \leq h^{\eta}, \quad \eta > 0.
\endeq
It follows from (\ref{eq6.18.1}) that
\begeq
\label{eq6.18.12}
\frac{\eps^{2\delta +1}}{{\cal O}(1)}\norm{\psi_2 u}^2 \leq {\cal O}(1)\norm{v}\,\norm{\psi_2 u} + {\rm Im}\, ([P_{\eps},\psi_2]\widetilde{\psi}_2 u,\psi_2 u)
+ {\cal O}(h^{\infty})\norm{u}^2.
\endeq
Here $\widetilde{\psi}_2\in C^{\infty}_0(\Lambda)$ has the same properties as $\psi_2$ and is such that $\widetilde{\psi}_2 = 1$ near ${\rm supp}\,(\psi_2)$.
When estimating the commutator $[P_{\eps},\psi_2]$ in (\ref{eq6.18.12}), we get by the Weyl calculus, using
(\ref{eq6.11.1}) together with the fact that the subprincipal symbol of $P_{\eps=0}$ vanishes and $p$ and $\psi_2$ Poisson commute,
\begeq
\label{eq6.18.4}
[P_{\eps},\psi_2] = [P_{\eps=0},\psi_2] + {\cal O}\left(\frac{\eps h}{\eps^{2\delta}}\right) =  {\cal O}\left(\frac{h^3}{\eps^{6\delta}}\right) +
{\cal O}\left(\frac{\eps h}{\eps^{2\delta}}\right) = {\cal O}\left(\frac{\eps h}{\eps^{2\delta}}\right).
\endeq
Here we have also used that $h^2 \ll \eps^{1+4\delta}$, in view of (\ref{eq6.18}). Combining (\ref{eq6.18.12}) and (\ref{eq6.18.4}), we get
\begeq
\frac{\eps^{2\delta +1}}{{\cal O}(1)}\norm{\psi_2 u}^2 \leq {\cal O}(1)\norm{v}\,\norm{\psi_2 u} +
{\cal O}\left(\frac{\eps h}{\eps^{2\delta}}\right)\norm{\widetilde{\psi}_2 u}^2 + {\cal O}(h^{\infty})\norm{u}^2,
\endeq
and therefore,
\begeq
\label{eq6.18.5}
\norm{\psi_2 u}^2 \leq \frac{{\cal O}(1)}{\eps^{4\delta +2}}\norm{v}^2 + {\cal O}\left(\frac{h}{\eps^{4\delta}}\right)
\norm{\widetilde{\psi}_2 u}^2 + {\cal O}(h^{\infty})\norm{u}^2.
\endeq
Combining (\ref{eq6.18.11}), (\ref{eq6.18.5}), and a standard iteration argument, we conclude that
\begeq
\label{eq6.19}
\norm{\psi_2 u} \leq \frac{{\cal O}(1)}{\eps^{1+2\delta}}\norm{v} + {\cal O}(h^{\infty})\norm{u}.
\endeq

\bigskip
\noindent
Using (\ref{eq6.10}), (\ref{eq6.13}), and (\ref{eq6.19}), we obtain the following a priori estimate for the problem (\ref{eq6.8}), (\ref{eq6.9}),
\begeq
\label{eq6.20}
\norm{(1-\chi)u} \leq \frac{{\cal O}(1)}{\eps^{1+2\delta}}\norm{v} + {\cal O}(h^{\infty})\norm{u},
\endeq
which holds provided that $\delta \in (0,1/2)$ and the conditions (\ref{eq6.18}), (\ref{eq6.18.11}) are fulfilled. In the subsequent analysis, we may therefore
concentrate the attention on the region ${\rm supp}\,(\chi)$, for the cutoff function $\chi$ in (\ref{eq6.10}).

\medskip
\noindent
Let us recall that the function $\chi\in C^{\infty}_0(\Lambda)$, $\nabla^m \chi = {\cal O}(\eps^{-2\delta m})$, $m\geq 0$, in (\ref{eq6.10}) is supported in an
$\eps^{\delta}$--neighborhood of $\widehat{\Lambda}_0$ intersected with the region where $\abs{{\rm Re}\, P_{\eps}} \leq \eps^{2\delta}/C$. Writing
$$
(P_{\eps} - z)\chi u = \chi v + [P_{\eps},\chi] u,
$$
and applying the Fourier integral operator $U$ of Proposition 7.1, we get
\begeq
\label{eq6.20.01}
\left(P(x_1,hD_{x},\eps;h) -z\right) U\chi u = U\chi v + U [P_{\eps},\chi] u + Tu.
\endeq
Here, using (\ref{eq6.701}), we see that
\begeq
\label{eq6.20.1}
T = {\cal O}(h^M): H(\Lambda) \rightarrow L^2_{\theta}({\bf T}^2),
\endeq
where $M$ can be taken as large as we wish, provided that the integer $N$ in Proposition 7.1 is taken large enough. Furthermore, as discussed above, we may
arrange so that
$$
U \chi = \chi_0 U + {\cal O}(h^{\infty}): H(\Lambda) \rightarrow H(\Lambda),
$$
where $\chi_0 = \chi_0(hD_x,\eps)$ is of the form
$$
\chi_0(\xi,\eps) = \chi_1\left(\frac{\xi_1}{\eps^{\delta}}\right)\chi_1\left(\frac{\xi_2}{\eps^{2\delta}}\right),
$$
where $\chi_1 \in C^{\infty}_0(\real)$ is a standard cutoff to a neighborhood of $0$. In particular, using (\ref{eq6.71}) we see that the support of $\chi_0$ is
contained in the region where
$$
\abs{\xi} = {\cal O}(\eps^{\delta}),\quad \abs{p(\xi)} \leq {\cal O}(\eps^{2\delta}).
$$

\medskip
\noindent
Modifying the operator $T$ in (\ref{eq6.20.01}) slightly, we get
\begeq
\label{eq6.20.2}
\left(P(x_1,hD_x,\eps;h)-z\right)\chi_1\left(\frac{hD_{x_1}}{\eps^{\delta}}\right)\chi_2\left(\frac{hD_{x_2}}{\eps^{2\delta}}\right) U u =
U\chi v + U [P_{\eps},\chi] u + Tu
\endeq
In the subsequent analysis we shall therefore be working on the cotangent space $T^*{\bf T}^2$, in the region where
\begeq
\label{eq6.211}
\xi_1 = {\cal O}(\eps^{\delta}),
\endeq
while
\begeq
\label{eq6.22}
\xi_2 = {\cal O}(\eps^{2\delta}).
\endeq

\medskip
\noindent
Taking a Fourier series expansion in $x_2$, we get a direct sum decomposition
\begeq
\label{eq6.23}
P(x_1,hD_x,\eps;h) = \bigoplus_{j\in {\bf Z}} P(x_1,hD_{x_1},\xi_2,\eps;h),\quad \xi_2 = h(j-\theta_2),
\endeq
where, according to (\ref{eq6.22}), the summation is restricted only to those $j\in {\bf Z}$  for which $\xi_2 = {\cal O}(\eps^{2\delta})$. We shall consider
the question of inverting the operator
\begeq
\label{eq6.24}
P(x_1,hD_x,\eps;h) - z = \bigoplus_{j} \left(P(x_1,hD_{x_1},\xi_2,\eps;h) - z\right),
\endeq
where, compared to (\ref{eq6.9}), the real part of $z$ will be localized further to the region
\begeq
\label{eq6.25}
\abs{{\rm Re}\, z} \leq \frac{h}{C\sqrt{\eps}},
\endeq
where $C>0$ is large enough but fixed. Since in Proposition 7.1 we have introduced errors that are ${\cal O}(h^{M})$, $M\gg 1$, see (\ref{eq6.20.1}), we would
first like to show that the one-dimensional non-selfadjoint operator
\begeq
\label{eq6.25.1}
P(x_1,hD_{x_1},\xi_2,\eps;h) -z: L^2_{\theta_1}({\bf T}) \rightarrow L^2_{\theta_1}({\bf T})
\endeq
is invertible, microlocally in the region where $\xi_1 = {\cal O}(\eps^{\delta})$, with an inverse of temperate growth in $1/h$, when
$\xi_2 = {\cal O}(\eps^{2\delta})$ is such that $\abs{\xi_2} \geq h/C_1\sqrt{\eps}$, for a suitable fixed $C_1$, satisfying $0 < C_1 < C$. In doing so,
it will be convenient to distinguish two cases, depending on the sign of $\xi_2$.

\medskip
\noindent
{\bf Case 1}. Let us assume first that $\xi_2 = {\cal O}(\eps^{2\delta})$ is such that
\begeq
\label{eq6.26}
\xi_2 \geq \frac{h}{C_1 \sqrt{\eps}}.
\endeq
Then, after a unitary conjugation, we can write, on the level of operators,
\begin{multline}
\label{eq6.27}
e^{-if(\xi_2)x_1/h} P(x_1,hD_{x_1},\xi_2,\eps;h)e^{if(\xi_2)x_1/h} - z \\
= p(f(\xi_2),\xi_2) +  g(f(\xi_2) + hD_{x_1},\xi_2) (hD_{x_1})^2 + i\eps \widetilde{q}(x_1,f(\xi_2) + hD_{x_1},\xi_2) +
{\cal O}(\eps^2) \\
+ h{\cal O}(\eps) + {\cal O}(h^2) -z.
\end{multline}
Here the conjugated operator, acting on the space of Floquet periodic functions $L^2_{\theta_1 + f(\xi_2)}({\bf T})$, is still considered microlocally in the
region where $\xi_1 = {\cal O}(\eps^{\delta})$, since $f(\xi_2) = {\cal O}(\xi_2) = {\cal O}(\eps^{2\delta})$. Recalling that the derivative of the function
$\xi_2 \mapsto p(f(\xi_2),\xi_2)$ is strictly positive near $\xi_2 = 0$, we conclude, using (\ref{eq6.25}), (\ref{eq6.26}), and the positivity of
$g(f(\xi_2) + hD_{x_1},\xi_2)(hD_{x_1})^2$, that the real part of the operator in (\ref{eq6.27}), which is of the form
$$
p(f(\xi_2),\xi_2) +  g(f(\xi_2) + hD_{x_1},\xi_2) (hD_{x_1})^2 -{\rm Re}\,z + {\cal O}(\eps^2 + h^2) + h{\cal O}(\eps),
$$
is $\geq h/(\widetilde{C}\sqrt{\eps})$, for some $\widetilde{C}>0$, and is therefore invertible, microlocally in the region $\xi_1 = {\cal O}(\eps^{\delta})$,
with the norm of the inverse being ${\cal O}(\sqrt{\eps}/h)$. Here we also use that $\eps^2 \ll h/\sqrt{\eps}$, in view of (\ref{eq6.0}).
It is therefore clear that the full operator in (\ref{eq6.27}) is invertible, microlocally in the region
$\xi_1 = {\cal O}(\eps^{\delta})$, with a microlocal inverse of the norm ${\cal O}(\sqrt{\eps}/h)$.

\bigskip
\noindent
{\bf Case 2}. We assume now that $\xi_2 = {\cal O}(\eps^{2\delta})$ is such that
\begeq
\label{eq6.28}
\xi_2 \leq -\frac{h}{C_1\sqrt{\eps}}.
\endeq
Similarly to (\ref{eq6.27}), we write
\begin{multline}
\label{eq6.29}
e^{-if(\xi_2)x_1/h} P(x_1,hD_{x_1},\xi_2,\eps;h)e^{if(\xi_2)x_1/h} - z \\
= g(f(\xi_2) + hD_{x_1},\xi_2) (hD_{x_1})^2 + i\eps \widetilde{q}(x_1,f(\xi_2) + hD_{x_1},\xi_2) +
{\cal O}(\eps^2) \\
+ h{\cal O}(\eps) + {\cal O}(h^2) - w,
\end{multline}
where
\begeq
\label{eq6.29.1}
w = z - p(f(\xi_2),\xi_2)
\endeq
satisfies
\begeq
\label{eq6.30}
{\rm Re}\, w \geq \frac{h}{C_2 \sqrt{\eps}}, \quad {\rm Im}\, w \leq \eps \inf Q_{\infty}(\Lambda_0) + {\cal O}(\sqrt{\eps}h),
\endeq
for a suitable $C_2 > 0$. In view of (\ref{eq6.22}), we have
\begeq
\label{eq6.30.001}
\widetilde{q}(x_1,f(\xi_2) + \xi_1,\xi_2) = \widetilde{q}(x_1,\xi_1,0) + {\cal O}(\eps^{2\delta}),
\endeq
and therefore, on the operator level we obtain that
\begin{multline}
\label{eq6.30.1}
e^{-if(\xi_2)x_1/h} P(x_1,hD_{x_1},\xi_2,\eps;h)e^{if(\xi_2)x_1/h} - z \\
= g(f(\xi_2)+hD_{x_1},\xi_2)(hD_{x_1})^2 + i\eps \widetilde{q}(x_1,hD_{x_1},0) + R -w,
\end{multline}
where
\begeq
\label{eq6.30.11}
R = {\cal O}(\eps^{1+2\delta} + \eps h + \eps^2 + h^2): L^2_{\theta_1 + f(\xi_2)}({\bf T}) \rightarrow L^2_{\theta_1 + f(\xi_2)}({\bf T}).
\endeq
We may also assume that in (\ref{eq6.30.1}), the operator $\widetilde{q}(x_1, hD_{x_1},0)$ is given by the classical $h$--quantization. It follows from
(\ref{eq6.20.2}) that thanks to the presence of the cutoff $\chi_1(hD_x/\eps^{\delta})$, to invert the operator in (\ref{eq6.25.1}), microlocally in the
region where $\xi_1 = {\cal O}(\eps^{\delta})$, we should consider the following equation
\begin{multline}
\label{eq6.30.12}
\left(g(f(\xi_2) + hD_{x_1},\xi_2)(hD_{x_1})^2 + i\eps \widehat{q}(x_1,hD_{x_1}) + R -w_1\right) \chi_1\left(\frac{hD_{x_1}}{\eps^{\delta}}\right) u = v,
\end{multline}
for $u$, $v\in L^2_{\theta_1 + f(\xi_2)}({\bf T})$. Here $w_1 = w -i\eps \inf Q_{\infty}(\Lambda_0)$ and
$$
\widehat{q}(x_1,\xi_1) = \widehat{q}(x_1,0) + k(x_1,\xi_1) \varphi\left(\frac{\xi_1}{\eps^{\delta}}\right),
$$
where
$$
\widehat{q}(x_1,0) = \langle{q\rangle}_2(x_1,0) - \inf Q_{\infty}(\Lambda_0)\geq 0,
$$
$$
k(x_1,\xi_1) = \xi_1 \int_0^1 \left(\partial_{\xi_1}\widetilde{q}\right)(x_1,t\xi_1)\, dt
$$
and $\varphi \in C^{\infty}_0(\real)$ is such that $\varphi = 1$ near ${\rm supp}\,(\chi_1)$. In particular, the function $\widehat{q}(x_1,\xi_1)$ satisfies
the assumptions for the function $\widetilde{V}$ in Proposition A.4 in Appendix.

\medskip
\noindent
Let us set
$$
A(x_1,hD_{x_1}) = g(f(\xi_2) + hD_{x_1},\xi_2)(hD_{x_1})^2 + i\eps \widehat{q}(x_1,hD_{x_1}),
$$
We would like to invert the operator $A(x_1,hD_{x_1}) + R -w_1$, occurring in the left hand side of (\ref{eq6.30.12}) by an application of Proposition A.4 and to
that end, we shall assume that
\begeq
\label{eq6.31}
\eps^{1 + \frac{\delta}{2}} \ll h.
\endeq
Write
\begeq
\label{eq6.31.1}
A(x_1,hD_{x_1}) - w_1 = \eps \left(\frac{A(x_1,hD_{x_1})}{\eps} - w_2\right), \quad w_2 = \frac{w_1}{\eps} = \frac{w}{\eps} - i \inf Q_{\infty}(\Lambda_0).
\endeq
It follows from (\ref{eq6.30}) that the spectral parameter $w_2$ satisfies
\begeq
\label{eq6.32}
{\rm Re}\, w_2 \geq \frac{h}{C_2 \eps^{3/2}}, \quad {\rm Im}\, w_2 \leq {\cal O}(\widetilde{h}).
\endeq
Here we recall that $\widetilde{h} = h/\sqrt{\eps}$. In order to be able to apply Proposition A.4 to (\ref{eq6.31.1}) we finally have to impose the smallness
condition
\begeq
\label{eq6.32.1}
\widetilde{h}\abs{w_2}^{1/2} \ll 1,
\endeq
and using (\ref{eq6.22}), (\ref{eq6.25}), (\ref{eq6.29.1}), and (\ref{eq6.31.1}), we see that (\ref{eq6.32.1}) holds provided that
$$
\frac{h}{\eps} \eps^{\delta} \ll 1.
$$
We shall therefore require that the condition
\begeq
\label{eq6.32.2}
h \ll \eps^{1-\delta}
\endeq
holds. Once the conditions (\ref{eq6.31}) and (\ref{eq6.32.2}) both hold, we are in the position to apply Proposition A.4 to (\ref{eq6.31.1}), obtaining that
\begin{multline}
\label{eq6.32.3}
\left(A(x_1,{h}D_{x_1}) - w_1\right)^{-1} \\
= \eps^{-1} {\cal O}(\widetilde{h}^{-2/3}\abs{w_2}^{-1/3}): L^2_{\theta_1 + f(\xi_2)}({\bf T}) \rightarrow
L^2_{\theta_1 + f(\xi_2)}({\bf T}).
\end{multline}
We get, using that $\abs{w_2}\geq h/(C_2\eps^{3/2})$,
\begin{multline}
\label{eq6.33}
\left(A(x_1,{h}D_{x_1}) - w_1\right)^{-1} \\ = \eps^{-1} {\cal O}(\widetilde{h}^{-1}\eps^{1/3}) = {\cal O}(h^{-1}\eps^{-1/6}):
L^2_{\theta_1 + f(\xi_2)}({\bf T}) \rightarrow L^2_{\theta_1 + f(\xi_2)}({\bf T}).
\end{multline}

\medskip
\noindent
Returning to the equation (\ref{eq6.30.12}), we would like to use a standard Neumann series argument
to invert the operator $A(x_1,hD_{x_1}) + R -w_1$ in the left hand side of (\ref{eq6.30.12}), and according to (\ref{eq6.33}) and
(\ref{eq6.30.11}), we know that this is possible provided that
\begeq
\label{eq6.34}
{h}^{-1}\eps^{-1/6} \left(\eps^{2\delta+1}  + \eps h + \eps^2 + h^2\right) \ll 1,
\endeq
which, in view of (\ref{eq6.0}) and (\ref{eq6.18}), is equivalent to the condition
\begeq
\label{eq6.35}
{h}^{-1} \eps^{\frac{5}{6} + 2 \delta} \ll 1.
\endeq
Comparing the upper bounds (\ref{eq6.35}) and (\ref{eq6.31}), we see that the latter is implied by the former, provided that
\begeq
\label{eq6.35.1}
0 < \delta < 1/9.
\endeq
In what follows, we shall adopt the smallness condition (\ref{eq6.35.1}). We arrive therefore at the following upper bound on $\eps$,
\begeq
\label{eq6.36}
\eps \ll h^{6/(5+12\delta)},
\endeq
which is a strengthening of the upper bound in (\ref{eq6.0}).

\bigskip
\noindent
We shall now also examine the lower bounds on $\eps$ that we have imposed in the course of our argument in this section. To that end, we recall that the
lower bounds have been introduced in (\ref{eq6.18}), (\ref{eq6.18.11}), and (\ref{eq6.32.2}). Comparing first the lower bounds (\ref{eq6.18}) and (\ref{eq6.32.2}),
we see that we have
$$
h^{2/(1 + 4\delta)} \ll h^{1/(1-\delta)}
$$
when (\ref{eq6.35.1}) holds, and our lower bound on $\eps$ becomes
\begeq
\label{eq6.37}
h^{1/(1-\delta)} \ll \eps.
\endeq
We should then check the validity of (\ref{eq6.18.11}) when (\ref{eq6.37}) holds, and to that end we observe that indeed,
$$
\frac{h}{\eps^{6\delta}} \leq \frac{h}{h^{6\delta/1-\delta}}\leq h^{\eta},\quad \eta > 0,
$$
thanks to (\ref{eq6.35.1}).

\bigskip
\noindent
Combining the bounds (\ref{eq6.36}) and (\ref{eq6.37}), we get the permissible range
\begeq
\label{eq6.37.1}
h^{1/(1 - \delta)} \ll \eps \ll h^{6/(5+12\delta)},
\endeq
where $\delta \in (0,1/9)$.
The range in (\ref{eq6.37.1}) is non-empty for $\delta \in (0,1/9)$ precisely when
$$
\frac{1}{1-\delta} > \frac{6}{5 + 12\delta} \Longleftrightarrow \delta > \frac{1}{18}.
$$

\bigskip
\noindent
Let us summarize the discussion above in the following result.
\begin{prop}
\label{prop6.2}
Let us consider the operator
$$
P(x_1,hD_{x},\eps;h) = \bigoplus_{j\in {\bf Z}} P(x_1,hD_{x_1},\xi_2,\eps;h),\quad \xi_2 = h(j-\theta_2) = {\cal O}(\eps^{2\delta}),
$$
microlocally in the region $\xi_1 = {\cal O}(\eps^{\delta})$, $\xi_2 = {\cal O}(\eps^{2\delta})$, where $1/18 < \delta < 1/9$. Assume that the spectral parameter
$z\in \comp$ is such that
\begeq
\label{eq6.37.2}
\abs{{\rm Re}\, z} \leq \frac{h}{C\sqrt{\eps}}, \quad {\rm Im}\, z \leq \eps \inf\, Q_{\infty}(\Lambda_0) + {\cal O}(\sqrt{\eps}h),
\endeq
for some constant $C>0$. Assume furthermore that
\begeq
\label{eq6.37.3}
h^{1/(1-\delta)} \ll \eps \ll h^{6/(5+12\delta)},
\endeq
and let us assume that the quantum numbers $\xi_2 = {\cal O}(\eps^{2\delta})$ satisfy
$$
\abs{\xi_2}\geq \frac{h}{{\cal O}(1)\sqrt{\eps}}.
$$
Then there exists a family of operators
$$
E(\xi_2,\eps;h) = {\cal O}(\eps^{-1/6}h^{-1}): L^2_{\theta_1} \rightarrow L^2_{\theta_1}
$$
such that
$$
\left(E(\xi_2,\eps;h)(P(x_1,hD_{x_1},\xi_2,\eps;h) -z) - 1\right) \chi\left(\frac{hD_{x_1}}{\eps^{\delta}}\right) =
{\cal O}(h^{\infty}): L^2_{\theta_1} \rightarrow L^2_{\theta_1}
$$
for every $\chi \in C^{\infty}_0(\real)$ with support in a sufficiently small but fixed neighborhood of $0$.
\end{prop}

\medskip
\noindent
{\it Remark}. Notice that to reach powers of $h$ that are $<1$ in (\ref{eq6.37.3}), it suffices to take $\delta > 1/12$. To obtain the range in (\ref{eq6.37.3})
that is as large as possible, we should choose $\delta \in (1/18,1/9)$ to be close to $1/9$.

\bigskip
\noindent
In what follows, we continue to assume that the spectral parameter $z\in \comp$ is confined to the region (\ref{eq6.37.2}), and we shall assume that
(\ref{eq6.37.3}) holds, for some $\delta \in (1/18, 1/9)$. It follows therefore from Proposition \ref{prop6.2} that in the orthogonal sum decomposition (\ref{eq6.24}), we can restrict the
attention to the quantum numbers $\xi_2 = h(j-\theta_2)$ such that
\begeq
\label{eq6.39}
\abs{\xi_2} \leq \frac{h}{C_1 \sqrt{\eps}} = \frac{\widetilde{h}}{C_1},\quad C_1 > 0.
\endeq
Using this refined localization in the parameter $\xi_2$, we shall now proceed to show that the spectrum of the operator $P_{\eps}$ in the region
(\ref{eq6.37.2}) is contained in the union of the pairwise disjoint bands of the form
\begeq
\label{eq6.40}
\abs{p(f(\xi_2),\xi_2) - {\rm Re}\, z} \leq C_0 \sqrt{\eps} h, \quad \xi_2 = h(j-\theta_2) = {\cal O}(\widetilde{h}),
\endeq
where $C_0 > 1$ is large enough but fixed. When doing so, we shall proceed similarly to the arguments above, relying upon Proposition A.4 and treating the
parameter $\xi_2$ in a perturbative way.

\medskip
\noindent
Let us assume that $z\in \comp$ satisfies (\ref{eq6.37.2}) and is such that for some sufficiently large fixed $C_0 > 1$, we have
\begeq
\label{eq6.41}
\abs{p(f(\xi_2),\xi_2) - {\rm Re}\, z} \geq C_0 \sqrt{\eps} h,
\endeq
for all $\xi_2 = h(j-\theta_2) = {\cal O}(\widetilde{h})$. Similarly to (\ref{eq6.29}), we write
\begin{multline}
\label{eq6.42}
e^{-if(\xi_2)x_1/h} P(x_1,hD_{x_1},\xi_2,\eps;h)e^{if(\xi_2)x_1/h} - z \\
= g(f(\xi_2) + hD_{x_1},\xi_2) (hD_{x_1})^2 + i\eps \widetilde{q}(x_1,f(\xi_2) + hD_{x_1},\xi_2) +
{\cal O}(\eps^2) \\
+ h{\cal O}(\eps) + {\cal O}(h^2) - w,
\end{multline}
where $w = z - p(f(\xi_2),\xi_2)$ satisfies
\begeq
\label{eq6.43}
\abs{{\rm Re}\, w}\geq C_0 \sqrt{\eps} h,\quad {\rm Im}\, w \leq \eps \inf\, Q_{\infty}(\Lambda_0) +
{\cal O}(\sqrt{\eps}h).
\endeq
Now, in view of (\ref{eq6.39}), we have
$$
\widetilde{q}(x_1,f(\xi_2) + hD_{x_1},\xi_2) = \widetilde{q}(x_1,hD_{x_1},0) + {\cal O}(\widetilde{h}),
$$
and arguing as in the discussion of "Case 2" above, we see that we have to invert the problem
\begin{multline}
\label{eq6.45}
\left(g(f(\xi_2) + hD_{x_1},\xi_2)(hD_{x_1})^2 + i\eps \widehat{q}(x_1,hD_{x_1}) + R -w_1\right) \chi_1\left(\frac{hD_{x_1}}{\eps^{\delta}}\right) u = v,
\end{multline}
where $\widehat{q}(x_1,hD_{x_1})$ satisfies the assumptions in Proposition A.4 and
\begeq
\label{eq6.45.1}
R = {\cal O}(\eps \widetilde{h} + \eps^2 +\eps h + h^2): L^2_{\theta_1 + f(\xi_2)}({\bf T}) \rightarrow L^2_{\theta_1 + f(\xi_2)}({\bf T}).
\endeq
The spectral parameter $w_1$ in (\ref{eq6.45}) satisfies, in view of (\ref{eq6.43}),
\begeq
\label{eq6.46}
\frac{1}{\eps} \abs{{\rm Re}\, w_1}\geq C_0 \widetilde{h},\quad \frac{1}{\eps} {\rm Im}\, w_1 \leq {\cal O}(\widetilde{h}).
\endeq

\medskip
\noindent
An application of Proposition A.4 gives, as before,
\begin{multline}
\label{eq6.46.1}
\left(g(f(\xi_2) + hD_{x_1},\xi_2)(hD_{x_1})^2 + i\eps \widehat{q}(x_1,hD_{x_1}) -w_1\right)^{-1} \\
= \eps^{-1} {\cal O}\left(\widetilde{h}^{-2/3}\abs{w_1}^{-1/3} \eps^{1/3}\right): L^2_{\theta_1 + f(\xi_2)}({\bf T}) \rightarrow L^2_{\theta_1 + f(\xi_2)}({\bf T}),
\end{multline}
and using (\ref{eq6.46}), we see that the bound on the operator norm in (\ref{eq6.46.1}) does not exceed
\begeq
\label{eq6.46.2}
\eps^{-1} {\cal O}(C_0^{-1/3} \widetilde{h}^{-1}).
\endeq
To invert the full operator
\begeq
\label{eq6.46.3}
g(f(\xi_2) + hD_{x_1},\xi_2)(hD_{x_1})^2 + i\eps \widehat{q}(x_1,hD_{x_1}) + R -w_1,
\endeq
in the left hand side of (\ref{eq6.45}), in view of (\ref{eq6.45.1}) and (\ref{eq6.46.2}), we have to check that
\begeq
\label{eq6.48}
\eps^{-1} \widetilde{h}^{-1}C_0^{-1/3}\left(\eps\widetilde{h} + \eps^2 + h^2\right) \ll 1,
\endeq
which is satisfied for $C_0 > 1$ large enough, since clearly, $\eps \ll \widetilde{h}$, in view of (\ref{eq6.37.3}). The bound on the norm of the
inverse of the operator in (\ref{eq6.46.3}) is therefore also given by (\ref{eq6.46.2}).

\medskip
\noindent
Combining Proposition 7.2 with the discussion above, including the estimates (\ref{eq6.46.1}), (\ref{eq6.46.2}), we conclude that if $z\in \comp$ satisfies
(\ref{eq6.37.2}) and is such that (\ref{eq6.41}) holds, then the operator $P(x_1,hD_{x},\eps;h) - z$ is invertible, microlocally in the region where
$\xi_1 = {\cal O}(\eps^{\delta})$, $\xi_2 = {\cal O}(\eps^{2\delta})$, with a microlocal inverse of the norm
\begeq
\label{eq6.48.1}
\left(P(x_1,hD_{x},\eps;h) - z\right)^{-1} = {\cal O}(\eps^{-1/2}h^{-1}): L^2_{\theta} \rightarrow L^2_{\theta}.
\endeq
Coming back to (\ref{eq6.20.2}), we therefore obtain for such $z$'s,
\begeq
\label{eq6.48.2}
\norm{U\chi u} \leq {\cal O}(\eps^{-1/2}h^{-1})\norm{v} + {\cal O}(\eps^{-1/2}h^{-1}){\cal O}\left(\frac{\eps h}{\eps^{2\delta}}\right)\norm{u} +
{\cal O}(h^{M'})\norm{u},
\endeq
where $M' \gg 1$. Here we have also used (\ref{eq6.10.1}). Combining (\ref{eq6.48.2}) and (\ref{eq6.20}), we obtain the following result.
\begin{prop}
\label{prop6.3}
Assume that
\begeq
\label{eq6.49}
h^{1/(1-\delta)} \ll \eps \ll h^{6/(5+12\delta)},
\endeq
for some $\delta \in (1/18, 1/9)$. Then the spectrum of the operator $P_{\eps}: H(\Lambda,m) \rightarrow H(\Lambda)$ in the region
\begeq
\label{eq6.50}
\abs{{\rm Re}\, z} \leq \frac{h}{C\sqrt{\eps}}, \quad {\rm Im}\, z \leq \eps \inf\, Q_{\infty}(\Lambda_0) +
{\cal O}(\sqrt{\eps}h).
\endeq
is contained in the disjoint union of the bands of the form
\begeq
\label{eq6.51}
\abs{p(f(\xi_2),\xi_2) - {\rm Re}\, z} \leq C_0 \sqrt{\eps} h, \quad \xi_2 = h(j-\theta_2) = {\cal O}(\widetilde{h}),
\endeq
where $C_0 > 1$ is large enough but fixed.
\end{prop}

\bigskip
\noindent
We shall finally obtain a precise description of the spectrum of $P_{\eps}$ in the region (\ref{eq6.51}), for a given value of $j\in {\bf Z}$, such that
$\xi_2 = h(j-\theta_2) = {\cal O}(\widetilde{h})$. In doing so, in view of the localization for ${\rm Im}\, z$, we may assume that
\begeq
\label{eq6.52}
\abs{z - p(f(\xi_2),\xi_2) - i \eps \langle{q\rangle}_2(x_1(\xi_2), f(\xi_2),\xi_2)} \leq C_0 \sqrt{\eps} h,
\endeq
where $\xi_2 = h(j-\theta_2) = {\cal O}(\widetilde{h})$, and we then know that only the operator
\begeq
\label{eq6.52.1}
P(x_1,hD_{x_1},\xi_2,\eps;h): L^2_{\theta_1} \rightarrow L^2_{\theta_1}
\endeq
in (\ref{eq6.23}) contributes to the spectrum in this region. Let us introduce the quadratic elliptic operator
\begeq
\label{eq6.53}
Q(t,D_t;\xi_2) = g(f(\xi_2),\xi_2) D_t^2 + \frac{i}{2} \left(\partial_{x_1}^2 \langle{q\rangle}_2 (x_1(\xi_2), f(\xi_2),\xi_2)\right) t^2,
\endeq
and let $e_{k,\xi_2} \in L^2 (\real)$, $k \in \nat$, be eigenfunctions of $Q(t,D_t;\xi_2)$ corresponding to the eigenvalues $\lambda_k(\xi_2)$ given in
(\ref{eq5.32}). Let also $f_{k,\xi_2}$ be eigenfunctions of the adjoint $Q^*(t,D_t;\xi_2)$, corresponding to the eigenvalues $\overline{\lambda_k(\xi_2)}$.
An application of Proposition 6.2 allows us to conclude that if (\ref{eq6.52}) holds and the rescaled spectral parameter
$$
\frac{1}{\sqrt{\eps} h}\left(z - p(f(\xi_2),\xi_2) - i \eps \langle{q\rangle}_2(x_1(\xi_2), f(\xi_2),\xi_2)\right)
$$
avoids a small but fixed neighborhood of the eigenvalues $\lambda_k(\xi_2)$ in the disc $\abs{z} < C_0$, then $z$ is not in the spectrum of the operator in
(\ref{eq6.52.1}), with
$$
\left(P(x_1,hD_{x_1},\eps;h)-z\right)^{-1} = {\cal O}\left(\frac{1}{\eps^{1/2}h}\right): L^2_{\theta_1} \rightarrow L^2_{\theta_1}.
$$
In view of the analysis above, we conclude that then $z\notin {\rm Spec}(P_{\eps})$. It remains therefore for us to discuss
the setup of the global Grushin problem for $P_{\eps}$ when the spectral parameter $z\in \comp$ is such that
\begeq
\label{eq6.53.1}
\frac{1}{\sqrt{\eps}h}\Biggl(z - p(f(\xi_2),\xi_2) - i \eps \langle{q\rangle}_2(x_1(\xi_2), f(\xi_2),\xi_2)\Biggr) \in {\rm neigh}(\lambda_k(\xi_2),\comp),
\endeq
for some $k\in \nat$ with $k = {\cal O}(1)$. Using the notation of Proposition 7.1, let us set
\begeq
\label{eq6.54}
R_+: H(\Lambda) \rightarrow \comp, \quad R_+ u = R_+(\xi_2,k) \left(U\chi u, e_{\xi_2}\right)_{L^2_{\theta_2}},
\endeq
where $e_{\xi_2}(x_2) = e^{i\xi_2 x_2/h}$ and $R_+(\xi_2,k)$ has been introduced in (\ref{eq5.20.01}), using the eigenfunctions $f_{k,\xi_2}$. Define also
\begeq
\label{eq6.55}
R_-: \comp \rightarrow H(\Lambda),\quad R_- u_- = U^{-1}\left(R_-(\xi_2,k)u_- \otimes e_{\xi_2}\right).
\endeq
Here $R_-(\xi_2, k)$ has been introduced in (\ref{eq5.20.01}), and $U^{-1}$ is a microlocal inverse of $U$. Arguing as in Section 6 of~\cite{HiSj1},
we obtain that when (\ref{eq6.53.1}) holds, the Grushin operator
$$
{\cal P}(z)=\begin{pmatrix} (P_{\eps}-z)/\eps  &R_-\\ R_+ &0\end{pmatrix}: H(\Lambda,m)\times \comp \rightarrow H(\Lambda) \times \comp
$$
is invertible, and the corresponding effective Hamiltonian $E_{-+}(z): \comp \rightarrow \comp$ vanishes precisely when $z$ is of the form (\ref{eq5.31}),
(\ref{eq5.32}). This completes the proof of Theorem 2.1.

\section{Numerical illustrations of spectra}
\label{num}
\setcounter{equation}{0}
The purpose of this section is to present the results of numerical computations of the spectra of $P_{\eps}$, in the following situation, which is easily
implemented: let us consider
\begin{equation}
\label{num.1}
\begin{split}
P_\eps &=-h^2\Delta _{x,y}+i\eps q(x,y;hD_x,hD_y ),\\
q(x,y;hD_x,hD_y )&=q_0(x,y)+q_1(x,y)hD_x+q_2(x,y)hD_y,
\end{split}
\end{equation}
on the torus $M={\bf T}_{x,y}^2=({\bf R}/2\pi {\bf Z})^2$. Here $q_0,q_1,q_2$ are real trigonometric polynomials of degree $\le F\in \{1,2,... \}$. We shall consider the
spectrum of this operator near the energy $E_0=1$.

\noindent
\medskip
The general assumptions (\ref{st.6}), (\ref{st.7}), (\ref{st.9}) are fulfilled, the operator $P_{\eps = 0}$ is selfadjoint, and the leading
semiclassical symbol is of the form (\ref{st.10}) with
\begin{equation}\label{num.2}
p=\xi ^2+\eta ^2,\ \ q(x,y;\xi ,\eta )=q_0(x,y)+q_1(x,y)\xi +q_2(x,y)\eta .
\end{equation}
We also have $dp\ne 0$ along $p^{-1}(1)\cap T^*M$.

\medskip
\noindent
The Hamilton flow of $p$ is completely integrable and we have the decomposition (\ref{st.11}), for $p^{-1}(1)$ rather than $p^{-1}(0)$, where
$$
J=\bigcup_{(\xi ,\eta )\in {\bf T}}\Lambda _{\xi ,\eta },\ \ \Lambda _{\xi ,\eta }={\bf T}^2_{x,y}\times \{(\xi ,\eta ) \} .
$$

\medskip
\noindent
We have
$$
q_{\ell}(x,y)=\sum_{|j|,\,|k|\le F}\widehat{q}_\ell (j,k)e^{i(jx+ky)},
$$
where the reality of $q_\ell$ is equivalent to the property,
$$
\widehat{q}_\ell (-j,-k)=\overline{\widehat{q}}_\ell (j,k).
$$
Rather than taking some particular explicit choice of $q$, we generate $\widehat{q}_\ell$ at random by choosing
$$
\widehat{q}_\ell (j,k)=\frac{1}{2}\left(A_\ell (j,k)+\overline{A}_\ell (-j,-k)\right),\ \ A_\ell (j,k)=e^{-\kappa |j-k|}\alpha _{j,k}^\ell,
$$
where $\alpha^\ell_{j,k}\sim {\cal N}(0,1)$ are independent Gaussian random variables. The parameter $\kappa >0$ induces an off-diagonal exponential decay,
corresponding to the assumption that $q(x,y;\xi ,\eta )$ is analytic in $(x,y)$. Then,
$$
q(x,y;\xi ,\eta )=\sum_{(j,k)\in [-F,F]^2}\widehat{q}(j,k;\xi ,\eta )e^{i(jx+ky)},
$$
where
$$
\widehat{q}(j,k;\xi ,\eta )=\widehat{q}_0(j,k)+\widehat{q}_1(j,k)\xi +\widehat{q}_2(j,k)\eta .
$$
Here and below it is understood that $[-F,F]$ is an interval in ${\bf Z}$.

\medskip
\noindent
Let $\Lambda _{\xi ,\eta }\in J$ be a rational torus, so that $(\xi ,\eta )\in {\bf T}$ and $\xi /\eta \in {\bf Q}\cup \{\infty\}$. The
$H_p$-trajectories in $\Lambda _{\xi ,\eta }$ are of the form
$$
\gamma :{\bf R}\ni s\mapsto ((x_0,y_0)+2s(\xi ,\eta ),(\xi ,\eta )).
$$
The restriction of $q$ to such a trajectory
is
\begin{equation}
\label{num.3}
q(\gamma (s);\xi ,\eta )=\sum_{(j,k)\in [-F,F]^2}\widehat{q}(j,k;\xi,\eta )e^{i((x_0,y_0)+2s(\xi ,\eta ))\cdot (j,k)}.
\end{equation}

\medskip
\noindent
For the corresponding limit of the trajectory average,
$$
\langle q\rangle _\gamma =\lim_{T\to \infty}\frac{1}{T}\int_{-T/2}^{T/2}q(\gamma (s);\xi ,\eta )ds,
$$
the terms in (\ref{num.3}) with $(\xi ,\eta )\cdot (j,k)\ne 0$ give a zero contribution, and we get
$$
\langle q\rangle_\gamma =\sum_{(j,k)\in [-F,F]^2\cap (\xi
  ,\eta )^\perp} \widehat{q}(j,k;\xi ,\eta )e^{i(x_0,y_0)\cdot (j,k)}.
$$
If we write $(\xi ,\eta )=(-n,m)/|(m,n)|$ with $(n,m)\in {\bf Z}^2$, $\gcd (n,m)=1$, then ${\bf Z}^2\cap (\xi ,\eta )^\perp = {\bf Z}(m,n)$
and the intersection of this set with $[-F,F]^2$ (viewed as a subset of ${\bf Z}^2$) is equal to
$$
\left\{ \mu (m,n);\ \mu \in {\bf Z},\ |\mu |\le F/\max (|m|,|n|)
  \right\} .
$$
This gives,
\begin{equation}\label{num.4}
\langle q\rangle _\gamma =\sum_{-[F/\max
  (|m|,|n|)]}^{[F/\max (|m|,|n|)]} \widehat{q}(\mu (m,n);\xi ,\eta )e^{i\mu
  t(\gamma )},
\end{equation}
where $t(\gamma )=(x_0,y_0)\cdot (m,n)$ varies in $[0,2\pi [$ and can
take any value in that interval and $[\cdot ]$ denotes the integer part. It follows that
$$
Q_\infty (\Lambda _{\xi ,\eta })=\left\{ \sum_{-[F/\max
  (|m|,|n|)]}^{[F/\max (|m|,|n|)]} \widehat{q}(\mu (m,n);\xi,\eta)e^{i\mu t};\ 0\le t\le 2\pi  \right\} .
$$
When $\max (|m|,|n|)>F$ this interval reduces to the torus average $\widehat{q}(0,0;\xi ,\eta )=\langle q\rangle_{\Lambda _{\xi ,\eta }}$, so we get
non-trivial intervals only for the finitely many values $(m,n)\in
[-F,F]^2$ with $\gcd (m,n)=1$.

\medskip
\noindent
We have written MatLab programs for the production of $q$ and for the calculation of $\langle q\rangle _\Lambda $, $Q_\infty (\Lambda )$, as well
as the supremum and infimum of $q$ over each torus in $J$. For the graphics, we parametrize $J$ by $\mathrm{arg\,}(\xi +i\eta )$ and the figure below shows:
\begin{itemize}
\item \textcolor{blue}{The torus average $\langle q\rangle_\Lambda $,}
\item \textcolor{green}{The torus max and min of $q$,}
\item \textcolor{red}{$Q_\infty (\Lambda )$ for each
    relevant rational torus.}
\end{itemize}
\includegraphics[width=400pt]{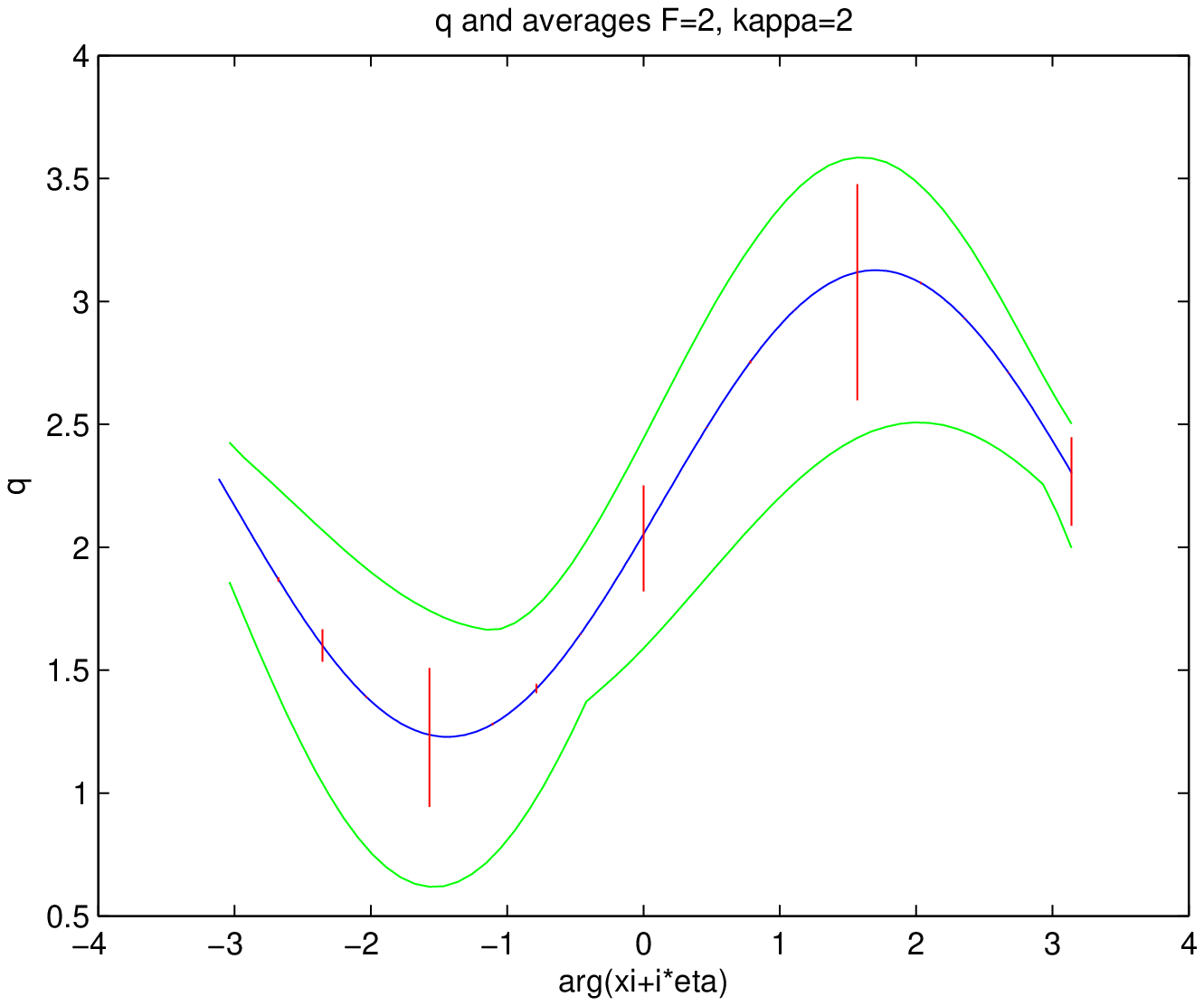}

\medskip
\noindent
By running the simulation several times we get a series of figures where quite a few exhibit the features above. In order to have a
numerical illustration of the main result of this work it is important that some of the vertical segments (corresponding to $Q_\infty
(\Lambda )$ for rational tori) reach above the supremum or below the infimum of the curve of torus averages. A larger $F$ will produce a
richer picture with more vertical segments, but it will also complicate the numerical calculations of the eigenvalues, so we settled for $F=2$ as a
reasonable choice. We also found that $\kappa =2$  produces some -- not too many -- visible vertical segments.

\medskip
\noindent
Once an interesting $q$ has been selected, we compute the spectrum numerically by working on the level of Fourier
coefficients. Thus, if we are interested in the eigenvalues with real parts in $[E_1,E_2]$, where $E_1<E_2$ are close to 1, we work with
Fourier modes $e^{i(jx+ky)}$ for $(j,k)$ in the set ${\cal E}$ of $(j,k)\in {\bf Z}^2$, satisfying $(hj)^2+(hk)^2\in [E_1,E_2]$, i.e.
$$
|(j,k)|\in [\sqrt{E_1}/h, \sqrt{E_2}/h ].
$$
The number $\#{\cal E}$ of such modes is $\approx \pi (E_2-E_1)/h^2$ and we ask MatLab to compute the spectrum of the ${\cal E}\times {\cal
E}$-matrix ${\cal A}_\eps =\left( a_\eps (j,k;\widetilde{j},\widetilde{k}) \right)_{(j,k),\,(\widetilde{j},\widetilde{k})\in {\cal E}}$, given by
\[\begin{split}
a_\eps (j,k;\widetilde{j},\widetilde{k})=&h^2(j^2+k^2)\delta
_{(j,k),(\widetilde{j},\widetilde{k})}\\ &+i\eps
\left( \widehat{q}_0(j-\widetilde{j},k-\widetilde{k})
+\widehat{q}_1(j-\widetilde{j},k-\widetilde{k})h\widetilde{j}
+\widehat{q}_2(j-\widetilde{j},k-\widetilde{k})h\widetilde{k}\right) .
\end{split}
\]
We cannot let $\# {\cal E}$ be larger than a few thousand and still we would like $h$ to be small and the energy shell $[E_1,E_2]$ thick
enough so that the eigenvalues with real part inside are not influenced by boundary effects. In the simulations below we have
chosen the same $q$ as the one in the figure above and we settled for $E_1=0.85$, $E_2=1$, $h=1/100$, leading to $\#{\cal E}\approx
5000$. Since the spectra are of width $\eps$, we rescale the imaginary axis and represent graphically the set of $(\Re z, \Im
z/\eps)$ for $z$ in the spectrum of $P_\eps $. We let $\eps$ take the values $h/2,\, h,\, 2h,\, 4h,\, 8h,\, 16h$, in agreement with Theorem 2.1.\\
\includegraphics[width=400pt]{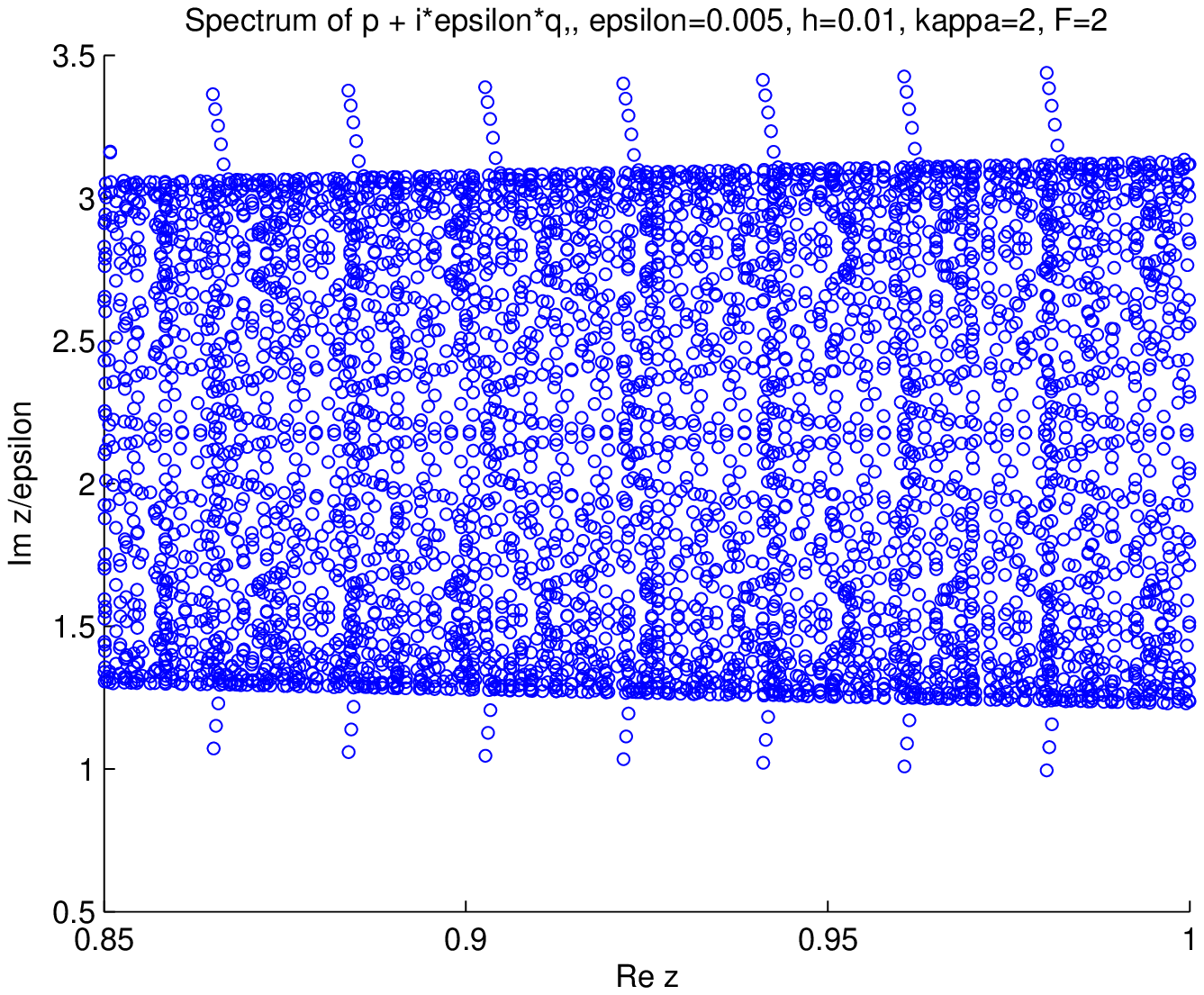}\\
\includegraphics[width=400pt]{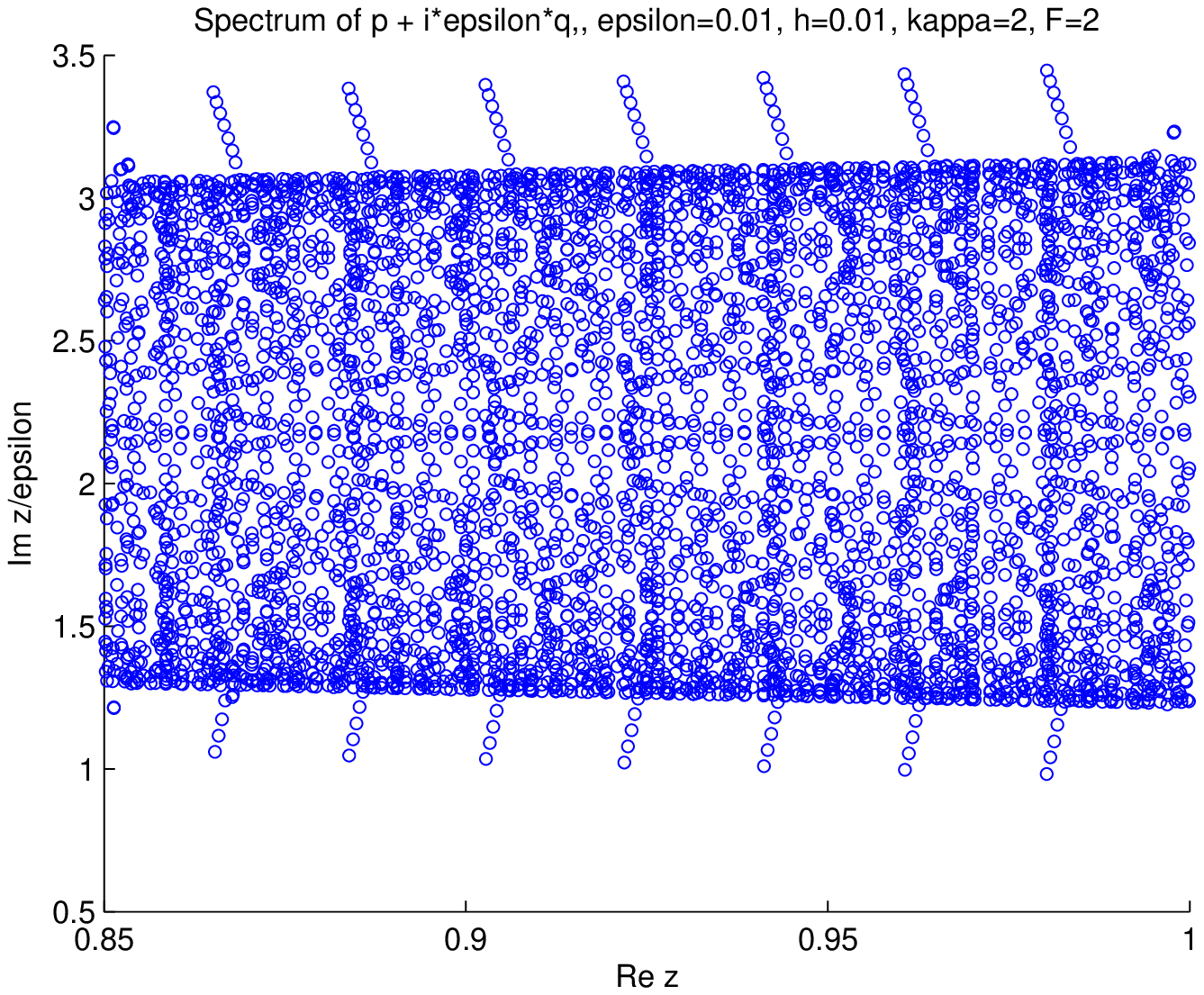}\\
\includegraphics[width=400pt]{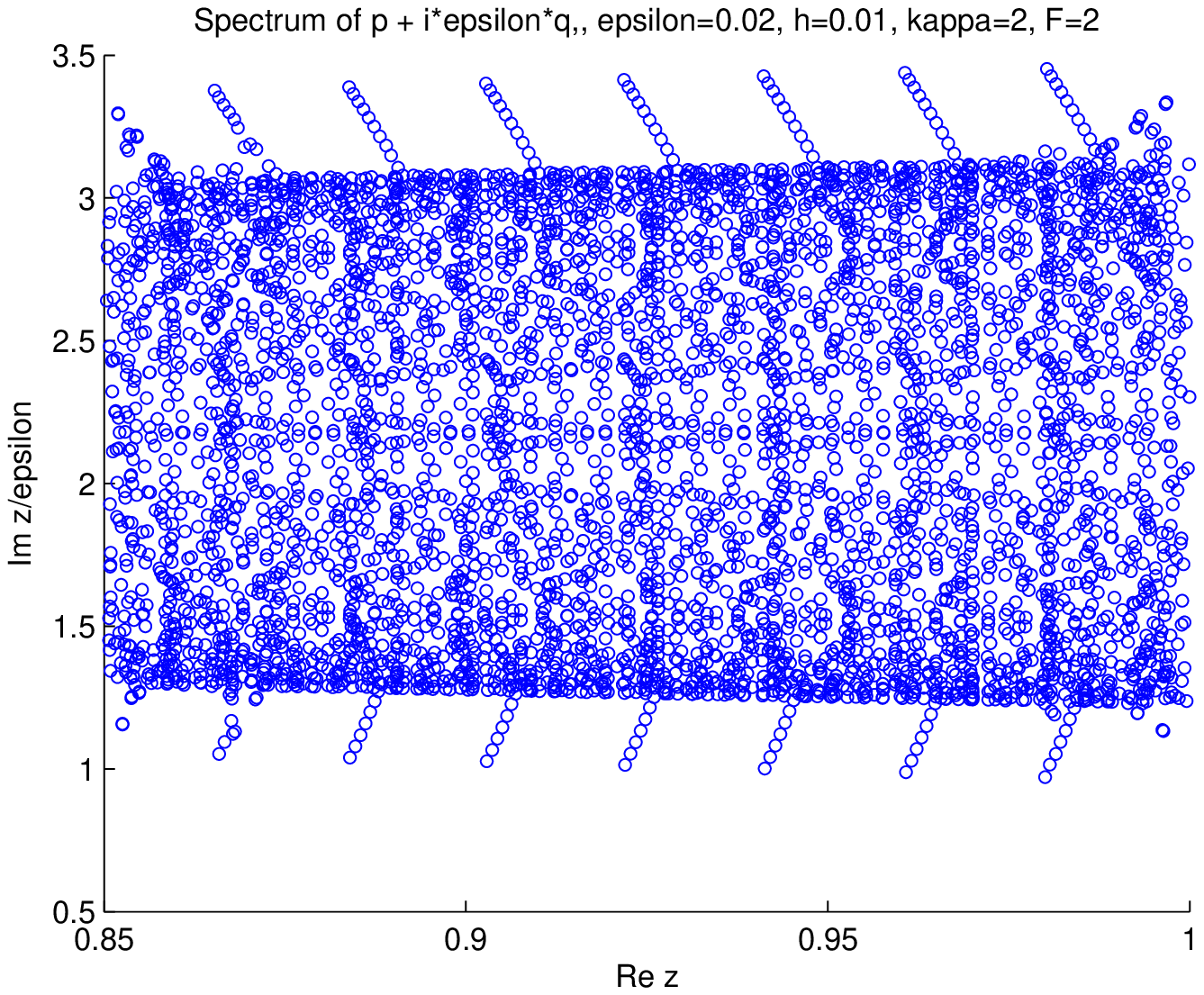}\\
\includegraphics[width=400pt]{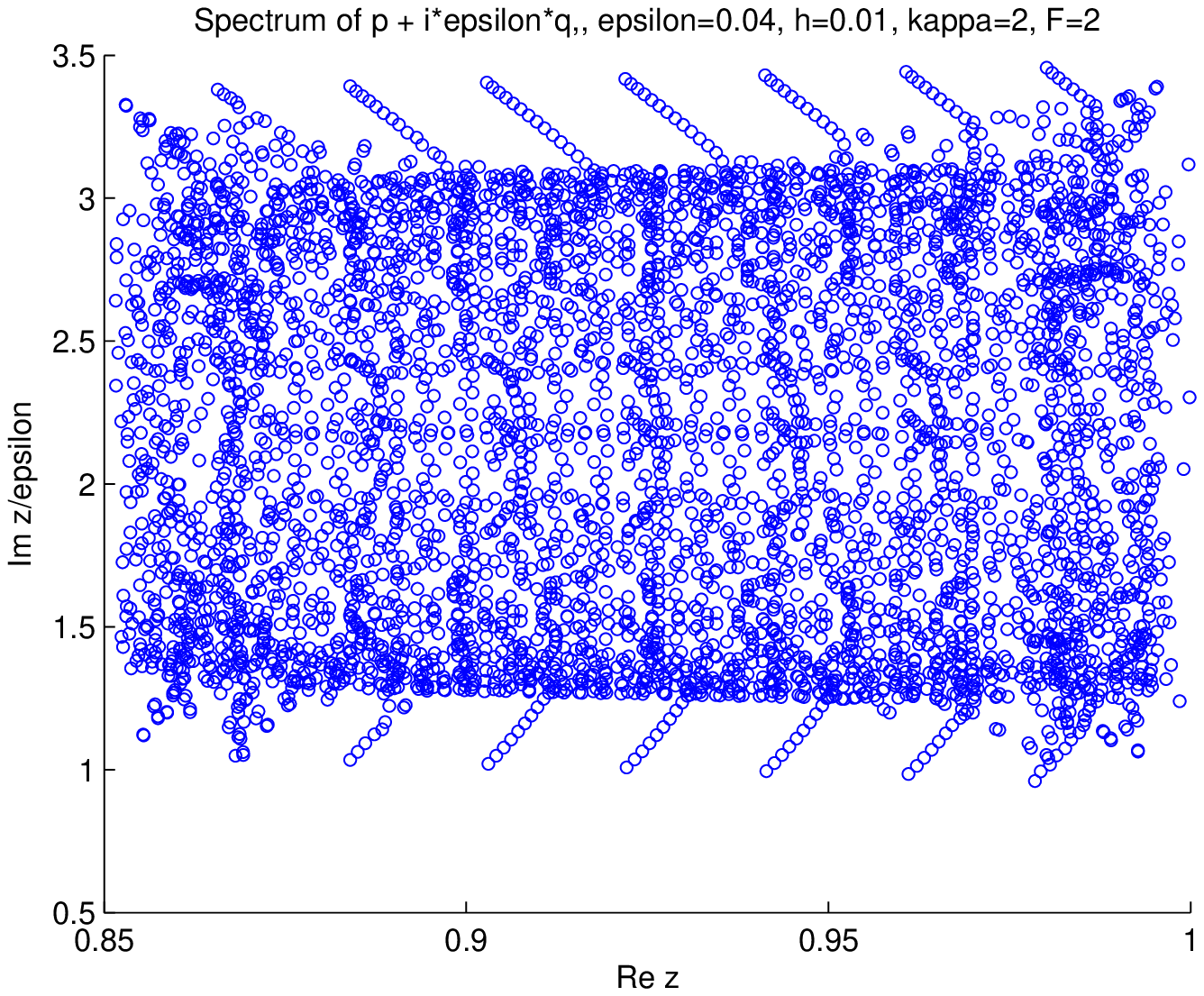}\\
\includegraphics[width=400pt]{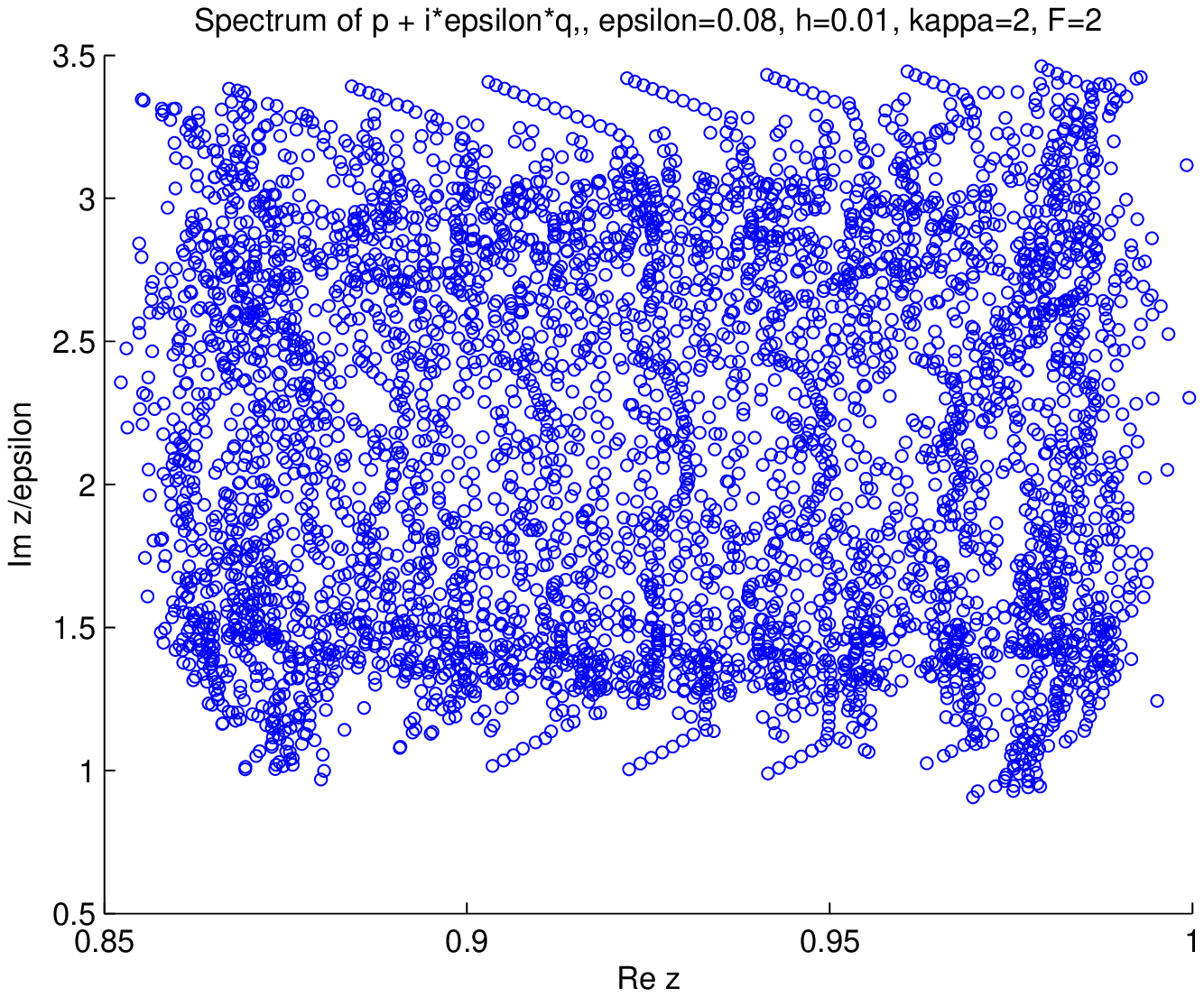}\\
\includegraphics[width=400pt]{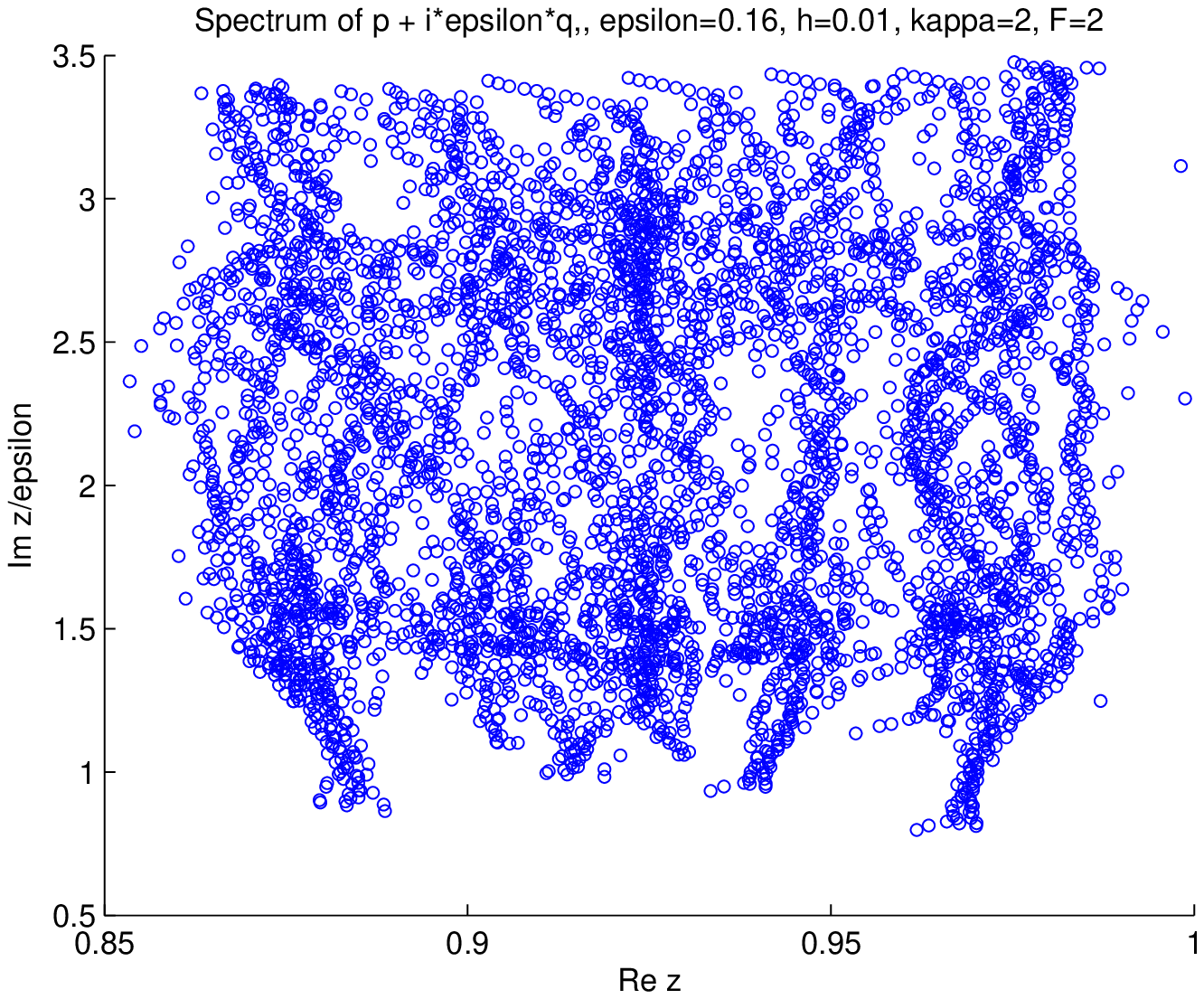}

\medskip
\noindent
These eigenvalues form a kind of a centipede with legs sticking out from the main body. The majority of the eigenvalues are in the body whose position
corresponds nicely to the range of the curve of torus averages on the first picture. The legs reach out to the supremum of the highest
and the infimum of the lowest vertical segments corresponding to $Q_\infty (\Lambda )$ for rational tori $\Lambda$.

\medskip
\noindent
Undoing the scaling of ${\rm Im}\,z$, the inclination of the legs should theoretically be close to 45 degrees and by measuring this for one of
the legs on one of the figures we found some (but not excellent) agreement.

\medskip
\noindent
The main result of this work, Theorem 2.1, describes the individual eigenvalues near the extremities of the legs in terms of rational tori. A
mathematical treatment of the eigenvalues further inside seems more difficult because of the pseudospectral effects that are likely to get stronger there.

\medskip
\noindent
By staring at the pictures directly from the pdf file and creating a movie by switching the pages, we see that most of the (rescaled)
eigenvalues remain fixed while those in the legs and some others move. The fixed ones probably correspond to irrational tori and the moving ones to tori that
are rational.

\begin{appendix}
\section{Subelliptic estimates for Schr\"odinger type ope\-ra\-tors}
\label{Schr}
\setcounter{equation}{0}
The purpose of this appendix is to establish suitable resolvent estimates for some non-selfadjoint operators of Schr\"odinger type, instrumental in the
pseudospectral analysis of Section 7. While in the considerations of Section 7, we are concerned with operators on the one-dimensional torus, it will be convenient
to analyze the case of $\real$ first. See also~\cite{LeLe}.

\medskip
\noindent
Let
\begeq
\label{app1}
P_0 = g(hD_x)(hD_x)^2 + i V(x),\quad V\in C^{\infty}(\real;\real).
\endeq
Assume that the function $g\in C^{\infty}(\real;\real)$ is such that
\begeq
\label{app1.001}
g -1 \in C^{\infty}_0(\real)
\endeq
with
\begeq
\label{app1.01}
g\geq 1,\quad \abs{\xi g'(\xi)} \ll 1.
\endeq
We may notice that the conditions (\ref{app1.01}) are invariant under the scaling $g(\xi)\rightarrow g(\lambda\xi)$, $\lambda > 0$.
We also assume that the potential $V$ is such that
\begeq
\label{app1.1}
0\leq V,\quad \partial_x^j V \in L^{\infty}(\real), \quad j \geq 2,
\endeq
and let us make the ellipticity assumption,
\begeq
\label{app2}
V(x) \geq \frac{x^2}{C},\quad \abs{x}\geq C,
\endeq
for some constant $C>0$. The semiclassical symbol of $P_0$, $p_0(x,\xi) = g(\xi)\xi^2 + iV(x)$ satisfies $p_0\in S(m)$, where $m(x,\xi) = 1 + x^2 + \xi^2$, and when
equipped with the domain $H(m)$, the natural Sobolev space associated to the order function $m$, the operator $P_0$ becomes closed densely defined on
$L^2(\real)$. The spectrum of $P_0$ is discrete and we notice that
\begeq
\label{app3}
{\rm Spec}\,(P_0) \subset p_0(\real^{2}) = \{z\in \comp; {\rm arg}\,z \in [0,\pi/2]\}.
\endeq

\bigskip
\noindent
Let us make the basic assumption that
\begeq
\label{app3.1}
V^{-1}(0) = \{0\} \subset \real
\endeq
and that
\begeq
\label{app3.2}
V''(0)>0.
\endeq
We are interested in estimates for the resolvent of $P_0$,
$$
\left(P_0 - z\right)^{-1}: L^2(\real) \rightarrow L^2(\real),
$$
when the spectral parameter $z\in \comp$ is such that $0\leq {\rm Im}\, z \leq {\cal O}(h)$ and $\abs{z}\gg h$. When establishing those, we shall combine
some of the results and techniques of~\cite{HeSjSt} and~\cite{HiPrSt}.

\bigskip
\noindent
In what follows, rather than working with $P_0$, it will be convenient to consider the operator
\begeq
\label{app4}
P = -iP_0 = V(x) - i g(hD_x)(hD_x)^2
\endeq
with the symbol $p = p_1 + ip_2$, where
$$
p_1(x,\xi) = V(x) = V_0(x) + {\cal O}(x^3),\quad V_0(x) = \frac{1}{2} V''(0)x^2 > 0,
$$
and
$$
p_2(x,\xi) = -g(\xi)\xi^2.
$$
It follows from (\ref{app1.01}) that
$$
\abs{\partial_{\xi}p_2} \sim \abs{\xi},
$$
and therefore we obtain the fundamental property
\begeq
\label{app6}
V_0(x) + H_{p_2}^2 V_0(x,\xi)\sim \abs{(x,\xi)}^2, \quad (x,\xi) \in \real^2.
\endeq
Following~\cite{HeSjSt}, let us set, writing $X = (x,\xi)\in \real^2$,
\begeq
\label{app7}
G_0(X;h) = h^{2/3} \frac{H_{p_2}V_0}{\abs{X}^{4/3}} \psi\left(M \frac{V_0(x)}{(h\abs{X})^{2/3}}\right),\quad \abs{X}\geq h^{1/2}.
\endeq
Here $M\geq 1$ is a constant to be taken large enough and $\psi\in C^{\infty}_0(\real; [0,1])$ is such that ${\rm supp}\,(\psi) \subset (-2,2)$ and $\psi = 1$
on $[-1,1]$. It is then straightforward to verify that in the region where $\abs{X}\geq h^{1/2}$, we have
\begeq
\label{app7.1}
G_0 = {\cal O}(h),\quad H_{G_0} = {\cal O}(1)\frac{h^{2/3}}{\abs{X}^{1/3}} = {\cal O}(h^{1/2}),
\endeq
and
\begeq
\label{app7.2}
\partial^2 G_0 = {\cal O}(1)\left(\frac{h^{2/3}}{\abs{X}^{4/3}} + \frac{h^{1/3}}{\abs{X}^{2/3}} + \frac{h^{2/3}}{\abs{X}^{1/3}}\right) = {\cal O}(1).
\endeq
Indeed, the validity of (\ref{app7.1}) and (\ref{app7.2}) follows easily once we observe that in the region where $0\leq M V_0(x) \leq 2 (h\abs{X})^{2/3}$,
and $\abs{X}\geq h^{1/2}$, we have
$$
H_{p_2}V_0 = {\cal O}(\abs{X} h^{1/3}\abs{X}^{1/3}),\quad \nabla(H_{p_2}V_0) = {\cal O}(\abs{X}),\quad \nabla^2 (H_{p_2}V_0) = {\cal O}(1 + \abs{X}),
$$
and
$$
\nabla^j \left(\psi\left(M \frac{V_0(x)}{(h\abs{X})^{2/3}}\right)\right) = {\cal O}(1)\frac{1}{h^{j/3}\abs{X}^{j/3}},\quad j =1,2.
$$

\medskip
\noindent
Still working in the region $\abs{X}\geq h^{1/2}$ and following~\cite{HeSjSt} closely, let us obtain a lower bound for the function $V_0 + \eps_0 H_{p_2} G_0$,
where $\eps_0 > 0$ is a constant to be chosen. We have, in view of (\ref{app7.1}),
$$
H_{p_2}G_0 = {\cal O}(h^{2/3} \abs{X}^{2/3}),
$$
and therefore, in the region where $M V_0\geq h^{2/3}\abs{X}^{2/3}$, we get
$$
V_0 + \eps_0 H_{p_2} G_0 \geq \left(\frac{1}{M} - {\cal O}(\eps_0)\right) h^{2/3} \abs{X}^{2/3} \geq \frac{1}{{\cal O}(1)M} h^{2/3} \abs{X}^{2/3},
$$
if we choose $\eps_0 > 0$ small enough. In the region where $M V_0 < h^{2/3}\abs{X}^{2/3}$, we have
$$
G_0 = h^{2/3} \frac{H_{p_2}V_0}{\abs{X}^{4/3}},
$$
and therefore
$$
V_0 + \eps_0 H_{p_2} G_0 = V_0 + \eps_0 h^{2/3} \abs{X}^{-4/3} H_{p_2}^2 V_0 + R,
$$
where
$$
R = \eps_0 h^{2/3} (H_{p_2}V_0) H_{p_2} \abs{X}^{-4/3} = {\cal O}\left(\frac{\eps_0 h}{M^{1/2}}\right) =
{\cal O}\left(\frac{\eps_0}{M^{1/2}}\right)h^{2/3}\abs{X}^{2/3}.
$$
Using (\ref{app6}), we see therefore that
$$
V_0 + \eps_0 H_{p_2} G_0 \geq \eps_0 \frac{h^{2/3}\abs{X}^{2/3}}{{\cal O}(1)} - {\cal O}\left(\frac{\eps_0}{M^{1/2}}\right)h^{2/3}\abs{X}^{2/3} \geq
\eps_0 \frac{h^{2/3}\abs{X}^{2/3}}{{\cal O}(1)},
$$
provided that we take $M$ sufficiently large but fixed. It follows that in the entire region where $\abs{X}\geq h^{1/2}$, we get
\begeq
\label{app7.3}
V_0 + \eps_0 H_{p_2}G_0 \geq \frac{h^{2/3}\abs{X}^{2/3}}{{\cal O}(1)}.
\endeq
We shall now extend the definition of $G_0$ to all of $\real^2$, and following~\cite{HeSjSt}, let us set
\begeq
\label{app7.4}
G(X;h) = \left(1-\chi\left(\frac{X}{h^{1/2}}\right)\right)G_0(X;h).
\endeq
Here $\chi\in C^{\infty}_0(\real^2;[0,1])$ is such that $\chi = 1$ when $\abs{X} \leq 1$. It follows from (\ref{app7.1}) and (\ref{app7.2}) that
$$
G = {\cal O}(h),\quad H_G = {\cal O}(h^{1/2}), \quad \partial^2 G = {\cal O}(1).
$$
Furthermore, using (\ref{app7.3}) we immediately check that on all of $\real^2$, we have
$$
V_0 + \eps_0 H_{p_2} G \geq \frac{h^{2/3}\abs{X}^{2/3}}{{\cal O}(1)} - {\cal O}(h).
$$
We may therefore summarize the discussion above by stating that there exist constants $\eps_0 > 0$ and $c_0 > 0$, such that we have for all $h>0$ sufficiently small,
\begeq
\label{app8}
V_0(x) + \eps_0 H_{p_2} G(X) + c_0 h \geq \frac{h^{2/3}\abs{X}^{2/3}}{{\cal O}(1)}, \quad X=(x,\xi)\in \real^{2}.
\endeq
Here the real-valued weight function $G = G(X,h)\in C^{\infty}(\real^2)$ has been defined in (\ref{app7}), (\ref{app7.4}).

\medskip
\noindent
It follows from (\ref{app2}), (\ref{app3.1}), (\ref{app3.2}) that $p_1(X) = V(x) \geq (1/C)V_0(x)$, for some constant $C>1$, and therefore using (\ref{app8}), we
get with $\eps_1 > 0$, $c_1 > 0$,
\begeq
\label{app11}
{\rm Re}\, p(X) + \eps_1 H_{{\rm Im}\, p} G(X) + c_1 h \geq \frac{1}{{\cal O}(1)} \left(h\abs{X}\right)^{2/3},\quad X\in \real^{2}.
\endeq

\medskip
\noindent
The estimate (\ref{app11}) is analogous to the estimate (4.26) of Section 4 of~\cite{HiPrSt}, if we take $k_0 = 1$ there. Taking (\ref{app11}) as the starting point and arguing
exactly as in that work, we find that everything works without any change, provided that the spectral parameter $z\in \comp$ is such that for some fixed $C_0>1$,
we have
\begeq
\label{app12}
{\rm Re}\, z \leq {\cal O}(1) h^{2/3}\abs{z}^{1/3}, \quad Ch \leq \abs{z}\leq C_0.
\endeq
Here $C\gg 1$ is a constant large enough and the implicit constant in (\ref{app12}) does not depend on $C$. We therefore obtain the a priori estimate
\begeq
\label{app13}
h^{2/3}\abs{z}^{1/3}\norm{u}_{L^2} \leq {\cal O}(1)\norm{(P-z)u}_{L^2},\quad u\in {\cal S}(\real^n),
\endeq
for $z$ satisfying (\ref{app12}).

\medskip
\noindent
It therefore remains to discuss the case when $z\in \comp$ is such that
\begeq
\label{app14}
{\rm Re}\, z \leq {\cal O}(1) h^{2/3}\abs{z}^{1/3},\quad \abs{z} \geq C_0.
\endeq
Continuing to follow the arguments of Section 4 in~\cite{HiPrSt}, we obtain from the equation (4.34) there that there exist positive constants $c_1$, $c_2$
such that for $z\in \comp$ satisfying (\ref{app14}), we have
\begeq
\label{app15}
\norm{(P-z)u}_{L^2}\, \norm{u}_{L^2} + c_1 h^{2/3}\abs{z}^{1/3}\left(\varphi\left(\frac{\abs{X}^2}{\abs{z}}\right)^w u,u\right) \geq c_2 h^{2/3}\abs{z}^{1/3}
\norm{u}_{L^2}^2.
\endeq
Here $\varphi\in C^{\infty}_0(\real^{2n};[0,1])$ is a cut-off near $0$ such that
\begeq
\label{app16}
\abs{p(X)-z} \geq \frac{\abs{z}}{2},
\endeq
on the support of $\varphi(\abs{X}^2/\abs{z})$. The spectral parameter $z$ in (\ref{app14}) can be arbitrarily large and when estimating the scalar product in
the left hand side of (\ref{app15}), we can apply Lemma 8.2 in~\cite{HeSjSt}, exactly as it stands, to conclude that
\begeq
\label{app17}
\left(\varphi\left(\frac{\abs{X}^2}{\abs{z}}\right)^w u,u\right) \leq \frac{{\cal O}(1)}{\abs{z}^2} \norm{(P-z)u}_{L^2}^2 + {\cal O}(1)h \norm{u}^2_{L^2}.
\endeq
Indeed, it is easily seen that the proof of Lemma 8.2 of~\cite{HeSjSt} applies in the present situation, using the ellipticity property (\ref{app16}) and the fact
that the symbol $p$ satisfies
\begeq
\label{app18}
\abs{p(X)}\leq {\cal O}(1)\abs{X}^2,\quad X\in \real^2, \quad \partial^{\alpha} p\in L^{\infty}(\real^2),\quad \abs{\alpha}\geq 2.
\endeq

\medskip
\noindent
Combining (\ref{app15}) and (\ref{app17}), we get
\begeq
\label{app19}
Z \norm{u}^2_{L^2} \leq {\cal O}(1)\frac{Z}{\abs{z}^2} \norm{(P-z)u}_{L^2}^2 + {\cal O}(1)\norm{(P-z)u}_{L^2}\, \norm{u}_{L^2}.
\endeq
Here we have written $Z=h^{2/3}\abs{z}^{1/3}$ for brevity. It follows that
\begeq
\label{app20}
Z \norm{u}^2 \leq {\cal O}(1)\left(\frac{1}{Z} + \frac{Z}{\abs{z}^2}\right)\norm{(P-z)u}_{L^2}^2,
\endeq
and using also that $Z \leq \abs{z}$, we get
\begeq
\label{app21}
h^{2/3}\abs{z}^{1/3} \norm{u}_{L^2} \leq {\cal O}(1)\norm{(P-z)u}_{L^2}.
\endeq

\bigskip
\noindent
We summarize the discussion above in the following result.
\begin{theo}
Let $P_0 = g(hD_x)(hD_x)^2 + iV(x)$ be such that {\rm (\ref{app1.001})}, {\rm (\ref{app1.01})}, {\rm (\ref{app1.1})}, {\rm (\ref{app2})}, {\rm (\ref{app3.1})},
{\rm (\ref{app3.2})} hold. There exist constants $c_0>0$ and $c_1>0$ such that for every $C>1$ large enough, we have for
\begeq
\label{app21.1}
{\rm Im}\, z \leq c_1 h^{2/3}\abs{z}^{1/3},\quad \abs{z}\geq Ch,
\endeq
the following estimate,
\begeq
\label{app22}
h^{2/3}\abs{z}^{1/3} \norm{u}_{L^2} \leq c_1 \norm{(P_0-z)u}_{L^2},\quad u\in H(m).
\endeq
\end{theo}

\medskip
\noindent
It follows that when $z$ satisfies (\ref{app21.1}) then $z$ is in the resolvent set of $P_0$ and we get the resolvent estimate
\begeq
\label{app23}
(P_0 -z)^{-1} = {\cal O}\left(h^{-2/3}\abs{z}^{-1/3}\right): L^2(\real) \rightarrow L^2(\real).
\endeq

\medskip
\noindent
{\it Remark}. The discussion above and the result of Theorem A.1 extends to the case of operators on $\real^n$.

\bigskip
\noindent
In what follows, the result of Theorem A.1 will only be applied in the case when ${\rm Im}\, z = {\cal O}(h)$. Furthermore, in the considerations in Section 7, we
are concerned with operators on the one-dimensional torus ${\bf T} = \real/2\pi \z$, and our next task is therefore to adapt Theorem A.1 to this setting. Let
us consider therefore
\begeq
\label{app23.1}
P = g(hD_x)(hD_x)^2 + i V(x),
\endeq
where $g\in C^{\infty}(\real;\real)$ satisfies (\ref{app1.001}), (\ref{app1.01}), and let $0\leq V \in C^{\infty}({\bf T})$ be such that $V^{-1}(0) = \{x_0\}$,
with $V''(x_0) > 0$. We may write
\begeq
\label{app23.15}
P = (hD_x)^2 + \varphi(hD_x) + i V(x),\quad \varphi \in C^{\infty}_0(\real;\real).
\endeq
Let $\chi \in C^{\infty}({\bf T}^n;[0,1])$ be supported in a small neighborhood of $x_0$, and such that $\chi = 1$ near $x_0$. On the support of $\chi$, the result
of Theorem A.1 can be applied and we conclude that
\begeq
\label{app23.2}
h^{2/3}\abs{z}^{1/3} \norm{\chi u}_{L^2} \leq {\cal O}(1) \norm{(P-z)\chi u}_{L^2} + {\cal O}(h^{\infty})\norm{u}_{L^2},\quad u\in C^{\infty}({\bf T}^n),
\endeq
provided that
\begeq
\label{app23.3}
\abs{z} \geq Ch, \quad {\rm Im}\, z \leq {\cal O}(h).
\endeq

\medskip
\noindent
Using that on the support of $1-\chi$, the potential $V$ is bounded from below and ${\rm Im}\,z = {\cal O}(h)$, we see that for all $h>0$ small enough, we have
\begeq
\label{app23.31}
{\rm Im}((P-z)(1-\chi)u,(1-\chi)u)_{L^2} \geq \frac{1}{{\cal O}(1)}\norm{(1-\chi)u}^2_{L^2},
\endeq
and therefore,
\begeq
\label{app23.4}
\norm{(1-\chi)u}_{L^2} \leq C \norm{(P-z)(1-\chi)u}_{L^2}.
\endeq

\medskip
\noindent
Combining the estimates (\ref{app23.2}) and (\ref{app23.4}), we see that for all $h>0$ small enough,
\begeq
\label{app23.5}
Z \norm{\chi u}_{L^2} + \norm{(1-\chi)u}_{L^2} \leq {\cal O}(1)\norm{(P-z)u}_{L^2} + {\cal O}(1)\norm{[P,\chi]u}_{L^2}.
\endeq
Here we continue to write $Z = h^{2/3}\abs{z}^{1/3}$ and we notice that $h \ll Z$. We would like to estimate the commutator term in the right hand side
of (\ref{app23.5}), and to that end we write, using (\ref{app23.15}),
\begeq
\label{app23.6}
{\cal O}(1)\norm{[P,\chi]u}_{L^2} \leq {\cal O}(h)\norm{u}_{L^2} + {\cal O}(h)\norm{h u'}_{L^2}.
\endeq
The first term in the right hand side of (\ref{app23.6}) can be absorbed into the left hand side of (\ref{app23.5}), and we only have to estimate
${\cal O}(h) \norm{hu'}_{L^2}$. Now
$$
(\varphi(hD_x)u,u)_{L^2} + \norm{hu'}_{L^2}^2 = {\rm Re}\, ((P-z)u,u)_{L^2} + {\rm Re}\, z \norm{u}_{L^2}^2,
$$
and therefore,
\begeq
\label{app23.7}
\norm{hu'}_{L^2} \leq \norm{(P-z)u}_{L^2}^{1/2}\norm{u}_{L^2}^{1/2} + \abs{z}^{1/2}\norm{u}_{L^2} + {\cal O}(1)\norm{u}_{L^2}.
\endeq
We obtain, combining (\ref{app23.5}), (\ref{app23.6}), and (\ref{app23.7}),
\begin{multline}
\label{app23.8}
Z \norm{\chi u}_{L^2} + \norm{(1-\chi)u}_{L^2} \\
\leq {\cal O}(1)\norm{(P-z)u} + {\cal O}(h) \abs{z}^{1/2}\norm{u}_{L^2} +
{\cal O}(h) \norm{(P-z)u}_{L^2}^{1/2}\norm{u}_{L^2}^{1/2},
\end{multline}
so that
\begeq
\label{app23.9}
Z \norm{\chi u}_{L^2} + \norm{(1-\chi)u}_{L^2} \leq {\cal O}(1) \norm{(P-z)u}_{L^2} + {\cal O}(Z^{3/2})\norm{u}_{L^2}.
\endeq
Assuming that
$$
Z = h^{2/3} \abs{z}^{1/3} \ll 1,
$$
we may absorb the second term in the right hand side of (\ref{app23.9}) into the left hand side, obtaining that
$$
h^{2/3}\abs{z}^{1/3} \norm{u}_{L^2} \leq {\cal O}(1) \norm{(P-z)u}_{L^2}.
$$

\medskip
\noindent
We may summarize the discussion above in the following proposition.
\begin{prop}
Let $P = g(hD_x)(h D_x)^2 + iV(x)$ be such that $g-1\in C^{\infty}_0(\real;\real)$ satisfies $g\geq 1$, $\abs{\xi g'(\xi)}\ll 1$. Assume furthermore that
$0\leq V \in C^{\infty}({\bf T})$, with $V^{-1}(0) = \{x_0\}$ and $V''(x_0)>0$. Then for every $C>1$ large enough,
we have for
$$
{\rm Im}\,z = {\cal O}(h), \quad \abs{z}\geq Ch,
$$
satisfying
\begeq
\label{app23.10}
h \abs{z}^{1/2} \ll 1,
\endeq
the following estimate
\begeq
\label{app23.11}
h^{2/3}\abs{z}^{1/3} \norm{u}_{L^2} \leq {\cal O}(1) \norm{(P-z)u}_{L^2}.
\endeq
\end{prop}

\bigskip
\noindent
The a priori estimate (\ref{app23.11}) is equivalent to the corresponding estimate for the resolvent of $P$, which provides a resolvent bound in the model case,
fundamental for the analysis in Section 7. Now the operators that one encounters there are somewhat more general than the Schr\"odinger type operator in
(\ref{app23.1}), in that the potential $V(x)$ should be replaced by a more general $h$-pseudodifferential operator, which furthermore is multiplied by a small
coupling constant. We shall now proceed to analyze this more general case, essentially by reducing it to the model situation treated above.

\bigskip
\noindent
Let us first consider the following operator on $\real$,
\begeq
\label{app24}
P_{\eps}(x,hD_x) =g(hD_x) (hD_x)^2 + i\eps \widetilde{V}^w(x,hD_x),
\endeq
where $g$ satisfies (\ref{app1.001}), (\ref{app1.01}), and following (\ref{eq6.0}), we assume that
\begeq
\label{eq24.1}
h^2 \ll \eps \ll h^{4/5}.
\endeq
The function $\widetilde{V} \in C^{\infty}(\real^{2})$ in (\ref{app24}) is of the form
\begeq
\label{app25}
\widetilde{V}(x,\xi) = V(x) + k(x,\xi) \varphi\left(\frac{\xi}{\eps^{\delta}}\right), \quad \delta \in (0,1/2),
\endeq
where $V$ is assumed to satisfy (\ref{app1.1}), (\ref{app2}), (\ref{app3.1}), and (\ref{app3.2}), and $\varphi \in C^{\infty}_0(\real^n)$ is a standard cutoff
function near $\xi = 0$. The function $k$ is such that
\begeq
\label{app25.1}
\partial^{\alpha} k\in L^{\infty}(\real^{2}),\quad \alpha\in {\nat}^{2},\quad k(x,0) = 0.
\endeq
We would like to extend Theorem A.1 to the operator $P_{\eps}(x,hD_x)$, and to that end we shall simply inspect the arguments above. Writing
\begeq
\label{app26}
\frac{1}{i \eps}P_{\eps}(x,hD_{x}) = -i g(\sqrt{\eps}\widetilde{h}D_x)(\widetilde{h} D_x)^2 + \widetilde{V}^w(x,\sqrt{\eps}\widetilde{h}D_{x}),
\quad \widetilde{h} = \frac{h}{\sqrt{\eps}} \ll 1,
\endeq
we shall view $(1/i \eps) P_{\eps}$ as an $\widetilde{h}$--pseudodifferential operator with the symbol
\begeq
\label{app27}
p(x,\xi) = p_1 + i p_2,
\endeq
where
\begeq
\label{app27.1}
p_1(x,\xi) = \widetilde{V}(x,\sqrt{\eps}\xi) , \quad p_2(x,\xi) = -g(\sqrt{\eps}\xi)\xi^2.
\endeq
Writing
$$
p_1(x,0) = V(x) = V_0(x) + {\cal O}(x^3),\quad V_0(x) = \frac{1}{2}V''(0)x^2,
$$
we see that uniformly in $\eps$, we have
\begeq
\label{app27.2}
V_0 + H_{p_2}^2 V_0 \sim \abs{X}^2,\quad X\in \real^2.
\endeq
Here we also notice that $\partial^{\alpha} p\in L^{\infty}(\real^{2})$ for all $\alpha \in \nat^{2}$ with $\abs{\alpha}\geq 2$, uniformly in $\eps$.
Arguing as in (\ref{app7}), (\ref{app8}), (\ref{app11}), we conclude that there exists a real-valued weight function $G\in C^{\infty}(\real^2)$ with
$$
G(X) = {\cal O}(\widetilde{h}),
$$
such that for some constants $\delta_1 > 0$ and $c_1 > 0$, we have for $0 < \widetilde{h}$ small enough,
\begeq
\label{app28}
{\rm Re}\, p(x,0) + \delta_1 H_{{\rm Im}\,p} G(X) + c_1 \widetilde{h} \geq \frac{1}{C} \widetilde{h}^{2/3}\abs{X}^{2/3},\quad X=(x,\xi) \in \real^{2}.
\endeq
Using (\ref{app25}) and (\ref{app28}), we obtain that
\begeq
\label{app29}
{\rm Re}\, p(X) + \delta_1 H_{{\rm Im}\,p} G(X) + c_1 \widetilde{h}
\geq \frac{1}{C} \widetilde{h}^{2/3}\abs{X}^{2/3} - 1_{\sqrt{\eps}\abs{\xi} \leq {\cal O}(\eps^{\delta})} {\cal O}(\sqrt{\eps}\abs{\xi}).
\endeq
Here we have used that
\begeq
\label{app29.1}
k(x,\xi) = {\cal O}(\abs{\xi}).
\endeq
When estimating the right hand side in (\ref{app29}), we notice that
\begin{multline}
\frac{1}{2C}\widetilde{h}^{2/3}\abs{X}^{2/3} - 1_{\sqrt{\eps}\abs{\xi} \leq {\cal O}(\eps^{\delta})} {\cal O}(\sqrt{\eps}\abs{\xi})
\geq \frac{1}{2C} \abs{\xi}^{2/3}\left(\widetilde{h}^{2/3} - 1_{\sqrt{\eps}\abs{\xi} \leq {\cal O}(\eps^{\delta})} {\cal O}(\sqrt{\eps}\abs{\xi}^{1/3})\right) \\
\geq \frac{1}{2C} \abs{\xi}^{2/3} \left(\widetilde{h}^{2/3} - {\cal O}(\sqrt{\eps} \eps^{\delta/3 - 1/6})\right) \geq 0,
\end{multline}
provided that
\begeq
\label{app30}
\eps^{\delta/3 + 1/2 - 1/6} \ll \widetilde{h}^{2/3} = \frac{h^{2/3}}{\eps^{1/3}}.
\endeq
The latter condition is equivalent to
\begeq
\label{app31}
\eps^{1 + \frac{\delta}{2}} \ll h.
\endeq
Assuming that (\ref{app31}) holds, we conclude that
\begeq
\label{app32}
{\rm Re}\, p(X) + \delta_1 H_{{\rm Im}\,p} G(X) + c_1 \widetilde{h} \geq \frac{1}{{\cal O}(1)} \widetilde{h}^{2/3}\abs{X}^{2/3}, \quad X\in \real^{2}.
\endeq
In order to be able to apply the general arguments of~\cite{HiPrSt} and~\cite{HeSjSt} to the operator $P_{\eps}/i\eps$, similarly to the discussion leading to
Theorem A.1 above, we should also observe that for $\abs{z} \gg \widetilde{h}$, we have in the region where
$$
\abs{X}^2 \leq \frac{\abs{z}}{C},
$$
that the ellipticity condition
$$
\abs{p(X) - z} \geq \frac{\abs{z}}{C_1},
$$
holds, uniformly in $\eps$. Indeed, this follows from (\ref{app25}), (\ref{app27}), (\ref{app27.1}), (\ref{app29.1}), as well as the fact that $\eps \ll \widetilde{h} \ll \abs{z}$.
It is then straightforward to check that the arguments in the beginning of the appendix apply and we obtain the following result.

\begin{theo}
Let $P_{\eps} = g(hD_x)(hD_x)^2 + i\eps \widetilde{V}^w(x,h D_x)$, where $h^2 \ll \eps \ll h^{4/5}$. Assume that $g\in C^{\infty}(\real;\real)$ satisfies
{\rm (\ref{app1.001})}, {\rm (\ref{app1.01})},
$\widetilde{V}$ is of the form
{\rm (\ref{app25})}, {\rm (\ref{app25.1})}, and that {\rm (\ref{app1.1})}, {\rm (\ref{app2})}, {\rm (\ref{app3.1})}, {\rm (\ref{app3.2})} hold. Assume that
$$
\eps^{1 + \frac{\delta}{2}} \ll h.
$$
Then we have, writing $\widetilde{h} = h/\sqrt{\eps}$,
\begeq
\label{app33}
\left(\frac{P_{\eps}}{\eps} - z\right)^{-1} = {\cal O}\left(\widetilde{h}^{-2/3}\abs{z}^{-1/3}\right): L^2(\real) \rightarrow L^2(\real),
\endeq
provided that $\abs{z} \geq C \widetilde{h}$ for $C>1$ sufficiently large, and
$$
{\rm Im}\,z \leq {\cal O}(\widetilde{h}).
$$
\end{theo}

\medskip
\noindent
Repeating the arguments leading to Proposition A.2, with Theorem A.1 replaced by Theorem A.3 and with an estimate of the form (\ref{app23.31}) obtained by an
application of G\aa{}rding's inequality, we next obtain an adaptation of Theorem A.3 to the setting of the torus.
\begin{prop}
Let $P_{\eps} = g(hD_x)(h D_x)^2 + i\eps \widetilde{V}^w(x,hD_x)$, where $h^2 \ll \eps \ll h^{4/5}$, and $g\in C^{\infty}(\real;\real)$ is such that
$g-1\in C^{\infty}_0(\real)$ with
$$
g\geq 1,\quad \abs{\xi g'(\xi)} \ll 1.
$$
Assume that $\widetilde{V}$ is of the form
$$
\widetilde{V}(x,\xi) = V(x) + k(x,\xi) \varphi\left(\frac{\xi}{\eps^{\delta}}\right),\quad \delta\in (0,1/2).
$$
Here $0\leq V\in C^{\infty}({\bf T})$, $V^{-1}(0) = \{x_0\}$, $V''(x_0) > 0$, $k\in S(T^*{\bf T}, 1)$, $k(x,0) = 0$, and $\varphi \in C^{\infty}_0(\real)$.
Assume that $\eps^{1 + \frac{\delta}{2}} \ll h$ and let us set
$$
\widetilde{h} = \frac{h}{\sqrt{\eps}} \ll 1.
$$
Let $z\in \comp$ be such that ${\rm Im}\, z = {\cal O}(\widetilde{h})$, $\abs{z}\geq C\widetilde{h}$, for $C>1$ sufficiently large, and
$$
\widetilde{h}\abs{z}^{1/2} \ll 1.
$$
Then we have
\begeq
\label{app34}
\left(\frac{P_{\eps}}{\eps} - z\right)^{-1} = {\cal O}\left(\widetilde{h}^{-2/3}\abs{z}^{-1/3}\right): L^2({\bf T}) \rightarrow L^2({\bf T}).
\endeq
\end{prop}

\end{appendix}


\begin{thebibliography}{30}





\bibitem{CdV} Y.~Colin de Verdi\`ere, {\it La m\'ethode de moyennisation en m\'ecanique semi-classique}, Journ\'ees "\'Equations aux D\'eriv\'ees Partielles'',
Saint-Jean-de-Monts, 1996, Exp. No. V, 11 pp., \'Ecole Polytech., Palaiseau, 1996.

\bibitem{DeSjZw} N. Dencker, J. Sj\"ostrand, and M. Zworski, {\it Pseudo-\-spectra of se\-mi\-clas\-si\-cal
(pseudo)\-diffe\-ren\-tial ope\-rators}, Comm. Pure Appl. Math. {\bf 57} (2004), 384-415.


\bibitem{GaSh} S. Galtsev and A. Shafarevich, {\it Quantized Riemann surfaces and semiclassical spectral series for a nonselfadjoint Schr\"odinger
operator with periodic coefficients}, (Russian) Teoret. Mat. Fiz. {\bf 148} (2006), 206--226; translation in Theoret. and Math. Phys. {\bf 148} (2006), 1049–-1066.







\bibitem{HeSjSt} F. H\'erau, J. Sj\"ostrand, and C. Stolk, {\it Semiclassical analysis for the Kramers--Fokker--Planck equation},
Comm. Partial Differential Equations {\bf 30} (2005), 689--760.



\bibitem{Hi04} M. Hitrik, {\it Boundary spectral behavior for semiclassical operators in dimension one}, Int. Math. Res. Not. {\bf 64} (2004), 3417--3438.


\bibitem{HiPrSt} M. Hitrik and K. Pravda-Starov, {\it Eigenvalues and subelliptic estimates for non-selfadjoint semiclassical operators with double characteristics},
Ann. Inst. Fourier {\bf 63} (2013), 985--1032.

\bibitem{HiSj1} M. Hitrik and J. Sj\"ostrand, {\it Non-selfadjoint perturbations of selfadjoint operators in {\rm 2} dimensions {\rm I}},
Ann. Henri Poincar\'e {\bf 5} (2004), 1--73.

\bibitem{HiSj2} M. Hitrik and J. Sj\"ostrand, {\it Non-selfadjoint perturbations of selfadjoint operators in {\rm 2}
dimensions {\rm II}. Vanishing averages}, Comm. Partial Differential Equations {\bf 30} (2005), 1065--1106.

\bibitem{HiSj3a} M. Hitrik and J. Sj\"ostrand, {\it Non-selfadjoint perturbations of selfadjoint operators in
{\rm 2} dimensions {\rm III a}. One branching point}, Canadian J. Math. {\bf 60} (2008), 572–-657.

\bibitem{HiSj08b} M.~Hitrik and J.~Sj\"ostrand, {\it Rational invariant tori, phase
space tunneling, and spectra for non-selfadjoint operators in dimension {\rm 2}}, Annales Sci ENS, s\'er. {\bf 41} (2008), 511--571.

\bibitem{HiSj12} M. Hitrik and J. Sj\"ostrand, {\it Diophantine tori and Weyl laws for non-selfadjoint operators in dimension two},
Comm. Math. Phys. {\bf 314} (2012), 373-–417.

\bibitem{HiSjVu07} M. Hitrik, J. Sj\"ostrand, and S. V\~u Ng\d{o}c, {\it Diophantine tori and spectral asymptotics for non-selfadjoint operators},
Amer. J. Math. {\bf 129} (2007), 105--182.


\bibitem{Kato} T. Kato, {\it Perturbation theory for linear operators}, Grundlehren der Mathematischen Wissenschaften, Springer Verlag, Berlin--New York, 1976.

\bibitem{Le} G. Lebeau, {\it \'Equation des ondes amorties}, Algebraic and geometric methods in mathematical physics
(Kaciveli, 1993), 73--109, Math. Phys. Stud., {\bf 19}, Kluwer
Acad. Publ., Dordrecht, 1996.

\bibitem{LeLe} M. L\'eautaud and N. Lerner, {\it Energy decay for a locally undamped wave equation}, preprint, 2014, {\sf http://arxiv.org/abs/1411.7271}.

\bibitem{LcLb} A. J. Lichtenberg and M. A. Lieberman, {\it Regular and chaotic dynamics}, Second edition, Springer-Verlag, New York, 1992.

\bibitem{Ma} A. S. Markus, {\it Introduction to the spectral theory of polynomial operator pencils}, Translations
of Mathematical Monographs 71, American Mathematical Society, Providence RI, 1998.

\bibitem{MaMa} A. S. Markus and V. I. Matsaev, {\it Comparison theorems for spectra of linear operators and spectral
asymptotics} (Russian), Trudy Moskov. Mat. Obsch. {\bf 45} (1982), 133--181.

\bibitem{MeSj1} A. Melin and J. Sj\"ostrand,
{\it Determinats of pseudodifferential operators and complex deformations of phase space}, Methods and Appl. of Analysis {\bf 9} (2002), 177--238.

\bibitem{MeSj2} A. Melin and J. Sj\"ostrand, {\it Bohr-Sommerfeld quantization condition for non-selfadjoint operators
in dimension {\rm 2}}, Ast\'erisque {\bf 284} (2003), 181--244.

\bibitem{LN} L. Nedelec, {\it Perturbations of non self-adjoint Sturm-Liouville problems, with applications to harmonic oscillators},
Methods Appl. Anal. {\bf 13} (2006), 123–-148.

\bibitem{PrSt1} K. Pravda-Starov, {\it A complete study of the pseudo-spectrum for the rotated harmonic oscillator},
J. London Math. Soc. {\bf 73} (2006), 745–-761.


\bibitem{R} P. Redparth, {\it Spectral properties of non-self-adjoint operators in the semi-classical regime},
J. Diff. Equations {\bf 177} (2001), 307–-330.

\bibitem{Sj82} J. Sj\"ostrand, {\it Singularit\'es analytiques microlocales}, Ast\'erisque, 1982.


\bibitem{Sj92} J. Sj\"ostrand, {\it Semi-excited states in non-degenerate potential wells}, Asymptot. Analysis {\bf 6} (1992), 29--43.

\bibitem{Sj96} J.~Sj\"ostrand, {\it Density of resonances for strictly convex analytic obstacles}, Canadian J. Math. {\bf 48} (1996), 397--447.


\bibitem{Sj00} J. Sj\"ostrand, {\it Asymptotic distribution of eigenfrequencies for damped wave equations}, Publ. Res.
Inst. Math. Sci. {\bf 36} (2000), 573--611.



\bibitem{SjZw1} J. Sj\"ostrand and M. Zworski, {\it Asymptotic distribution of resonances for convex obstacles}, Acta
Math. {\bf 183} (1999), 191--253.


\bibitem{Viola12} J. Viola, {\it Resolvent estimates for non-selfadjoint operators with double characteristics}, J. Lond. Math. Soc. {\bf 85} (2012), 41-–78.

\bibitem{San} S. V\~u Ng\d{o}c, {\it Syst\`emes int\'egrables
semi-classiques: du local au global}, Panoramas et Synth\`eses 22, Soci\'et\'e Math\'ematique de France, Paris, 2006.

\bibitem{We} A. Weinstein, {\it Asymptotics of eigenvalue clusters for the Laplacian plus a potential}, Duke Math. J. {\bf 44} (1977), 883--892.

\end{thebibliography}
\end{document}